%% file: main.tex
\documentclass[review,hidelinks,onefignum,onetabnum]{siamart220329}
\newtheorem{example}{Example}[section]
\input{ex_shared}
\ifpdf
\hypersetup{
  pdftitle={Structure preserving quaternion conjugate gradient-type methods for solving non-Hermitian quaternion linear systems},
  pdfauthor={Baohua Huang}
}
\fi

\begin{document}

\maketitle

% REQUIRED
\begin{abstract}
In this paper, we consider the non-Hermitian quaternion linear systems arising from color image restoration and three-dimensional signal filtering problems. For exploring to solve such systems, we present two innovative structure-preserving conjugate gradient-type methods, QNHERLQ and QNHERQR, which are based on the unitary equivalence transformations of the non-Hermitian quaternion matrices to tridiagonal forms, called quaternion Saunders-Simon-Yip tridiagonalization procedure. The proposed tridiagonalization procedure for non-Hermitian quaternion matrices is closely related to the quaternion Lanczos process for Hermitian matrices,  and is very different from the quaternion Lanczos biorthogonalization process for non-Hermitian matrices. The convergence of QNHERLQ and QNHERQR is discussed, which depends on the singular values of the coefficient matrix. Also we show that both algorithms have the finite termination property and constant costs per iteration step. Numerical results illustrate that the proposed algorithms are with the robustness and effectiveness compared with QGMRES and QQMR.
\end{abstract}

%does not require preprocessing and has a wider range of applicability for solving large-scale linear systems.

% REQUIRED

% REQUIRED
\section{Introduction}\label{sec1}
The set of quaternions is an associative but not commutative algebra over $\mathbb{R}$, given by
\begin{equation*}
  \mathbb{Q}=\mathrm{span}\{1,\mathbf{i},\mathbf{j},\mathbf{k}\}=\{q_0+q_1\mathbf{i}+q_2\mathbf{j}+q_3\mathbf{k}:
  q_0,q_1,q_2,q_3\in\mathbb{R}\},
\end{equation*}
where $\{1,\mathbf{i},\mathbf{j},\mathbf{k}\}$ is a basis of $\mathbb{Q}$, and the imaginary units $\mathbf{i}$, $\mathbf{j}$ and $\mathbf{k}$ satisfy $\mathbf{i}^2=\mathbf{j}^2=\mathbf{k}^2=\mathbf{i}\mathbf{j}\mathbf{k}=-1$  (see \cite{h1844,h1866}).
Denote   the set of all $m\times n$ quaternion matrices by $\mathbb{Q}^{m\times n}$.
%%%%%%%%%%%%%%%%%%%%%%%%%
%is  of the form
% $\mathbf{q} =q_0+q_1\mathbf{i}+q_2\mathbf{j}+q_3\mathbf{k}$
%with $q_0, q_1, q_2, q_3\in\mathbb{R}$   and  three imaginary units $\mathbf{i}$, $\mathbf{j}$, $\mathbf{k}$ satisfying  $\mathbf{i}^2=\mathbf{j}^2=\mathbf{k}^2=\mathbf{i}\mathbf{j}\mathbf{k}=-1$. Denote $\mathbb{Q}$ as the quaternion-skew  field and $\mathbb{Q}^{m\times n}$ as the $m\times n$ matrices on $\mathbb{Q}$.
%%%%%%%%%%%%%%%%%%%%%%%%%
In this paper, we consider the quaternion linear systems of the form
\begin{equation}\label{1.1}
  \mathscr{A}\mathbf{x}=\mathbf{b},
\end{equation}
where $\mathscr{A}\in \mathbb{Q}^{n\times n}$ is an invertible non-Hermitian quaternion matrix,  $\mathbf{b}\in \mathbb{Q}^{n}$ is a known
quaternion vector, and $\mathbf{x}\in\mathbb{Q}^{n}$ is an unknown quaternion vector to be determined. Such quaternion linear systems have a wide range of applications in many problems in science and engineering, including quantum physics \cite{a1994},  computer graphics \cite{p1989,gcc2018}, quaternion convolutional neural network \cite{pml2020,sgr2017},   signal processing \cite{bm2004}, color image processing \cite{sv2011,jns2019,jnw2019,cjpp2021,sdn2021}, and so on.

\subsection{Existing methods}
Existing iterative methods for solving the quaternion linear systems include the complex method, the real method and the structure-preserving method.
The complex or real method is built upon the complex  or real
counterpart such that the dimension is two or four times the original.

\begin{itemize}
\item The complex method.  By a linear homeomorphic mapping from quaternion matrices/vectors to their complex counterparts $\mathcal{C}(\cdot )$, the quaternion linear systems (\ref{1.1}) can be rewritten equivalently  as the following complex
matrix equation
\begin{equation}\label{1.2}
\mathcal{C}(\mathscr{A}) \mathcal{C}(\mathbf{x}) = \mathcal{C}(\mathbf{b}),
\end{equation}
which can be solved by iterative methods, e.g., nested splitting conjugate gradient method \cite{zt2013}, HSS splitting method \cite{zwz2016},  Jacobi and Gauss-Seidel-type splitting method \cite{ttlx2017} and modified CG method \cite{hm2016}, etc.
\item The real method. The quaternion linear systems (\ref{1.1}) can be transformed into the real matrix equation
\begin{equation}\label{1.3}
\mathcal{R}(\mathscr{A}) \mathcal{R}(\mathbf{x}) = \mathcal{R}(\mathbf{b})
\end{equation}
by the linear homeomorphic mapping from quaternion matrices/vectors to their real counterparts $\mathcal{R}(\cdot )$. To solve the general real linear systems, Krylov subspace methods, such as GMRES \cite{s1981}, CGS \cite{s1989}, QMR \cite{fn1991} and conjugate gradient-type  methods \cite{l1952,ssz2009,ssy1988}, are feasible.
\end{itemize}

%\item The structure-preserving method.

It is noted that the computational work and storage requirements of the real and complex   methods increase rapidly with the increasing of the dimension. This may make the iterative methods for solving (\ref{1.2}) and (\ref{1.3})  less efficient.

\begin{itemize}
\item The structure-preserving method. Compared to the traditional iterative met-\\hods for solving the real/complex representations of the quaternion linear systems, the structure-preserving methods inheriting the algebraic symmetry of real representations do not need to expand the dimension.  For example, Jia and Ng \cite{jn2021} developed a structure-preserving quaternion generalized minimal residual (QGMRES) method for solving the quaternion linear systems (\ref{1.1}).  Li and Wang \cite{lw2023} proposed the structure-preserving quaternion full orthogonalization (QFOM) method and its preconditioned variant for approximating the solutions of  (\ref{1.1}).   Both QGMRES and QFOM are based on the quaternion Arnoldi procedure and preserve the upper Hessenberg structure in quaternion representation. Besides, the quaternion biconjugate gradient (QBiCG) method \cite{lw2024}, which depends on the quaternion Lanczos biorthogonalization process and preserves the quaternion tridiagonal form within the iterations, was proposed. Based on three-term recurrences, the QBiCG method requires less memory and has better performance compared with QGMRES and QFOM methods. In \cite{lwz2024}, the quaternion-coupled two-term biconjugate orthonormalization procedure was established and two different structure-preserving quaternion quasi-minimal residual (QQMR) methods were derived for the quaternion linear systems (\ref{1.1}).
\end{itemize}

\subsection{Motivation}
By making full use of the JRS-symmetry of the real counterpart, structure-preserving methods save three-quarters of the arithmetic operations in comparison with solving the real equivalent linear equation (\ref{1.3}). It is very interesting to further study efficient numerical algorithms for the quaternion linear systems under the structure-preserving framework.

It is well known that the conjugate gradient (CG) method has three properties: the finite termination property, the minimization property,  and the fact that the method is based on a three-term recurrence. These properties imply that the CG method terminates after a finite number of iterations in the absence of round-off errors, that some measure of the error decreases during the iteration, and that the amount of computation for each step is constant. When the coefficient matrix $\mathscr{A}$ is Hermitian and positive definite, CG method \cite{o2005} performs well. This inspires us to explore an iterative method for the general non-Hermitian  quaternion linear systems that enjoy all three nice properties of the CG algorithm. To maintain three recurrences, the most important thing is to establish the tridiagonalization procedure of the non-Hermitian quaternion matrices.

\subsection{Contributions}
The contributions of this paper are give below:
\begin{itemize}
\item Borrowing the idea of Saunders-Simon-Yip tridiagonalization for asymmetric matrices \cite{ssy1988}, we construct a tridiagonalization process for non-Hermitian quaternion matrices, called quaternion Saunders-Simon-Yip tridiagonalization, which is distinctively different from the quaternion Lanczos biorthogonalization for non-Hermitian matrices \cite{lw2024}.

\item  By using the new quaternion Saunders-Simon-Yip tridiagonalization process, we propose two structure-preserving conjugate gradient-type (QCG-type) methods for solving the non-Hermitian quaternion linear systems, one is based on the Galerkin condition, another is on the minimum residual condition. For the sake of convenience, we say these two methods to be QNHERLQ and  QNHERQR, respectively. The QNHERLQ and  QNHERQR methods depend on the stability of quaternion operations and the rapidity of real computations only performs real operations without dimension expanding.  The resulting tridiagonal quaternion linear system   (least-squares problem) in QNHERLQ (QNHERQR) is solved by the quaternion LQ (QR) factorization with the quaternion Givens transformations.

\item Following the ideas for the derivation of the algorithms SYMMLQ and MINRES for symmetric indefinite matrices \cite{ps1975}, we present the practical implementations of QNHERLQ and  QNHERQR.

\end{itemize}

Different from  existing methods in \cite{jn2021,lw2023,lw2024,lwz2024}, the proposed algorithm has the  following advantages:
\begin{itemize}
\item The proposed  quaternion Saunders-Simon-Yip tridiagonalization process can be generated by constructing the following transformation of $\mathscr{A}$:
 \begin{equation}\label{1.4}
\mathscr{P}^*\mathscr{A}\mathscr{Q}=\mathscr{T},
\end{equation}
where $\mathscr{P}$ and $\mathscr{Q}$ are unitary matrices and $\mathscr{T}$ is strictly tridiagonal quaternion matrix whose   off-diagonal elements are nonnegative real numbers.
The quaternion Saunders-Simon-Yip tridiagonalization process
is distinctively different from the quaternion Lanczos biorthogonalization process \cite{lw2024}, which is based on the similarity transformation
\begin{equation}\label{1.5}
\mathscr{V}^{-1}\mathscr{A}\mathscr{V}=\mathscr{S},
\end{equation}
$\mathscr{S}$ is a tridiagonal quaternion matrix whose sub-diagonal   elements are nonnegative real numbers. Whereas (\ref{1.5}) preserves the eigenvalues of $\mathscr{A}$, (\ref{1.4}) preserves the singular values of $\mathscr{A}$, which is very suitable for solving the quaternion linear systems.

\item The proposed quaternion Saunders-Simon-Yip tridiagonalization process is  quite different from quaternion coupled three-term biconjugate orthonormalization procedure \cite{lwz2024}, which is also based on the similarity transformation of $\mathscr{A}$:
\begin{equation}\label{1.6}
\mathscr{W}^*\mathscr{A}\mathscr{V}=\mathscr{D}\mathscr{H},
\end{equation}
$\mathscr{D}$ is an invertible diagonal quaternion matrix, $\mathscr{H}$ is a tridiagonal quaternion matrix whose  sub-diagonal elements are nonnegative real numbers, and $\mathscr{W}^*\mathscr{V}=\mathscr{D}$.

\item The resulting tridiagonal quaternion linear system   or least-squares problem in our proposed QNHERLQ and  QNHERQR methods is directly solved by the quaternion LQ or QR factorization with the quaternion Givens transformations. In contrast, the upper Hessenberg quaternion least-squares problem resulted in QGMRES method \cite{jn2021} is first transformed into an upper tridiagonal least-squares problem, and then solved by a back substitution method on the quaternion skew-field. The upper Hessenberg quaternion linear system resulted in the QFOM method \cite{lw2023} is reduced to an upper tridiagonal quaternion linear system, and then similarly solved by a back substitution method.

\item Compared with the classical conjugate gradient-type method, the QNHERLQ and  QNHERQR methods significantly save  storage requirements and computations by using the structure-preserving strategy.
\end{itemize}

\subsection{Organization}
The rest of this paper is organized as follows. In Section 2, we start with some definitions and notations and review some results about quaternion matrices. In Section 3, we present the structure-preserving quaternion Saunders-Simon-Yip tridiagonalization procedure based on the three-term recurrences, which preserves the quaternion strictly tridiagonal form during the iterations. In Section 4,  based on the Galerkin condition  and   the minimum residual condition, we propose the corresponding structure-preserving conjugate gradient-type (QCG-type) methods for solving the non-Hermitian quaternion linear systems. The convergence analysis is provided and selection of the initial vector is discussed. In Section 5, we give some theoretical results which indicate that it is always preferable to solve the original quaternion linear systems rather than equivalent real ones. In Section 6, some numerical examples are given to show the robustness and effectiveness of the proposed methods. Some concluding remarks are given in Section 7.

\section{Preliminaries}\label{sec10}

 In this section, we introduce some notations and recall basic information of quaternion algebra.
For a quaternion $\mathbf{q}=q_0+q_1\mathbf{i}+q_2\mathbf{j}+q_3\mathbf{k}\in\mathbb{Q}$, its conjugate    is defined as $\mathbf{q}^*=q_0-q_1\mathbf{i}-q_2\mathbf{j}-q_3\mathbf{k}$,
while the conjugate of the product satisfies $(\mathbf{p}\mathbf{q})^*=\mathbf{q}^*\mathbf{p}^*$.
The modulus of $\mathbf{q}=q_0+q_1\mathbf{i}+q_2\mathbf{j}+q_3\mathbf{k}\in\mathbb{Q}$ is defined
as $|\mathbf{q}|=\sqrt{\mathbf{q}\mathbf{q}^*}=\sqrt{q_0^2+q_1^2+q_2^2+q_3^2}$ and it holds
$|\mathbf{p}\mathbf{q}|=|\mathbf{p}||\mathbf{q}|$.
Each nonzero quaternion is invertible and the unique inverse is specified by $\mathbf{q}^{-1}=\mathbf{q}^*/|\mathbf{q}|^2$. For $\mathscr{X}=(\mathbf{x}_{ij}) \in\mathbb{Q}^{m\times n}$, let $\overline{\mathscr{X}}=(\mathbf{x}_{ij}^*)\in\mathbb{Q}^{m\times n}$ and $\mathscr{X}^*=(\mathbf{x}_{ji}^*)\in\mathbb{Q}^{n\times m}$ be the conjugate and conjugate transpose of $\mathscr{X}$, respectively. The unit quaternion matrix $\mathbf{I}$ is just as the classical unit matrix. For a square quaternion matrix $\mathscr{X}\in\mathbb{Q}^{n\times n}$, we say $\mathscr{X}$ is   unitary and Hermitian,   respectively, if $\mathscr{X}^*\mathscr{X} = \mathscr{X}\mathscr{X}^* = \mathbf{I}$ and $\mathscr{X}^*=\mathscr{X}$. Let $\mathscr{A}\in\mathbb{Q}^{n\times n}$ be Hermitian. We call $\mathscr{A}$ is positive definite if $\mathbf{x}^*\mathscr{A}\mathbf{x}>0$ for any $\mathbf{0}\neq\mathbf{x}\in\mathbb{Q}^{n}$. For a Hermitian and positive definite matrix $\mathscr{A}\in\mathbb{Q}^{n\times n}$, there exist  an unitary quaternion matrix $\mathscr{U}\in\mathbb{Q}^{n\times n}$ and a real diagonal matrix $\Sigma\in\mathbb{R}^{n\times n}$ with positive diagonal elements  such that
$\mathscr{A}=\mathscr{U}\Sigma\mathscr{U}^*$. Thus we define $\sqrt{\mathscr{A}}=\mathscr{U}\sqrt{\Sigma}\mathscr{U}^*$ for the Hermitian and positive definite matrix $\mathscr{A}$.
%$\mathscr{X}$ is called Hermitian if $\mathscr{X}^*=\mathscr{X}$.

The   inner product of two quaternion vectors $\mathbf{x}=(\mathbf{x}_i)\in\mathbb{Q}^n$ and $\mathbf{y}=(\mathbf{y}_i)\in\mathbb{Q}^n$  is defined as
\begin{equation*}
\langle\mathbf{x},\mathbf{y}\rangle=\sum\limits_{i=1}^n\mathbf{y}_i^*\mathbf{x}_i.
\end{equation*}
The norm of   $\mathbf{x}\in\mathbb{Q}^n$ generated by this inner product is of the form
\begin{equation*} \|\mathbf{x}\|_2=\sqrt{\sum\limits_{i=1}^n|\mathbf{x}_i|^2}.
\end{equation*}
From  \cite{gmp2013} it follows that $\mathbb{Q}^n$ is a right quaternionic Hilbert space with the above inner product and satisfies the following four properties:
\begin{itemize}
\item (Right linearity) $\langle \mathbf{x}\bm{\alpha}_1+\mathbf{y}\bm{\alpha}_2,\mathbf{z}\rangle =\langle\mathbf{x},\mathbf{z}\rangle\bm{\alpha}_1+\langle\mathbf{y},\mathbf{z}\rangle\bm{\alpha}_2$ for  $\mathbf{x},\mathbf{y},\mathbf{z}\in\mathbb{Q}^n$ and $\bm{\alpha}_1,\bm{\alpha}_2\in\mathbb{Q}$.
\item (Quaternionic hermiticity) $\langle \mathbf{x},\mathbf{y}\rangle=\langle \mathbf{y},\mathbf{x}\rangle^*$ for $\mathbf{x},\mathbf{y}\in\mathbb{Q}^n$.
\item (Positivity) $\langle \mathbf{x},\mathbf{x}\rangle\geq0$ for $\mathbf{x}\in\mathbb{Q}^n$ and $\mathbf{x}=\mathbf{0}$ if and only if $\langle \mathbf{x},\mathbf{x}\rangle=0$.
\item If $\mathbf{x},\mathbf{y}\in\mathbb{Q}^n$, then their distance is defined by $d(\mathbf{x},\mathbf{y})=\sqrt{\langle \mathbf{x}-\mathbf{y},\mathbf{x}-\mathbf{y}\rangle}$.
\end{itemize}

%%%%%%%%%%%%%%%%%%%%%%
For a quaternion matrix $\bm{\mathscr{M}}=M_0+M_1\mathbf{i}+M_2\mathbf{j}+M_3\mathbf{k}\in \mathbb{Q}^{m\times n}$, its real representation is given by
\begin{equation}\label{jrs}
  \mathcal{R}(\bm{\mathscr{M}})=\left[
  \begin{array}{cccc}
 M_0&M_2&M_1&M_3\\
  -M_2&M_0&M_3&-M_1\\
  -M_1&-M_3&M_0&M_2\\
  -M_3&M_1&-M_2&M_0
  \end{array}
  \right]\in\mathbb{R}^{4m\times 4n}.
\end{equation}
We denote its first block column by $\mathcal{R}(\bm{\mathscr{M}})_c=[M_0^T,-M_2^T,-M_1^T,-M_3^T]^T\in\mathbb{R}^{4m\times n}$.

Next, we give several important properties of  real counterparts of quaternion matrices.

\begin{definition} [\cite{jwzc2018}]\label{defn2.1}
Let
 \begin{equation*}
  J_m=\left[\begin{smallmatrix}
          0 & 0 & -I_m & 0 \\
          0 & 0 & 0 & -I_m \\
          I_m & 0 & 0 & 0 \\
          0 & I_m & 0 & 0 \\
        \end{smallmatrix}\right],
   R_m=\left[
       \begin{smallmatrix}
          0 & -I_m & 0 & 0 \\
          I_m & 0 & 0 & 0 \\
          0 & 0 & 0 & I_m \\
          0 & 0 & -I_m & 0 \\
       \end{smallmatrix}
      \right],
    S_m=\left[
       \begin{smallmatrix}
          0 & 0 & 0 & -I_m \\
          0 & 0 & I_m & 0 \\
          0 & -I_m & 0 & 0 \\
          I_m & 0 & 0 & 0 \\
       \end{smallmatrix}
      \right].
 \end{equation*}
$(1)$ A   matrix    $W\in\mathbb{R}^{4m\times 4n}$ is said to be JRS-symmetric if $J_mWJ_n^T=W$, $R_mWR_n^T=W$ and $S_mWS_n^T=W$.\\
$(2)$ If $m\leq n$, a matrix $Q\in\mathbb{R}^{4m\times 4n}$ is said to be JRS-symplectic if $QJ_n Q^T=J_m$, $QR_n Q^T=R_m$ and $QS_n Q^T=S_m$.\\
$(3)$  A   matrix    $W\in\mathbb{R}^{4n\times 4n}$ is said to be orthogonally JRS-symplectic if it is orthogonal and JRS-symplectic.
\end{definition}
%%%%%%%%%%%%%%%%%%%%%%%%%%
Obviously,   $W\in\mathbb{R}^{4m\times 4n}$ is JRS-symmetric if and only if $W$ is a real representation of a quaternion matrix, and the set of all JRS-symmetric matrices is closed under multiplication and addition.  Consequently, we define an inverse mapping   $\mathcal{R}^{-1}$ from the JRS-symmetric matrix to a quaternion matrix by $\mathcal{R}^{-1}( \mathcal{R}(\mathscr{M})) = \mathscr{M}$. A quaternion matrix $\mathscr{Q}\in\mathbb{Q}^{n\times n}$ is unitary if and only if $\mathcal{R}(\mathscr{Q})\in\mathbb{R}^{4n\times 4n}$ is orthogonally JRS-symplectic, see Lemma 2.1 and Theorem 2.1 in \cite{jwzc2018}.    It is noted that there are many different real representations, but they are permutation equivalent; see \cite[Remark 4.7]{jwzc2018}.

%%%%%%%%%%%%%%%
The $4n\times 4n$ generalized  symplectic Givens rotation is specified by
\begin{equation}\label{givens}
  G^{(l)}=
  \left[
  \begin{array}{cccc}
 G_0^{(l)}&G_2^{(l)}&G_1^{(l)}&G_3^{(l)}\\
  -G_2^{(l)}&G_0^{(l)}&G_3^{(l)}&-G_1^{(l)}\\
  -G_1^{(l)}&-G_3^{(l)}&G_0^{(l)}&G_2^{(l)}\\
  -G_3^{(l)}&G_1^{(l)}&-G_2^{(l)}&G_0^{(l)}
  \end{array}
  \right],
\end{equation}
where
\begin{eqnarray*}
% \nonumber to remove numbering (before each equation)
   && G_0^{(l)}=\left[
   \begin{array}{ccc}
  I_{l-1}&0&0\\
  0&\cos\alpha_0&0\\
     0&0&I_{n-l}\\
   \end{array}
   \right]\in\mathbb{R}^{n\times n},
   G_1^{(l)}=\left[
   \begin{array}{ccc}
0&0&0\\
  0&\cos\alpha_1&0\\
     0&0&0\\
   \end{array}
   \right]\in\mathbb{R}^{n\times n},\\
   &&  G_2^{(l)}=\left[
   \begin{array}{ccc}
 0&0&0\\
  0&\cos\alpha_2&0\\
     0&0&0\\
   \end{array}
   \right]\in\mathbb{R}^{n\times n},
   G_3^{(l)}=\left[
   \begin{array}{ccc}
 0&0&0\\
  0&\cos\alpha_3&0\\
     0&0&0\\
   \end{array}
   \right]\in\mathbb{R}^{n\times n},
\end{eqnarray*}
for $l=1,2,\cdots,n$, and $\cos^2\alpha_0+\cos^2\alpha_1^2+\cos^2\alpha_2+\cos^2\alpha_3=1$ with $\cos\alpha_0$, $\cos\alpha_1$, $\cos\alpha_2$, $\cos\alpha_3\in[-1,1]$. It is easy to check that all $ G^{(l)}$s are orthogonally JRS-symplectic matrices. Another kind of  orthogonally JRS-symplectic matrices are the direct sum of four identical $n\times n$ Householder matrices
\begin{equation}\label{heq}
  \mathscr{H}^{(l)}\oplus  \mathscr{H}^{(l)} \oplus \mathscr{H}^{(l)}\oplus  \mathscr{H}^{(l)}[v,\beta],
\end{equation}
where $v$ is an $n$-dimensional vector whose first $l-1$ elements are all equal to zero, and $\beta$ is a scalar such that $\beta(\beta v^Tv-2)=0$.

For a positive integer $m$, the quaternion Krylov subspace generated by the  coefficient matrix
$\mathscr{A}\in\mathbb{Q}^{n\times n}$ and a nonzero vector $\mathbf{v}\in\mathbb{Q}^n$  is defined by
 \begin{equation}\label{keq}
  \mathcal{K}_m(\mathscr{A},\mathbf{v})=\mathrm{span}(\mathbf{v},\mathscr{A}\mathbf{v},\cdots,
  \mathscr{A}^{m-1}\mathbf{v}),
 \end{equation}
whose element  is a  right-hand side  linear combination of $\mathbf{v}$, $\mathscr{A}\mathbf{v}$, $\cdots$,
$\mathscr{A}^{m-1}\mathbf{v}$. We comment here that  the above-mentioned right-hand side linear combination cannot be rewritten as the multiplication of a polynomial with degree less than $m-1$ of $\mathscr{A}$ and
the vector $\mathbf{v}$ since the multiplication of the quaternion does not satisfy the commutativity.
For the sake of convenience, the right-hand side linear combination of $\mathbf{v}$, $\mathscr{A}\mathbf{v}$, $\cdots$, $\mathscr{A}^{m-1}\mathbf{v}$ is denoted by
\begin{equation}\label{lc}
  L_m(\mathscr{A},\mathbf{v})=\mathbf{v} \mathbf{a}_0+\mathscr{A}\mathbf{v}\mathbf{a}_1+\cdots+\mathscr{A}^{m-1}\mathbf{v}\mathbf{a}_{m-1},
\end{equation}
where $\mathbf{a}_0$, $\mathbf{a}_1$, $\cdots$, $\mathbf{a}_{m-1}$  are quaternion scalars.

Finally, we review the Givens transformation of quaternion matrices \cite{jo2003}. From \cite[Theorem 3.4]{jo2003}, for a nonzero  vector $\mathbf{x}=[\mathbf{x}_1,\mathbf{x}_2]^T\in\mathbb{Q}^2$, we define
\begin{equation}\label{gm1}
 \widehat{ \mathscr{G}}=
  \left[
  \begin{array}{cc}
  \bar{\mathbf{c}}&\mathbf{s}\\
  -\bar{\mathbf{s}}&\mathbf{c}
  \end{array}
  \right],~\text{with}~
  \mathbf{s}=-\bm{\sigma}\dfrac{\bar{\mathbf{x}}_2}{\|\mathbf{x}\|_2},
  \mathbf{c}=\bm{\sigma}\dfrac{\bar{\mathbf{x}}_1}{\|\mathbf{x}\|_2},|\bm{\sigma}|=1,
\end{equation}
where $\bm{\sigma}$ is arbitrary for the case that  $\mathbf{x}_1$, $\mathbf{x}_2$ are linearly dependent over $\mathbb{R}$, otherwise $\bm{\sigma}=\dfrac{a_1\mathbf{x}_1+a_2\mathbf{x}_2}{|a_1\mathbf{x}_1+a_2\mathbf{x}_2|}\in\mathbb{Q}$ with nonzero vector $[a_1,a_2]^T\in\mathbb{R}^2$. Then $\bar{\mathscr{G}}$ is a unitary matrix and $\widehat{\mathscr{G}}^*\mathbf{x}=\bm{\sigma}[\|\mathbf{x}\|_2,0]^T$. Especially, let $\bm{\sigma}=\dfrac{\mathbf{x}_1}{|\mathbf{x}_1|}$ and
\begin{equation}\label{gm2}
  \mathscr{G} =
  \left[
  \begin{array}{cc}
 c&\mathbf{s}\\
  -\bar{\mathbf{s}}&c
  \end{array}
  \right],~\text{with}~
  \mathbf{s}=\dfrac{\mathbf{x}_1}{|\mathbf{x}_1|}
  \dfrac{\bar{\mathbf{x}}_2}{\|\mathbf{x}\|_2},~
 c=\dfrac{|\mathbf{x}_1|}{\|\mathbf{x}\|_2}.
\end{equation}
Then
\begin{equation}\label{gm3}
  \mathscr{G}\left[
  \begin{array}{c}
  \mathbf{x}_1\\
  \mathbf{x}_2
  \end{array}\right]=\left[
  \begin{array}{c}
  \bm{\xi}\\
  0
  \end{array}
  \right],~
  \text{with}~\bm{\xi}=\dfrac{\mathbf{x}_1}{|\mathbf{x}_1|}\|\mathbf{x}\|_2.
\end{equation}
Obviously, the quaternion Givens matrix (\ref{gm2}) is the natural generalization of the real or complex Givens matrix.
%%%%%%%%%%%%%%%
Similarly, for a nonzero quaternion vector $\mathbf{y}=[\mathbf{y}_1,\mathbf{y}_2]\in\mathbb{Q}^2$, we have
\begin{equation}\label{gm4}
 \left[
  \begin{array}{cc}
  \mathbf{y}_1& \mathbf{y}_2
  \end{array}\right] \mathscr{G}'=\left[
  \begin{array}{cc}
  \bm{\xi}&
  0
  \end{array}
  \right],~
  \text{with}~\bm{\xi}=\dfrac{\mathbf{y}_1}{|\mathbf{y}_1|}\|\mathbf{y}\|_2,
\end{equation}
where
\begin{equation}\label{gm5}
\mathscr{G}' =
  \left[
  \begin{array}{cc}
 c'&\mathbf{s}'\\
  \bar{\mathbf{s}}'&-c'
  \end{array}
  \right],~\text{with}~
  \mathbf{s}'=\dfrac{\bar{\mathbf{y}}_1}{|\mathbf{y}_1|}
  \dfrac{\mathbf{y}_2}{\|\mathbf{y}\|_2},~
 c'=\dfrac{|\mathbf{y}_1|}{\|\mathbf{y}\|_2}.
\end{equation}

\section{Quaternion Saunders-Simon-Yip tridiagonalization }
In this section, we derive the quaternion Saunders-Simon-Yip tridiagonalization procedure
from employing three-term recursions.  We first give the following definitions and results.%First, we recall some useful results.

\begin{definition} [\cite{jwzc2018}]
Let $M\in\mathbb{R}^{4n\times 4n}$ be a JRS-symmetric matrix with the block form {\rm(\ref{jrs})}. If $M_0\in\mathbb{R}^{n\times n}$ is an upper Hessenberg matrix, and $M_1,M_2,M_3\in\mathbb{R}^{n\times n}$ are  upper triangular matrices, then  $M\in\mathbb{R}^{4n\times 4n}$ is called an upper JRS-Hessenberg matrix.
\end{definition}

\begin{definition}\label{defn2.1h}
Let $M\in\mathbb{R}^{4n\times 4n}$ be a JRS-symmetric matrix with the block form {\rm(\ref{jrs})}. If $M_0\in\mathbb{R}^{n\times n}$ and $M_1, M_2,M_3\in\mathbb{R}^{n\times n}$ are a tridiagonal matrix and diagonal matrices, respectively, then  $M\in\mathbb{R}^{4n\times 4n}$ is called a strictly JRS-tridiagonal matrix.
\end{definition}

\begin{proposition}[\cite{jn2021}]\label{thm1}
Let $M\in\mathbb{R}^{4n\times 4n}$ be a JRS-symmetric matrix with the block form  {\rm(\ref{jrs})}. Then there exists an orthogonally JRS-symplectic matrix $W\in\mathbb{R}^{4n\times 4n}$ such that $W^TMW=H$ is an upper JRS-Hessenberg
matrix.
\end{proposition}

The following theorem shows that a JRS-symmetric matrix can be transformed into a strictly JRS-tridiagonal matrix by using orthogonally JRS-symplectic matrices.

\begin{theorem}\label{thm2}
Let $M\in\mathbb{R}^{4n\times 4n}$ be a JRS-symmetric matrix with the block form  {\rm(\ref{jrs})}. Then there exist orthogonally JRS-symplectic matrices $P, Q\in\mathbb{R}^{4n\times 4n}$ such that
\begin{equation*}
	  P^TMQ=T=
	\left[
	\begin{array}{cccc}
		T_0&T_2&T_1&T_3\\
		-T_2&T_0&T_3&-T_1\\
		-T_1&-T_3&T_0&T_2\\
		-T_3&T_1&-T_2&T_0
	\end{array}
	\right]\in\mathbb{R}^{4n\times 4n}
\end{equation*}
is a strictly JRS-tridiagonal matrix, where off-diagonal elements of the tridiagonal $T_0$ are nonnegative real numbers.
\end{theorem}

The proof of Theorem \ref{thm2}  can be found in the Appendix.

\begin{theorem}\label{thm3}
Let $\mathscr{A}=A_0+A_1\mathbf{i}+A_2\mathbf{j}+A_3\mathbf{k}\in\mathbb{Q}^{n\times n}$. Then there exist two unitary quaternion  matrices $\mathscr{P}\in\mathbb{Q}^{n\times n}$ and $\mathscr{Q}\in\mathbb{Q}^{n\times n}$ such that
\begin{equation}\label{eq6}
\bm{\mathscr{P}}^*\mathscr{A}\mathscr{Q}=\mathscr{T},
\end{equation}
where
\begin{equation}\label{eq7}
\mathscr{T}=\left[\begin{array}{ccccc}
\bm{\alpha}_1&\gamma_1&&&\\
\beta_1&\bm{\alpha}_2&\gamma_2&&\\
&\ddots&\ddots&\ddots&\\
&&\beta_{n-2}&\bm{\alpha}_{n-1}&\gamma_{n-1}\\
&&&\beta_{n-1}&\bm{\alpha}_n
\end{array}
\right]\in\mathbb{Q}^{n\times n}
\end{equation}
is a tridiagonal quaternion matrix whose off-diagonal elements are nonnegative real numbers, which is called a strict tridiagonal quaternion matrix.
\end{theorem}
\begin{proof}
According to Theorem \ref{thm2}, there exist
orthogonally JRS-symplectic matrices $P,Q\in\mathbb{R}^{4n\times 4n}$ such that
\begin{equation}\label{eq8}
P^T\mathcal{R}(\mathscr{A})Q=T=
  \left[
  \begin{array}{cccc}
  T_0&T_2&T_1&T_3\\
  -T_2&T_0&T_3&-T_1\\
  -T_1&-T_3&T_0&T_2\\
  -T_3&T_1&-T_2&T_0
  \end{array}
  \right]
\end{equation}
is a strictly JRS-tridiagonal matrix, where the off-diagonal elements of $T_0\in\mathbb{R}^{n\times n}$  and $T_1, T_2,T_3\in\mathbb{R}^{n\times n}$ are nonnegative real numbers and zeros, respectively. Taking the inverse homeomorphic mapping $\mathcal{R}^{-1}$ on both sides of the last equation yields
\begin{equation*}
\mathscr{P}^*\mathscr{A}\mathscr{Q}=\mathscr{T},
\end{equation*}
where $\mathscr{P}=\mathcal{R}^{-1}(P)$,  $\mathscr{Q}=\mathcal{R}^{-1}(Q)$ and $\mathscr{T}=\mathcal{R}^{-1}(T)$. This yields the desired result.
\end{proof}

By the proof of Theorem \ref{thm3}, it is known that  JRS-symmetry of $\mathcal{R}(\mathscr{A})$ is preserved. Consequently,  the decomposition in (\ref{eq8}) is called the structure-preserving of four real matrices
$A_0, A_1, A_2, A_3 \in\mathbb{R}^{n\times n}$. This also means that we don't need to produce the explicit real counterparts  (the $4\times4$ real block matrices)  of quaternion matrices, just need to generate and store the first column blocks of real representations, which greatly reduces the storage and computation.

Rewriting (\ref{eq6}) yields
\begin{equation}\label{eq9}
\mathscr{A}\mathscr{Q}=\mathscr{P}\mathscr{T},\mathscr{A}^*\mathscr{P}=\mathscr{Q}\mathscr{T}^*.
\end{equation}
It is easy to see that if we fix the first column vectors of $\mathscr{P}$ and $\mathscr{Q}$ and the off-diagonal elements of $\mathscr{T}$ to be positive, then the transformation is uniquely determined.

By $\mathbf{q}_k$ and $\mathbf{p}_k$ we denote the $k$th column of $\mathscr{Q}$ and $\mathscr{P}$, respectively. From (\ref{eq9}), we can deduce the following recursion formulas: %for  and  $\bm{\alpha}_k$, $\beta_k$, $\gamma_k$. Given $\mathbf{p}_1$ and $\mathbf{q}_1$, we see that
\begin{eqnarray*}
% \nonumber to remove numbering (before each equation)
   &&\bm{\alpha}_1=\langle\mathscr{A}\mathbf{q}_1, \mathbf{p}_1\rangle, \widetilde{\mathbf{p}}_2=\mathscr{A}\mathbf{q}_1-\mathbf{p}_1\bm{\alpha}_1,
   \widetilde{\mathbf{q}}_2=\mathscr{A}^*\mathbf{p}_1-\mathbf{q}_1\bm{\alpha}_1^*\\
   &&
   \Rightarrow \beta_1=\|\widetilde{\mathbf{p}}_2\|_2, \mathbf{p}_2=\widetilde{\mathbf{p}}_2/\beta_1,
 \gamma_1=\|\widetilde{\mathbf{q}}_2\|_2, \mathbf{q}_2=\widetilde{\mathbf{q}}_2/\gamma_1,
\end{eqnarray*}
and for $k\geq 2$ with known $\mathbf{q}_1,\cdots,\mathbf{q}_k$, $\mathbf{p}_1,\cdots,\mathbf{p}_k$, $\bm{\alpha}_1$, $\cdots$, $\bm{\alpha}_{k-1}$, $\beta_1$, $\cdots$, $\beta_{k-1}$, $\gamma_1$, $\cdots$, $\gamma_{k-1}$:
\begin{eqnarray*}
% \nonumber to remove numbering (before each equation)
   &&\alpha_k=\langle \mathscr{A}\mathbf{q}_k,\mathbf{p}_k\rangle, \widetilde{\mathbf{p}}_{k+1}=\mathscr{A}\mathbf{q}_k-\mathbf{p}_k\bm{\alpha}_k-\mathbf{p}_{k-1}\gamma_{k-1},
   \widetilde{\mathbf{q}}_{k+1}=\mathscr{A}^*\mathbf{p}_k-\mathbf{q}_k\bm{\alpha}_k^*-\mathbf{q}_{k-1}\beta_{k-1}\\
   &&
   \Rightarrow \beta_k=\|\widetilde{\mathbf{p}}_{k+1}\|_2, \mathbf{p}_{k+1}=\widetilde{\mathbf{p}}_{k+1}/\beta_k,
 \gamma_k=\|\widetilde{\mathbf{q}}_{k+1}\|_2, \mathbf{q}_{k+1}=\widetilde{\mathbf{q}}_{k+1}/\gamma_k.
\end{eqnarray*}

Hence we may get the following quaternion Saunders-Simon-Yip tridiagonalization algorithm.

\begin{algorithm}[!ht]
\caption{Quaternion Saunders-Simon-Yip tridiagonalization procedure}
\label{alg3.1}
\begin{algorithmic}[1]
\State Choose two arbitrary quaternion vectors $\mathbf{b}\neq\mathbf{0}$ and $\mathbf{c}\neq\mathbf{0}$.
\State Set $\mathbf{p}_0=\mathbf{q}_0=\mathbf{0}$, $\beta_0=\|\mathbf{b}\|_2$, $\gamma_0=\|\mathbf{c}\|_2$, and  $\mathbf{p}_1=\mathbf{b}/\beta_0$, $\mathbf{q}_1=\mathbf{c}/\gamma_0$.
\State \textbf{for} $i=1,\cdots,m$ \textbf{do}
\State \quad $\bm{\alpha}_{i}=\langle\mathscr{A} \mathbf{q}_i,\mathbf{p}_i\rangle$;
\State \quad  $\widetilde{\mathbf{p}}=\mathscr{A} \mathbf{q}_i-\mathbf{p}_{i}\bm{\alpha}_{i}-\mathbf{p}_{i-1}\gamma_{i-1}$;
\State \quad $\widetilde{\mathbf{q}}=\mathscr{A}^*\mathbf{p}_i-\mathbf{q}_i\bm{\alpha}_i^*-\mathbf{q}_{i-1}\beta_{i-1}$;
\State \quad $\beta_i=\|\widetilde{\mathbf{p}}\|_2$;
\State \quad   $\gamma_i=\|\widetilde{\mathbf{q}}\|_2$;
\State \quad \textbf{if} $\beta_{i}=0$ or $\gamma_i=0$, \textbf{then} stop
\State \quad  $\mathbf{p}_{i+1}=\widetilde{\mathbf{p}}/\beta_i$;
\State \quad $\mathbf{q}_{i+1}=\widetilde{\mathbf{q}}/\gamma_i$;
\State \textbf{end~for}
\end{algorithmic}
\end{algorithm}

It is noted that the main work of quaternion Saunders-Simon-Yip tridiagonalization procedure  for non-Hermitian matrices is to compute $\mathscr{A} \mathbf{q}_i$, $\mathscr{A}^*\mathbf{p}_i$, and $\bm{\alpha}_{i}=\langle\mathscr{A} \mathbf{q}_i,\mathbf{p}_i\rangle$. In the calculation, the above quaternion matrix-vector multiplication and inner
product are realized by the function (\ref{jrs}) as follows:
\begin{eqnarray*}
% \nonumber to remove numbering (before each equation)
    &&\mathscr{A} \mathbf{q}_i=\mathcal{R}^{-1}(\mathcal{R}(\mathscr{A})\mathcal{R}(\mathbf{q}_i)),
   \mathscr{A}^*\mathbf{p}_i=\mathcal{R}^{-1}(\mathcal{R}(\mathscr{A}^*)\mathcal{R}(\mathbf{p}_i)),\\
   &&\bm{\alpha}_{i}=\langle\mathscr{A} \mathbf{q}_i,\mathbf{p}_i\rangle =\mathcal{R}^{-1}(\mathcal{R}(\mathbf{p}_i^*)\mathcal{R}(\mathscr{A} \mathbf{q}_i)).
\end{eqnarray*}
We can implement the above calculations by only using the first block columns of
their real counterparts, which saves three-quarters of the computational operations.
That is,
\begin{eqnarray*}
% \nonumber to remove numbering (before each equation)
    \mathcal{R} (\mathscr{A} \mathbf{q}_i)_c=\mathcal{R}(\mathscr{A})\mathcal{R}(\mathbf{q}_i)_c,
   \mathcal{R} (\mathscr{A}^*\mathbf{p}_i)_c= \mathcal{R}(\mathscr{A})^T\mathcal{R}(\mathbf{p}_i)_c,
 \mathcal{R}(\bm{\alpha}_{i}) =  \mathcal{R}(\mathbf{p}_i)^T\mathcal{R}(\mathscr{A} \mathbf{q}_i)_c.
\end{eqnarray*}

Let  $\mathscr{T}_m$ denote the leading $m\times m$ submatrix of $\mathscr{T}$. Then  $\mathscr{T}_m$ is also a  tridiagonal quaternion matrix. Set %$\mathscr{P}_m$ and $\mathscr{Q}_m$ be matrices whose columns are the vectors $\mathbf{p}_j$ and $\mathbf{q}_j$, i.e.,
$\mathscr{P}_m=[\mathbf{p}_1,\cdots,\mathbf{p}_m]$ and $\mathscr{Q}_m=[\mathbf{q}_1,\cdots,\mathbf{q}_m]$. Then the first $m$ steps can be written as
\begin{subequations}\label{eq10}
\begin{align}
\mathscr{A}\mathscr{Q}_m=\mathscr{P}_m\mathscr{T}_m+\beta_m\mathbf{p}_{m+1}\mathbf{e}_m^*,\label{eq10_1}\\
\mathscr{A}^*\mathscr{P}_m=\mathscr{Q}_m\mathscr{T}_m^*+\gamma_m\mathbf{q}_{m+1}\mathbf{e}_m^*,\label{eq10_2}
\end{align}
\end{subequations}
where $\mathbf{e}_m$ is the $m$th column of the unit matrix.

\begin{proposition}\label{prop1}
 Suppose that $m$ steps of {\rm Algorithm \ref{alg3.1}} have been taken. Then
 \begin{subequations}\label{eq11}
\begin{align}
\mathbf{q}_j^*\mathbf{q}_k=\langle\mathbf{q}_k,\mathbf{q}_j\rangle=0, 0\leq k<j\leq m,\label{eq11_1}\\
\mathbf{p}_j^*\mathbf{p}_k=\langle\mathbf{p}_k,\mathbf{p}_j\rangle=0, 0\leq k<j\leq m,\label{eq11_2}\\
\mathbf{p}_j^*\mathscr{A}\mathbf{q}_k=\langle\mathscr{A}\mathbf{q}_k,\mathbf{p}_j\rangle=0, 0\leq k<j-1\leq m-1,\label{eq11_3}\\
\mathbf{q}_j^*\mathscr{A}^*\mathbf{p}_k=\langle\mathscr{A}^*\mathbf{p}_k,\mathbf{q}_j\rangle=0, 0\leq k<j-1\leq m-1.\label{eq11_4}
\end{align}
\end{subequations}
\end{proposition}
\begin{proof}
From Algorithm \ref{alg3.1}, we have
\begin{equation}\label{eq12_1}
\mathbf{p}_{i+1}\beta_i=\mathscr{A}\mathbf{q}_i- \mathbf{p}_i\bm{\alpha}_i-\mathbf{p}_{i-1}\gamma_{i-1}, ~
\mathbf{q}_{i+1}\gamma_i=\mathscr{A}^*\mathbf{p}_i-\mathbf{q}_i\bm{\alpha}_i^*-\mathbf{q}_{i-1}
\beta_{i-1},%\label{eq12_2}
\end{equation}
where $i=1,2,\cdots,m$. Then
\begin{subequations}\label{eq12}
\begin{align}
\mathbf{p}_{i+1}^*=\frac{1}{\beta_i}(\mathbf{q}_i^*\mathscr{A}^*- \bm{\alpha}_i^*\mathbf{p}_i^*-\gamma_{i-1}\mathbf{p}_{i-1}^*),\label{eq12_4}\\
\mathbf{q}_{i+1}^*=\frac{1}{\gamma_i}(\mathbf{p}_i^*\mathscr{A} -\bm{\alpha}_i\mathbf{q}_i^*-\beta_{i-1}\mathbf{q}_{i-1}^*),\label{eq12_3}
\end{align}
\end{subequations}
where $i=1,2,\cdots,m$.

Now we prove the results  by induction on $m$. For $m=1$, (\ref{eq11_1}) and (\ref{eq11_2}) are trivial.
For $m=2$, we see from (\ref{eq12_4}) and (\ref{eq12_3}) that
\begin{eqnarray*}
% \nonumber to remove numbering (before each equation)
 \mathbf{q}_{2}^*\mathbf{q}_1&=& \frac{1}{\gamma_1}(\mathbf{p}_1^*\mathscr{A} \mathbf{q}_1 - \bm{\alpha}_1\mathbf{q}_1^*\mathbf{q}_1-\beta_{0} \mathbf{q}_{0}^*\mathbf{q}_1)
   = \frac{1}{\gamma_1}(\mathbf{p}_1^*\mathscr{A} \mathbf{q}_1 - \bm{\alpha}_1)=0\\
   \mathbf{p}_{2}^*\mathbf{p}_{1}&=&\frac{1}{\beta_1}(\mathbf{q}_1^*\mathscr{A}^*\mathbf{p}_{1}- \bm{\alpha}_1^*\mathbf{p}_1^*\mathbf{p}_{1}-\gamma_{0}\mathbf{p}_{0}^*\mathbf{p}_{1})=\frac{1}{\beta_1}(
   \mathbf{q}_1^*\mathscr{A}^*\mathbf{p}_{1}- \bm{\alpha}_1^*)=0.
\end{eqnarray*}
Also, it is obvious that  $\mathbf{q}_{2}^*\mathbf{q}_0=0$, $\mathbf{p}_{2}^*\mathbf{p}_0=0$, $\mathbf{p}_2^*\mathscr{A}\mathbf{q}_0=0$ and $\mathbf{q}_2^*\mathscr{A}^*\mathbf{p}_0=0$. Therefore, (\ref{eq11}) holds for $m=2$.
%%%%%%%%%%%%%%%%%%%%%%%%%%%%%
 Suppose that (\ref{eq11}) is true for  $m=l$.  For $m=l+1$, it follows from (\ref{eq12_3}) and line 4 of Algorithm \ref{alg3.1} that
 \begin{eqnarray*}
 % \nonumber to remove numbering (before each equation)
   \mathbf{q}_{l+1}^*\mathbf{q}_l&=& \frac{1}{ \gamma_l}(\mathbf{p}_l^*\mathscr{A} \mathbf{q}_l - \bm{\alpha}_l\mathbf{q}_l^*\mathbf{q}_l-\beta_{l-1} \mathbf{q}_{l-1}^*\mathbf{q}_l)=
   \frac{1}{ \gamma_l}(\mathbf{p}_l^*\mathscr{A} \mathbf{q}_l - \bm{\alpha}_l)=\frac{1}{ \gamma_l}( \bm{\alpha}_l
   - \bm{\alpha}_l)=0.
 \end{eqnarray*}
For showing (\ref{eq11_1}) for $k=l-1$, we compute
\begin{eqnarray*}
% \nonumber to remove numbering (before each equation)
   \mathbf{q}_{l+1}^*\mathbf{q}_{l-1}&=& \frac{1}{\gamma_l}(\mathbf{p}_l^*\mathscr{A} \mathbf{q}_{l-1} - \bm{\alpha}_l\mathbf{q}_l^*\mathbf{q}_{l-1}-\beta_{l-1} \mathbf{q}_{l-1}^*\mathbf{q}_{l-1})\\
   &=&\frac{1}{\gamma_l}( \mathbf{p}_l^*\mathscr{A} \mathbf{q}_{l-1}-\beta_{l-1})\\
    &=& \frac{1}{\gamma_l}[\mathbf{p}_l^*(\mathbf{p}_{l-2}\gamma_{l-2}+\mathbf{p}_{l-1}\bm{\alpha}_{l-1}+\mathbf{p}_{l}\beta_{l-1})-\beta_{l-1}]\\
    &=&\frac{1}{\gamma_l}(\beta_{l-1}-\beta_{l-1})=0.
\end{eqnarray*}
By  (\ref{eq12_3})  and the induction principle, for  $k=0,1,\cdots,l-2$, we obtain
\begin{eqnarray}
% \nonumber to remove numbering (before each equation)
   \mathbf{q}_{l+1}^*\mathbf{q}_k&=& \frac{1}{\gamma_l}(\mathbf{p}_l^*\mathscr{A} \mathbf{q}_k - \bm{\alpha}_l\mathbf{q}_l^*\mathbf{q}_k-\beta_{l-1} \mathbf{q}_{l-1}^*\mathbf{q}_k)\nonumber\\
   &=& \frac{1}{\gamma_l}\mathbf{p}_l^*\mathscr{A} \mathbf{q}_k\nonumber\\
    &=&\frac{1}{\gamma_l} \mathbf{p}_l^*(\mathbf{p}_{k-1}\gamma_{k-1}+\mathbf{p}_k\bm{\alpha}_k+\mathbf{p}_{k+1}\beta_k)=0,
    \label{eq250126}
\end{eqnarray}
where the third equality uses (\ref{eq12_1}) for $\mathscr{A} \mathbf{q}_k$.
Thus (\ref{eq11_1}) is proved for $m=l+1$. Similarly, (\ref{eq11_2})  can be shown for $m=l+1$.
In addition, it follows from the relation (\ref{eq250126}) that
\begin{equation*}
\mathbf{p}_l^*\mathscr{A} \mathbf{q}_k=0,k=0,1,\cdots,l-1.
\end{equation*}
Similarly, we get
\begin{equation*}
\mathbf{q}_l^*\mathscr{A}^* \mathbf{p}_k=0,k=0,1,\cdots,l-1.
\end{equation*}
Therefore, (\ref{eq11})   is true for $m=l+1$, which proves the desired results.
\end{proof}

\begin{remark}\label{rem3.1}
The above quaternion Saunders-Simon-Yip tridiagonalization procedure is different from the quaternion Lanczos biorthogonalization process given in {\rm \cite[Algorithm 3.2]{lw2024}}. In fact, the quaternion Lanczos biorthogonalization process for general non-Hermitian quaternion matrices $\mathscr{A}\in\mathbb{Q}^{n\times n}$ is computed from two arbitrarily selected unit length vectors $\mathbf{w}_1$ and $\mathbf{v}_1$, and then generate subsequently vectors $\mathbf{w}_2,\mathbf{w}_3,\cdots$ and $\mathbf{v}_2, \mathbf{v}_3,\cdots$ by three-term recurrence relations.

Let $\mathscr{W}_m=[\mathbf{w}_1,\mathbf{w}_2,\cdots,\mathbf{w}_m]$ and $\mathscr{V}_m=[\mathbf{v}_1,\mathbf{v}_2,\cdots,\mathbf{v}_m]$. We obtain in general a tridiagonal form  $\mathscr{S}_m=\mathscr{W}_m^*\mathscr{A}\mathscr{V}_m$, where $\mathscr{W}_m^*\mathscr{V}_m=\mathbf{I}_m$. The crucial difference to our proposed quaternion Saunders-Simon-Yip tridiagonalization process is the fact that we have $\mathscr{P}_m^*\mathscr{P}_m=\mathbf{I}_m$ and $\mathscr{Q}_m^*\mathscr{Q}_m=\mathbf{I}_m$, i.e., the partially unitary instead of a transformation with the matrices $\mathscr{W}_m$ and $\mathscr{V}_m$ such that $\mathscr{W}_m^*\mathscr{V}_m=\mathbf{I}_m$. Especially, if $m=n$,  the quaternion  Saunders-Simon-Yip tridiagonalization process preserves the singular values of the non-Hermitian quaternion matrix $\mathscr{A}\in\mathbb{Q}^{n\times n}$ since   $\mathscr{P}_n^*\mathscr{A}\mathscr{Q}_n=\mathscr{T}_n$ with $\mathscr{P}_n,\mathscr{Q}_n\in\mathbb{Q}^{n\times n}$ being unitary quaternion matrices. In contrast, the quaternion Lanczos biorthogonalization process is a similar transformation of the quaternion matrix $\mathscr{A}\in\mathbb{Q}^{n\times n}$ since  $\mathscr{S}_n=\mathscr{W}_n^*\mathscr{A}\mathscr{V}_n=\mathscr{V}_n^{-1}\mathscr{A}\mathscr{V}_n$ by using $\mathscr{W}_n^*\mathscr{V}_n=\mathbf{I}_n$.
\end{remark}

\begin{remark}\label{rem3.2}
The above quaternion Saunders-Simon-Yip tridiagonalization procedure extend the Hermitian quaternion Lanczos algorithm to the non-Hermitian case: instead of one three-term recurrence of the Hermitian quaternion Lanczos algorithm, we obtain two three-term recurrences. In particularly, if $\mathscr{A}$ is Hermitian, then the proposed algorithm reduces to the Hermitian quaternion Lanczos algorithm given in {\rm\cite{jn2021}}. %see Algorithm \ref{alg3.2}.
\end{remark}

\begin{remark}\label{rem3.3}
The distinct character of the quaternion Saunders-Simon-Yip tridiagonalization algorithm is apparent from the subspaces involved in the computation. From {\rm(\ref{eq10})}, it is known  that for $k=1,2,\cdots$,
\begin{eqnarray*}
% \nonumber to remove numbering (before each equation)
   && \mathbf{p}_{2k}\in\mathrm{span}(\mathbf{b},\mathscr{A}\mathscr{A}^*\mathbf{b},\cdots,
(\mathscr{A}\mathscr{A}^*)^{k-1}\mathbf{b}, \mathscr{A}\mathbf{c}, \mathscr{A}\mathscr{A}^*\mathscr{A}\mathbf{c},\cdots,(\mathscr{A}\mathscr{A}^*)^{k-1}\mathscr{A}\mathbf{c}),\\
 && \mathbf{p}_{2k+1}\in\mathrm{span}(\mathbf{b},\mathscr{A}\mathscr{A}^*\mathbf{b},\cdots,
(\mathscr{A}\mathscr{A}^*)^{k}\mathbf{b}, \mathscr{A}\mathbf{c}, \mathscr{A}\mathscr{A}^*\mathscr{A}\mathbf{c},\cdots,(\mathscr{A}\mathscr{A}^*)^{k-1}\mathscr{A}\mathbf{c}),\\
 && \mathbf{q}_{2k}\in\mathrm{span}(\mathbf{c},\mathscr{A}^*\mathscr{A}\mathbf{c},\cdots,
(\mathscr{A}^*\mathscr{A})^{k-1}\mathbf{c}, \mathscr{A}^*\mathbf{b}, \mathscr{A}^*\mathscr{A}\mathscr{A}^*\mathbf{b},\cdots,(\mathscr{A}^*\mathscr{A})^{k-1}\mathscr{A}^*\mathbf{b}),\\
 && \mathbf{q}_{2k+1}\in\mathrm{span}(\mathbf{c},\mathscr{A}^*\mathscr{A}\mathbf{c},\cdots,
(\mathscr{A}^*\mathscr{A})^{k}\mathbf{c}, \mathscr{A}^*\mathbf{b}, \mathscr{A}^*\mathscr{A}\mathscr{A}^*\mathbf{b},\cdots,(\mathscr{A}^*\mathscr{A})^{k-1}\mathscr{A}^*\mathbf{b}).
\end{eqnarray*}
Thus the underlying subspaces are not  Krylov subspaces, which can be seen as the union of two Krylov subspaces generated with the normal equation  matrix $\mathscr{A}\mathscr{A}^*$ $($or $\mathscr{A}^*\mathscr{A}$$)$ with initial vectors $\mathbf{b}$ and $\mathscr{A}\mathbf{c}$ $($or $\mathbf{c}$ and $\mathscr{A}^*\mathbf{b}$$)$. When $\mathbf{b}=\mathbf{c}$, we also see that the subspaces span$(\mathbf{q}_1,\cdots,\mathbf{q}_{2k})$ and span$(\mathbf{q}_1,\cdots,\mathbf{q}_{2k},\mathbf{q}_{2k+1})$ contain the Krylov subspace generated by $k$ steps of the Hermitian Lanczos algorithm applied to the normal equation.
\end{remark}

\section{The quaternion conjugate gradient-type method}
In this section, we propose two kinds of quaternion conjugate gradient-type methods for solving the non-Hermitian quaternion linear systems (\ref{1.1})  based on the Galerkin condition and the minimum residual condition, which is denoted by QNHERLQ and QNHERQR, respectively.

\subsection{The QNHERLQ method}
Let $\mathbf{x}_0$ be an initial guess of the solution for the quaternion linear systems (\ref{1.1}), and $\mathbf{r}_0$ be the corresponding initial residual. Let $\beta=\|\mathbf{r}_0\|_2$ and $\mathbf{p}_1=\mathbf{r}_0/\beta$. Choose an arbitrary unit vector $\mathbf{q}_1$ and start the quaternion Saunders-Simon-Yip tridiagonalization algorithm. From the quaternion vectors $\mathbf{p}_j$, $\mathbf{q}_j$, $j = 1,\cdots,m$,  and the
strictly tridiagonal quaternion matrix $\mathscr{T}_m$ generated by Algorithm \ref{alg3.1}, we define an approximation solution to (\ref{1.1}) as follows: let $\mathbf{x}_m$ be the vector from the affine subspace $\mathbf{x}_0+\mathrm{span}(\mathbf{q}_1,\mathbf{q}_2,\cdots,\mathbf{q}_m)$ such that the residual vector $\mathbf{r}_m=\mathbf{b}-\mathscr{A}\mathbf{x}_m$ is orthogonal to $\mathrm{span}(\mathbf{p}_1,\mathbf{p}_2,\cdots,\mathbf{p}_m)$. Obviously, any quaternion vector $\mathbf{x}$ in $\mathbf{x}_0+\mathrm{span}(\mathbf{q}_1,\mathbf{q}_2,\cdots,\mathbf{q}_m)$ can be rewritten as
$\mathbf{x} =\mathbf{x}_0+\mathscr{Q}_m\mathbf{y}$,
where $\mathbf{y}\in\mathbb{Q}^m$. The relation (\ref{eq11}) results in
\begin{eqnarray}
% \nonumber to remove numbering (before each equation)
  \mathbf{b}-\mathscr{A}\mathbf{x} &=&  \mathbf{b}-\mathscr{A}(\mathbf{x}_0+\mathscr{Q}_m\mathbf{y})=
   \mathbf{r}_0-\mathscr{A} \mathscr{Q}_m\mathbf{y} \nonumber\\
   &=&\beta\mathbf{p}_1-(\mathscr{P}_m\mathscr{T}_m+\beta_m\mathbf{p}_{m+1}\mathbf{e}_m^*)\mathbf{y}\nonumber\\
      &=&\mathscr{P}_m(\beta \mathbf{e}_1-\mathscr{T}_m\mathbf{y})-\beta_m\mathbf{p}_{m+1}\mathbf{e}_m^*\mathbf{y}.\label{4.2}
\end{eqnarray}
Using the orthogonality of $\mathbf{b}-\mathscr{A}\mathbf{x}$ and  $\mathrm{span}(\mathbf{p}_1,\mathbf{p}_2,\cdots,\mathbf{p}_m)$ yields
\begin{equation}\label{5.16}
\mathscr{T}_m\mathbf{y}=\beta\mathbf{e}_1,
\end{equation}
where $\beta=\|\mathbf{r}_0\|_2$. Once the solution $\mathbf{y}_m$ of (\ref{5.16}) is obtained, then the approximation solution has the following form
\begin{equation}\label{5.17}
\mathbf{x}_m=\mathbf{x}_0+\mathscr{Q}_m\mathbf{y}_m.
\end{equation}
From (\ref{4.2}) and (\ref{5.16}), we have
\begin{equation}\label{ree}
 \|\mathbf{r}_m\|_2=\beta_m|\mathbf{e}_m^*\mathbf{y}_m|,
\end{equation}
which can be  computed without the knowledge of $\mathbf{x}_m$.

Next, we discuss the practical implementation of the computation of $\mathbf{x}_m$ via the quaternion LQ factorization of $\mathscr{T}_m$, this method is called the QNHERLQ for simplitity.

Let the orthogonal trigonometric decomposition of $\mathscr{T}_m$ be given by
\begin{equation}\label{5.18}
  \mathscr{T}_m=\widetilde{\mathscr{L}}_m\mathscr{V}_m,\mathscr{V}_m^*\mathscr{V}_m=\mathbf{I}_m,
\end{equation}
where $\widetilde{\mathscr{L}}_m$ is a lower triangular quaternion matrix.    The tilde is used to indicate that $\widetilde{\mathscr{L}}_m$ differs from the $m\times m$
leading part of $\widetilde{\mathscr{L}}_{m+1}$ in the $(m, m)$ element only.  Substituting (\ref{5.18}) into (\ref{5.16}) and (\ref{5.17})  yields
\begin{equation*}
\widetilde{\mathscr{L}}_m\mathscr{V}_m\mathbf{y}_m=\beta\mathbf{e}_1,
 \mathbf{x}_m=\mathbf{x}_0+\mathscr{Q}_m\mathscr{V}_m^* (\mathscr{V}_m\mathbf{y}_m).
\end{equation*}
Set
\begin{equation}\label{e5910}
\widetilde{\mathbf{z}}_m=\mathscr{V}_m\mathbf{y}_m,~\widetilde{\mathscr{W}}_m=\mathscr{Q}_m\mathscr{V}_m^{*}.
\end{equation}
Then
\begin{equation}\label{e5911}
\widetilde{\mathscr{L}}_m\widetilde{\mathbf{z}}_m=\beta\mathbf{e}_1, ~\mathbf{x}_m=\mathbf{x}_0+\widetilde{\mathscr{W}}_m\widetilde{\mathbf{z}}_m.
\end{equation}

Now we analyze the structure of $\mathscr{V}_m$.
The tridiagonal structure shows that $\mathscr{V}_m$ can be written as
\begin{equation}\label{5.19}
\mathscr{V}_m=\mathscr{G}_{1,2}\mathscr{G}_{2,3}\cdots\mathscr{G}_{m-1,m},
\end{equation}
where
\begin{equation}\label{gmm1}
\mathscr{G}_{i,i+1}=\left[
                   \begin{array}{cccc}
                   \mathbf{I}_{i-1}&&&\\
                     & c _i & \mathbf{s}_i &\\
                     &\bar{\mathbf{s}}_i& - c_i& \\
                     &&&\mathbf{I}_{m-1-i}
                   \end{array}
                 \right]\in\mathbb{Q}^{m\times m},  ~~ c_i^2+|\mathbf{s}_i|^2=1,~~i=1,2,\cdots,m-1.
\end{equation}
For the purpose of establishing recursion formulae, we define $\bm{ \nu'}_1=\bm{\alpha}_1$  and $\widetilde{\delta}_2=\beta_1$. We have
\begin{equation*}
  \left[\begin{array}{cc}
  \bm{ \nu'}_i &\gamma_i\\
  \widetilde{\delta}_{i+1}&\bm{\alpha}_{i+1}\\
  0&\beta_{i+1}
  \end{array}
  \right]
  \left[
  \begin{array}{cc}
  c_i&\mathbf{s}_i\\
  \overline{\mathbf{s}}_i&-c_i
  \end{array}
  \right]=\left[\begin{array}{cc}
 \bm{\nu}_i &0\\
 \bm{\delta}_{i+1}&\bm{\nu'}_{i+1}\\
 \bm{\eta}_{i+2}&\widetilde{\delta}_{i+2}
  \end{array}
  \right].
\end{equation*}
The $i$th Givens transformation  zeros out $\gamma_i$, and hence we obtain
\begin{eqnarray}
% \nonumber to remove numbering (before each equation)
   && \bm{\nu}_i=\dfrac{ \bm{ \nu'}_i}{ |\bm{ \nu'}_i|} \sqrt{| \bm{ \nu'}_i|^2+\gamma_i^2},
   c_i=\dfrac{ |\bm{ \nu'}_i|}{\sqrt{| \bm{ \nu'}_i|^2+\gamma_i^2}},
   \mathbf{s}_i=\dfrac{ \bm{ \bar{\nu}}'_i}{ \bm{ \nu'}_i} \dfrac{ \bm{ \gamma}_i}{\sqrt{| \bm{ \nu'}_i|^2+\gamma_i^2}}\label{e5915}
\end{eqnarray}
and
\begin{eqnarray*}\label{e5916}
& \bm{\delta}_{i+1}=\widetilde{\delta}_{i+1}c_i+\bm{\alpha}_{i+1}\bar{\mathbf{s}}_i, ~\bm{\nu'}_{i+1}=\widetilde{\delta}_{i+1}\mathbf{s}_i-\bm{\alpha}_{i+1}c_i,\\[4pt]
&\widetilde{\delta}_{i+2}=-\beta_{i+1}c_i, ~\bm{\eta}_{i+2}=\beta_{i+1}\bar{\mathbf{s}}_i.
\end{eqnarray*}

By the special structure of the quaternion matrix $\mathscr{V}_m$ in  (\ref{5.19}), $\widetilde{\mathscr{W}}_m$ and  $\widetilde{\mathbf{z}}_m$ can be written as
\begin{eqnarray}
\widetilde{\mathscr{W}}_m=\big[\mathbf{w}_1,\mathbf{w}_2,\cdots,\mathbf{w}_{m-1},\widetilde{\mathbf{w}}_m\big],\label{e5912}\\
\widetilde{\mathbf{z}}_m=\big[\bm{\zeta}_1,\bm{\zeta}_2,\cdots,\bm{\zeta}_{m-1},
\bm{\widetilde{\zeta}}_m\big]^{\mathrm{T}}.\label{e5913}
\end{eqnarray}
From equations (\ref{e5912})  and (\ref{e5913}), $\mathbf{x}_m$ cannot be computed recursively by the relation (\ref{e5911}).
There we  use $\mathscr{L}_m$ to denote $\widetilde{\mathscr{L}}_m$  with $\bm{\nu'}_m$ replaced
by $\bm{\nu}_m$ in the following. Similarly,  we define
\begin{equation*}
\mathscr{W}_m=\big[\mathbf{w}_1,\mathbf{w}_2,\cdots,\mathbf{w}_{m-1}, \mathbf{w}_m\big],
\mathbf{z}_m=\big[\bm{\zeta}_1,\bm{\zeta}_2,\cdots,\bm{\zeta}_{m-1},\bm{\zeta}_m\big]^{T},
\end{equation*}
where $\mathbf{z}_m$ satisfies
\begin{equation}\label{e5917}
  \mathscr{L}_m\mathbf{z}_m=\beta\mathbf{e}_1,
\end{equation}
which, together with  (\ref{e5911}) yields $\bm{\nu'}_m\bm{\widetilde{\zeta}}_m=\bm{\nu}_m\bm{\zeta}_m$. By (\ref{e5915}), we have
\begin{equation*}\label{e5918}
 \bm{\zeta}_m= \bm{\nu}_m^{-1}(\bm{\nu'}_m\bm{\widetilde{\zeta}}_m)
 =\dfrac{ \bm{ \bar{\nu}'}_i}{ |\bm{ \nu'}_i|}\dfrac{1}{ \sqrt{| \bm{ \nu'}_i|^2+\gamma_i^2}}(\bm{\nu'}_m\bm{\widetilde{\zeta}}_m)
 = \dfrac{|\bm{ \nu'}_i|}{ \sqrt{| \bm{ \nu'}_i|^2+\gamma_i^2}}\bm{\widetilde{\zeta}}_m
 =c_m\bm{\widetilde{\zeta}}_m.
\end{equation*}
And, from (\ref{e5912}) and the form of $\mathscr{G}_{m,m+1}$ in (\ref{gmm1}), we obtain
\begin{equation*}\label{e5919}
  \big[\widetilde{\mathbf{w}}_m,\mathbf{q}_{m+1}\big]\left[
                                   \begin{array}{cc}
                                    c_m & \mathbf{s}_m \\
                                     \bar{\mathbf{s}}_m & -c_m \\
                                   \end{array}
                                 \right]=\big[\mathbf{w}_m,\widetilde{\mathbf{w}}_{m+1}\big], ~\widetilde{\mathbf{w}}_1=\mathbf{q}_1.
\end{equation*}

If $\bm{\widetilde{\nu}}_m=0$, then $\widetilde{\mathscr{L}}_m$ is not  invertible and the equation (\ref{e5911})  is unsolvable.  Instead we see from (\ref{e5915})   that $ \mathscr{L}_m$ is  invertible if
$\gamma_m\neq0$ and the equation (\ref{e5917}) is solvable, so $\mathbf{z}_m$ is defined. Hence, rather than updating $\mathbf{x}_m$ each step, we update
\begin{equation}\label{e5920}
  \widetilde{\mathbf{x}}_m=\mathbf{x}_0+\mathscr{W}_m\mathbf{z}_m=\widetilde{\mathbf{x}}_{m-1}+\mathbf{w}_m\bm{\zeta}_m.
\end{equation}
By the relations (\ref{e5911}),  (\ref{e5912}),  (\ref{e5913}) and (\ref{e5920}), we have
\begin{equation}\label{e5921}
 \mathbf{x}_{m+1}=\mathbf{x}_0+\widetilde{\mathscr{W}}_{m+1}\widetilde{\mathbf{z}}_{m+1}=\widetilde{\mathbf{x}}_{m}
 +\widetilde{\mathbf{w}}_{m+1}\bm{\widetilde{\zeta}}_{m+1}.
\end{equation}

Finally, from  the relation (\ref{e5917}), it follows that
\begin{equation*}\label{e5922}
\left\{
  \begin{array}{l}
\bm{\nu}_1\bm{\zeta}_1=\beta  \Longrightarrow \bm{\zeta}_1=\bm{\nu}_1^{-1}\beta,  \\
\bm{\delta}_2\bm{\zeta}_1+\bm{\nu}_2\bm{\zeta}_2=0 \Longrightarrow \bm{\zeta}_2=-\bm{\nu}_2^{-1} \bm{\delta}_2\bm{\zeta}_1,  \\
\bm{\eta}_i\bm{\zeta}_{i-2}+\bm{\delta}_i\bm{\zeta}_{i-1}+\bm{\nu}_i\bm{\zeta}_i=0  \Longrightarrow \bm{\zeta}_{i}=-\bm{\nu}_i^{-1}[ \bm{\eta}_i\bm{\zeta}_{i-2}+\bm{\delta}_i\bm{\zeta}_{i-1}],~i\geq 3.
  \end{array}
\right.
\end{equation*}
Thus we summarize the QNHERLQ scheme as in Algorithm \ref{alg5.2}.

\begin{algorithm}[!ht]
\caption{The QNHERLQ  method}
\label{alg5.2}
Input the coefficient matrix $\mathscr{A}\in\mathbb{Q}^{n\times n}$ and the right-hand side vector $\mathbf{b}\in\mathbb{Q}^n$. Choose an initial vector $\mathbf{x}_0\in\mathbb{Q}^n$ and the tolerance error $\varepsilon>0$.
\begin{algorithmic}[1]
\State Compute $\mathbf{r}_0=\mathbf{b}-\mathscr{A}\mathbf{x}_0$, $\beta=\|\mathbf{r}_0\|_2$, and $\mathbf{q}_1=\mathbf{r}_0/\beta$;
\State Choose $\mathbf{p}_1$ such that $\|\mathbf{p}_1\|_2=1$, e.g., $\mathbf{p}_1=\mathbf{q}_1$;
\State   $\bm{\alpha}_{1}=\langle\mathbf{A} \mathbf{q}_1,\mathbf{p}_1\rangle$,
 $\widetilde{\mathbf{p}}=\mathbf{A} \mathbf{q}_1-\mathbf{p}_{1}\bm{\alpha}_{1}$,
         $\widetilde{\mathbf{q}}=\mathscr{A}^*\mathbf{p}_1-\mathbf{q}_1\bm{\alpha}_1^*$;
\State  $\beta_1=\|\widetilde{\mathbf{p}}\|_2$,  $\gamma_1=\|\widetilde{\mathbf{q}}\|_2$;
\State  $\mathbf{p}_{2}=\widetilde{\mathbf{p}}/\beta_1$, $\mathbf{q}_{2}=\widetilde{\mathbf{q}}/\gamma_1$;
\State   $\widetilde{\mathbf{w}}_1=\mathbf{q}_1$, $\bm{ \nu'}_1=\bm{\alpha}_1$, $\widetilde{\delta}_2=\beta_1$, $\bm{\zeta}_0=0$, $\bm{\eta}_2=0$;
\State  $c_1=\dfrac{ |\bm{ \nu'}_1|}{\sqrt{| \bm{ \nu'}_1|^2+\gamma_1^2}}$,
   $\mathbf{s}_1=\dfrac{ \bm{ \bar{\nu}}'_1}{ \bm{ \nu'}_1} \dfrac{ \bm{ \gamma}_1}{\sqrt{| \bm{ \nu'}_1|^2+\gamma_1^2}}$, $\bm{\nu}_1=\dfrac{ \bm{ \nu'}_1}{ |\bm{ \nu'}_1|} \sqrt{| \bm{ \nu'}_1|^2+\gamma_1^2}$;
\State $\mathbf{w}_1=\widetilde{\mathbf{w}}_1c_1+\mathbf{q}_2\bar{\mathbf{s}}_1$,  $\widetilde{\mathbf{w}}_2=\widetilde{\mathbf{w}}_1\mathbf{s}_1-\mathbf{q}_2c_1 $;
\State $\bm{\zeta}_1=\bm{\nu}_1^{-1}\beta$, $\bm{\widetilde{\zeta}}_{1}=\bm{\zeta}_{1}/c_{1}$;
\State  $\mathbf{x}_{1}=\mathbf{x}_{0}
 +\widetilde{\mathbf{w}}_{1}\bm{\widetilde{\zeta}}_{1}$, $\widetilde{\mathbf{x}}_{1}=\mathbf{x}_0+\mathbf{w}_1\bm{\zeta}_1$;

\State  $\mathbf{r}_1=\mathbf{b}-\mathscr{A}\mathbf{x}_1$,
$m=1$;

\State  ${\bf while}~(\|\mathbf{r}_m\| > \beta\varepsilon)$
\State \quad $\alpha_{m+1}=\langle\mathbf{A} \mathbf{q}_{m+1},\mathbf{p}_{m+1}\rangle$;
\State \quad  $\widetilde{\mathbf{p}}=\mathbf{A} \mathbf{q}_{m+1}-\mathbf{p}_{m+1}\bm{\alpha}_{m+1}-\mathbf{p}_{m}\gamma_{m}$,
$\widetilde{\mathbf{q}}=\mathscr{A}^*\mathbf{p}_{m+1}-\mathbf{q}_{m+1}\bm{\alpha}_{m+1}^*
-\mathbf{q}_{m}\beta_{m}$;
\State \quad $\beta_{m+1}=\|\widetilde{\mathbf{p}}\|_2$, $\gamma_{m+1}=\|\widetilde{\mathbf{q}}\|_2$;

\State \quad \textbf{if} $\beta_{m+1}=0$ or $\gamma_{m+1}=0$
\State \quad\quad \quad \textbf{stop}
\State \quad  \textbf{else}
\State \quad\quad\quad   $\mathbf{p}_{m+2}=\widetilde{\mathbf{p}}/\beta_{m+1}$, $\mathbf{q}_{m+2}=\widetilde{\mathbf{q}}/\gamma_{m+1}$;
\State \quad  \textbf{end}
\State \quad $\bm{\delta}_{m+1}=\widetilde{\delta}_{m+1}c_m+\bm{\alpha}_{m+1}\bar{\mathbf{s}}_m$, $\bm{\nu'}_{m+1}=\widetilde{\delta}_{m+1}\mathbf{s}_m-\bm{\alpha}_{m+1}c_m$;
\State \quad  $\widetilde{\delta}_{m+2}=-\beta_{m+1}c_m$, $\bm{\eta}_{m+2}=\beta_{m+1}\bar{\mathbf{s}}_m$;

\State \quad  $c_{m+1}=\dfrac{ |\bm{ \nu'}_{m+1}|}{\sqrt{| \bm{ \nu'}_{m+1}|^2+\gamma_{m+1}^2}}$,
   $\mathbf{s}_{m+1}=\dfrac{ \bm{ \overline{\nu}}'_{m+1}}{ \bm{ \nu'}_{m+1}} \dfrac{ \bm{ \gamma}_{m+1}}{\sqrt{| \bm{ \nu'}_{m+1}|^2+\gamma_{m+1}^2}}$, $\bm{\nu}_{m+1}=\dfrac{ \bm{ \nu'}_{m+1}}{ |\bm{ \nu'}_{m+1}|} \sqrt{| \bm{ \nu'}_{m+1}|^2+\gamma_{m+1}^2}$;

\State \quad $\bm{\zeta}_{m+1}=-\bm{\nu}_{m+1}^{-1}[ \bm{\eta}_{m+1}\bm{\zeta}_{m-1}+\bm{\delta}_{m+1}\bm{\zeta}_{m}]$,
$\bm{\widetilde{\zeta}}_{m+1}=\bm{\zeta}_{m+1}/c_{m+1}$;

\State \quad $\mathbf{w}_{m+1}=\widetilde{\mathbf{w}}_{m+1}c_{m+1}+\mathbf{q}_{m+2}\bar{\mathbf{s}}_{m+1}$,  $\widetilde{\mathbf{w}}_{m+2}=\widetilde{\mathbf{w}}_{m+1}\mathbf{s}_{m+1}-\mathbf{q}_{m+2}c_{m+1}$;

\State \quad  $\mathbf{x}_{m+1}=\widetilde{\mathbf{x}}_{m}
 +\widetilde{\mathbf{w}}_{m+1}\bm{\widetilde{\zeta}}_{m+1}$,
 $\widetilde{\mathbf{x}}_{m+1} =\widetilde{\mathbf{x}}_{m}+\mathbf{w}_{m+1}\bm{\zeta}_{m+1}$;

\State \quad $\mathbf{r}_{m+1}=\mathbf{b}-\mathscr{A}\mathbf{x}_{m+1}$;
\State \quad $m=m+1$;
\State  ${\bf end}$
\end{algorithmic}
\end{algorithm}

\subsection{ The QNHERQR method}
Let $\mathbf{x}_0$ be an initial guess and  $\mathbf{p}_1=\mathbf{r}_0/\beta$, where $\mathbf{r}_0=\mathbf{b}-\mathscr{A}\mathbf{x}_0$ and $\beta=\|\mathbf{r}_0\|_2$. Choose an arbitrary unit vector $\mathbf{q}_1$ and run $m$ steps of the quaternion Saunders-Simon-Yip tridiagonalization procedure. We find an approximation solution $\mathbf{x}_m$ to (\ref{1.1}) from the affine subspace $\mathbf{x}_0+\mathrm{span}(\mathbf{q}_1,\mathbf{q}_2,\cdots,\mathbf{q}_m)$ which minimizes the residual vector norm $\|\mathbf{r}_m\|_2$, where $\mathbf{r}_m=\mathbf{b}-\mathscr{A}\mathbf{x}_m$. Let $\mathbf{x}_m$ be given by
\begin{equation}\label{xm1}
\mathbf{x}_m=\mathbf{x}_0+\mathscr{Q}_m\mathbf{y}_m.
\end{equation}
Direct calculations give
\begin{eqnarray}
% \nonumber to remove numbering (before each equation)
  \mathbf{r}_m&=&\mathbf{b}-\mathscr{A}\mathbf{x}_m= \mathbf{b}-\mathscr{A}(\mathbf{x}_0+\mathscr{Q}_m\mathbf{y}_m)=
   \mathbf{r}_0-\mathscr{A} \mathscr{Q}_m\mathbf{y}_m \nonumber\\
   &=&\beta\mathbf{p}_1-(\mathscr{P}_m\mathscr{T}_m+\beta_m\mathbf{p}_{m+1}\mathbf{e}_m^*)\mathbf{y}_m
   =\beta\mathbf{p}_1-\mathscr{P}_{m+1}\widetilde{\mathscr{T}}_m\mathbf{y}_m\nonumber\\
      &=&\mathscr{P}_{m+1}(\beta \mathbf{e}_1-\widetilde{\mathscr{T}}_m\mathbf{y}_m),
\end{eqnarray}
where
\begin{equation}\label{5.2}
\widetilde{\mathscr{T}}_m
=\left[\begin{array}{c}
\mathscr{T}_m\\
\beta_m\mathbf{e}_m^*
\end{array}\right]
\in\mathbb{Q}^{(m+1)\times m}.
\end{equation}
By the minimal residual norm condition and the orthogonality of the columns of $\mathscr{P}_{m+1}$,
we obtain the least squares problem
\begin{equation}\label{5.1}
\mathbf{y}_m=\arg\min\limits_{\mathbf{y}\in\mathbb{Q}^m}\{\|\beta\mathbf{e}_1
-\widetilde{\mathscr{T}}_m\mathbf{y}\|_2\},
\end{equation}
where $\beta=\|\mathbf{r}_0\|_2$.

The minimization problem (\ref{5.1}) is a least squares problem of a strictly tridiagonal quaternion coefficient matrix, which can be solved by the quaternion QR method. The quaternion QR factorization can be obtained with an update procedure to make $\mathbf{x}_m$ directly available at every step. We call the resulting method  QNHERQR.

From the special structure of $\widetilde{\mathscr{T}}_m$, we can compute $m$ quaternion Givens transformations
$\mathscr{G}_1,\mathscr{G}_2,\cdots,\mathscr{G}_m$ such that
\begin{equation}\label{5.3}
  \mathscr{G}_m\mathscr{G}_{m-1}\cdots\mathscr{G}_2\mathscr{G}_1\widetilde{\mathscr{T}}_m=
  \left[\begin{array}{c}
  \mathscr{R}_m\\
  \mathbf{0}
  \end{array}
  \right],
\end{equation}
where
\begin{eqnarray*}
% \nonumber to remove numbering (before each equation)
   && \mathscr{G}_i=
   \mathrm{diag}\left(\mathbf{I}_{i-1},\left[
  \begin{array}{cc}
 c_i&\mathbf{s}_i\\
  -\bar{\mathbf{s}}_i&c_i
  \end{array}
  \right],\mathbf{I}_{m-i}\right) \in\mathbb{Q}^{(m+1)\times (m+1)}, c_i^2+|\mathbf{s}_i|^2=1,\label{5.4}\\
  &&\mathscr{R}_m=
  \left[\begin{array}{ccccc}
  \bm{\sigma}_1&\bm{\delta}_1&\bm{\varepsilon}_1&&\\
  &\bm{\sigma}_2&\bm{\delta}_2&\ddots\\
  &&\ddots&\ddots&\bm{\varepsilon}_{m-2}\\
  &&&\bm{\sigma}_{m-1}&\bm{\delta}_{m-1}\\
  &&&&\bm{\sigma}_m\\
  \end{array}
  \right].\label{5.5}
\end{eqnarray*}
Notice that $\beta_i\neq0$ imply that $\bm{\sigma}_i\neq 0$ for $i=1,2,\cdots,m$ and $\mathscr{R}_m$ is invertible.

Let
\begin{equation}\label{5.6}
  \mathscr{G}=   \mathscr{G}_m   \mathscr{G}_{m-1}\cdots  \mathscr{G}_2   \mathscr{G}_1, ~~\left[\begin{array}{c}
\bm{t}_m\\ \bm{\rho}_m
\end{array}\right]=  \mathscr{G}(\beta \bm{e}_1),~~ \bm{t}_m=[\bm{\tau}_1,\bm{\tau}_2,\cdots,\bm{\tau}_m]^{T}.
\end{equation}
Obviously, $\mathscr{G}$ is  a $(m+1)\times (m+1)$ unitary quaternion matrix and
\begin{equation*}\label{5.7}
\bm{\rho}_m=(-1)^m\bar{\mathbf{s}}_m \cdots\bar{\bm{s}}_2 \bar{\mathbf{s}}_1\beta,~~\bm{\tau}_1=c_1\beta,~~\bm{\tau}_i=(-1)^{i-1} c_i \bar{\mathbf{s}}_{i-1}\cdots\bar{\mathbf{s}}_2\bar{\mathbf{s}}_1 \beta,~~ i=2,3,\cdots,m.
\end{equation*}
This implies that
\begin{equation*}
 \bm{\tau}_{m}\!=\!(-1)^{m-1} c_{m} \bar{\mathbf{s}}_{m-1}\cdots \bar{\mathbf{s}}_1 \beta\!=\!c_{m}\bm{\rho}_{m-1}, 
\bm{\rho}_{m}\!=\!(-1)^{m}\bar{\mathbf{s}}_{m}\bar{\bm{s}}_{m-1} \cdots\bar{\bm{s}}_2 \bar{\bm{s}}_1\beta\!=\!-\bar{\mathbf{s}}_{m}\bm{\rho}_{m-1}.
\end{equation*}
By the relations (\ref{5.3}) and (\ref{5.6}), we have
\begin{eqnarray*}
% \nonumber to remove numbering (before each equation)
  \|\beta\mathbf{e}_1\!-\!\widetilde{\mathscr{T}}_m\mathbf{y}\|_2^2\! =\! \|\mathscr{G}(\beta\mathbf{e}_1\!-\!\widetilde{\mathscr{T}}_m\mathbf{y})\|_2^2\!=\!
 \left\|\left[\begin{array}{c}
  \mathscr{R}_m\\
  \mathbf{0}
  \end{array}
  \right]\mathbf{y}-\left[\begin{array}{c}
\bm{t}_m\\ \bm{\rho}_m
\end{array}\right]\right\|_2^2 \!=\!\|\mathscr{R}_m\mathbf{y}\!-\!\mathbf{t}_m\|_2^2\!+\!|\bm{\rho}_m|^2.
\end{eqnarray*}
This gives an unique solution
 $\mathbf{y}_m=\mathscr{R}_m^{-1}\mathbf{t}_m$
of the least squares problem (\ref{5.1}) and
\begin{equation}\label{5.9}
   \|\mathbf{r}_m\|_2=\|\mathscr{P}_{m+1}(\beta \mathbf{e}_1-\widetilde{\mathscr{T}}_m\mathbf{y}_m)\|_2=\|\beta\mathbf{e}_1-\widetilde{\mathscr{T}}_m\mathbf{y}_m\|_2=|\bm{\rho}_m|.
\end{equation}
Denote $\mathscr{N}_m=\mathscr{Q}_m\mathscr{R}_m^{-1}=[\mathbf{n}_1, \mathbf{n}_2,\cdots,  \mathbf{n}_m]$. Thus the desired solution $\mathbf{x}_m$ satisfies
\begin{eqnarray*}\label{5.10}
\mathbf{x}_m&=&\mathbf{x}_0+\mathscr{Q}_m\mathbf{y}_m
=\mathbf{x}_0+\mathscr{Q}_m\mathscr{R}_m^{-1}\mathbf{y}_m
=\mathbf{x}_0+\mathscr{N}_m\mathbf{y}_m\\
&=&\mathbf{x}_0+[\mathscr{N}_{m-1},\mathbf{n}_{m}]\left[
\begin{array}{c}
\mathbf{t}_{m-1}\\
\bm{\tau}_m
\end{array}
\right]
=\mathbf{x}_{m-1}+\mathbf{n}_m\bm{\tau}_m.
\end{eqnarray*}

Next, we give the recursive formula of  $\mathbf{x}_m$. When the length of quaternion tridiagonalization decomposition is increased from $m$  to $m+1$, we have
\begin{displaymath}
\widetilde{\mathscr{T}}_{m+1}=\left[\begin{array}{c|c}
\widetilde{\mathscr{T}}_m&\widetilde{\mathbf{t}}_{m+1}\\
\hline
\mathbf{0}&\beta_{m+1}
\end{array}\right],\quad\widetilde{\mathbf{t}}_{m+1}=[0,\cdots,0,\gamma_m,\bm{\alpha}_{m+1}]^{\mathrm{T}},
\end{displaymath}
and then
\begin{equation*}\label{5.11}
\mathscr{R}_{m+1}=\left[\begin{array}{c|c}
\mathscr{R}_m& \mathbf{h}_{m+1}\\
\hline
\mathbf{0}&\bm{\sigma}_{m+1}
\end{array}\right],\quad \mathbf{h} _{m+1}=[0,\cdots,0,\bm{\varepsilon}_{m-1},\bm{\delta}_m]^{\mathrm{T}}.
\end{equation*}
Since
\begin{scriptsize}
\begin{eqnarray*}
% \nonumber to remove numbering (before each equation)
&& \!\!\!\!  \widetilde{\mathscr{T}}_{m+1}\!= \!\left[
                          \begin{array}{cccccc|c}
                           \!\! \bm{\alpha}_1\! \!&\!\! \gamma_1 \!\!&   &   &   &   &   \\
                           \!\! \beta_1\! \!&\!\! \bm{\alpha}_2\! \!&\!\! \gamma_2\! &   &   &   &   \\
                                & \ddots &\ddots & \ddots &   & & \\
   &  &\!\! \beta_{m-3}\! &\! \alpha_{m-2}\! &\! \gamma_{m-2}\! &  & \\
                              &   &  &\!\! \beta_{m-2}\! &\! \bm{\alpha}_{m-1}\! & \!\gamma_{m-1}\! &   \\
                              &   &   &   &\! \beta_{m-1}\! &\! \bm{\alpha}_{m}\! &\! \gamma_{m}\! \\
                             &   &   &   &   &\! \beta_{m}\! &\! \bm{\alpha}_{m+1}\! \\\hline
                              &   &   &   &  &   &\!  \beta_{m+1} \!\\
                          \end{array}
                         \right]\!\underrightarrow{\mathscr{G}_1}\!
\left[ \begin{array}{cccccc|c}
                         \!\!   \bm{\sigma}_1\! & \!\bm{\delta}_1\! \!&\!\! \bm{\varepsilon}_1 \!\! &   &   &   &   \\
                              \!\!\! &\! \!\bm{\widetilde{\alpha}}_2\!\!&\!\! \bm{\widehat{\gamma}}_2\! \!&   &   &   &   \\
                               \!\!\!   &  &\ddots & \ddots & \ddots  &  &\\
                  \!\!\!  &  &   & \!\!\bm{\alpha}_{m-2}\!\! &\! \!\gamma_{k-2}\!\! &  & \\
                              \!\!\!  &   &  &\!\! \beta_{m-2} \!&\! \bm{\alpha}_{m-1} & \gamma_{m-1} &   \\
                             \!\!\!   &   &   &   &\! \beta_{m-1}\!\! & \!\bm{\alpha}_{m}\! &\!\! \gamma_{m}\!\! \\
                             \!\!\!  &   &   &   &   &\! \beta_{m} \!&\! \bm{\alpha}_{m+1}\! \!\\\hline
                              \!\!\!  &   &   &   &  &   & \! \beta_{m+1}  \\
                          \end{array}
                       \right]
                      \!\underrightarrow{\mathscr{G}_{2}}\cdots
\end{eqnarray*}
\begin{eqnarray*}
&&\underrightarrow{\mathscr{G}_{m-2}}
\left[\begin{array}{cccccc|c}
                          \!\! \bm{\sigma}_1 & \bm{\delta}_1 & \bm{\varepsilon}_1  &   &   &   &   \\
                             \!\!\!\!  & \bm{\sigma}_2 & \bm{\delta}_2 &  \bm{\varepsilon}_2 &   &   &   \\
                              \!\!\!\!    &  &\ddots & \ddots & \ddots  &  &\\
                            \!\!\!\!  &  &   & \bm{\sigma}_{m-2} & \bm{\delta}_{m-2} & \bm{\varepsilon}_{m-2} & \\
                              \!\!\!\!  &   &  &    & \bm{\widetilde{\alpha}}_{m-1} & \bm{\widehat{\gamma}}_{m-1} &   \\
                              \!\!\!\!  &   &   &   & \beta_{m-1} & \bm{\alpha}_{m} & \gamma_{m} \\
                             \!\!\!\!  &   &   &   &   & \beta_{m} & \bm{\alpha}_{m+1} \\\hline
                            \!\!\!\!    &   &   &   &  &   &  \beta_{m+1} \\
                          \end{array}
                         \right]
                         \underrightarrow{\mathscr{G}_{m-1}}\!
\left[\begin{array}{cccccc|c}
                          \!\!    \bm{\sigma}_1\! \!&\! \bm{\delta}_1 \!\!& \!\bm{\varepsilon}_1 \!\! &   &   &   &   \\
                             \!\!\!\!  &\! \!\bm{\sigma}_2\!\! &\! \!\bm{\delta}_2\! & \! \bm{\varepsilon}_2\!\! &   &   &   \\
                               \!\!\!\!  &  &\ddots & \ddots & \ddots  & & \\
  \!\!\!\!  &  &   &\! \!\bm{\sigma}_{m-2}\! &\! \bm{\delta}_{m-2}\! &\! \bm{\varepsilon}_{m-2} \!\!& \\
    \!\!\!\!   &   &  &  &\!\! \bm{\sigma}_{m-1} \!&\! \bm{\delta}_{m-1}\!\! &\! \!\bm{\varepsilon}_{m-1}\!\!  \\
                             \!\!\!\!   &   &   &   &   &\! \bm{\widetilde{\alpha}}_{m}\! &\! \bm{\widehat{\gamma}}_{m}\! \\
                             \!\!\!\!  &   &   &   &   &\! \beta_{m} \!&\! \bm{\alpha}_{m+1}\! \\\hline
                             \!\!\!\!   &   &   &   &  &   & \! \beta_{m+1}\! \\
                          \end{array}
                        \right]                         
\end{eqnarray*}
\begin{eqnarray*}
&&~
                        \!\underrightarrow{\mathscr{G}_{m}}\!
\left[\!\!\begin{array}{cccccc|c}
                            \bm{\sigma}_1\! & \bm{\delta}_1\!\! & \!\!\bm{\varepsilon}_1 \!\! &   &   &   &   \\
                             &\! \!\bm{\sigma}_2 \!\!&\!\! \bm{\delta}_2\!\! &\! \! \bm{\varepsilon}_2\! \!&   &   &   \\
                                &  &\ddots & \ddots & \ddots  &  &\\
                           &  &   &\! \!\bm{\sigma}_{m-2}\! & \!\bm{\delta}_{m-2}\! &\! \bm{\varepsilon}_{m-2}\! & \\
                            &   &  &  &\! \!\bm{\sigma}_{m-1}\! & \!\bm{\delta}_{m-1}\! &\! \bm{\varepsilon}_{m-1}\! \! \\
                              &   &   &   &   &\! \bm{\sigma}_{m} \!&\! \bm{\delta}_{m} \!\\
                             &   &   &   &   &   &\! \bm{\widetilde{\alpha}}_{m+1}\! \\\hline
                              &   &   &   &  &   & \! \beta_{m+1}\! \\
                          \end{array}
                      \!\!  \right]
                      \!\underrightarrow{\mathscr{G}_{m+1}}\!
\left[\!\!\begin{array}{cccccc|c}
                            \bm{\sigma}_1\! & \bm{\delta}_1\!\! & \!\!\bm{\varepsilon}_1 \!\! &   &   &   &   \\
                             &\! \!\bm{\sigma}_2 \!\!&\!\! \bm{\delta}_2\!\! &\! \! \bm{\varepsilon}_2\! \!&   &   &   \\
                                &  &\ddots & \ddots & \ddots  &  &\\
                           &  &   &\! \!\bm{\sigma}_{m-2}\! & \!\bm{\delta}_{m-2}\! &\! \bm{\varepsilon}_{m-2}\! & \\
                            &   &  &  &\! \!\bm{\sigma}_{m-1}\! & \!\bm{\delta}_{m-1}\!&\! \bm{\varepsilon}_{m-1}\!  \\
                              &   &   &   &   &\! \bm{\sigma}_{m} \!&\! \bm{\delta}_{m} \!\\
                             &   &   &   &   &   &\! \bm{\sigma}_{m+1}\! \\\hline
                              &   &   &   &  &   & \! \!0\!\! \\
                          \end{array}
                      \!\!  \right],
\end{eqnarray*}
\end{scriptsize}
we have
\begin{equation*}
\left[
    \begin{array}{cc}
      c_{m-1} & \mathbf{s}_{m-1} \\
      -\bar{\mathbf{s}}_{m-1} & c_{m-1} \\
    \end{array}
  \right]\left[
    \begin{array}{c}
      0 \\
      \gamma_m \\
    \end{array}
  \right]=\left[
    \begin{array}{c}
      \bm{\varepsilon}_{m-1} \\
      \bm{\widehat{\gamma}}_m \\
    \end{array}
  \right], ~~\left[
    \begin{array}{cc}
      c_{m} & \mathbf{s}_{m} \\
      -\bar{\mathbf{s}}_{m} & c_{m} \\
    \end{array}
  \right]\left[
    \begin{array}{c}
      \bm{\widehat{\gamma}}_m \\
      \bm{\alpha}_{m+1} \\
    \end{array}
  \right]=\left[
    \begin{array}{c}
      \bm{\delta}_{m} \\
      \bm{\widetilde{\alpha}}_{m+1} \\
    \end{array}
  \right]
\end{equation*}
and
\begin{equation}\label{5.12}
\left[\begin{array}{cc}
c_{m+1}&\mathbf{s}_{m+1}\\
-\bar{\mathbf{s}}_{m+1}&c_{m+1}
\end{array}\right]\left[\begin{array}{c}
\bm{\widetilde{\alpha}}_{m+1}\\
\beta_{m+1}
\end{array}\right]=\left[\begin{array}{c}
\bm{\sigma}_{m+1}\\0
\end{array}\right].
\end{equation}
Thus
\begin{align*}%\label{a5.6.26}
&\bm{\varepsilon}_{m-1}=\mathbf{s}_{m-1}\gamma_m,\quad \bm{\widehat\gamma}_m=c_{m-1}\gamma_m,\nonumber\\
&\bm{\delta}_m=c_m\bm{\widehat{\gamma}}_m+\mathbf{s}_m\bm{\alpha}_{m+1},\quad\bm{\widetilde{\alpha}}_{m+1}
=-\bar{\mathbf{s}}_m\bm{\widehat\gamma}_m+c_m\bm{\alpha}_{m+1},
\end{align*}
According to (\ref{5.12}), we compute
\begin{equation*}\label{5.13}
  c_{m+1}=\dfrac{|\bm{\widetilde{\alpha}}_{m+1}|}{\sqrt{|\bm{\widetilde{\alpha}}_{m+1}|^2+\beta_{m+1}^2}},
  \mathbf{s}_{m+1}=\dfrac{\bm{\widetilde{\alpha}}_{m+1}}{|\bm{\widetilde{\alpha}}_{m+1}|}
  \dfrac{\beta_{m+1}}{\sqrt{|\bm{\widetilde{\alpha}}_{m+1}|^2+\beta_{m+1}^2}},
\end{equation*}
and
\begin{equation*}
  \bm{\sigma}_{m+1}=\dfrac{\bm{\widetilde{\alpha}}_{m+1}}{|\bm{\widetilde{\alpha}}_{m+1}|}
\sqrt{|\bm{\widetilde{\alpha}}_{m+1}|^2+\beta_{m+1}^2}.
\end{equation*}

Since $\mathscr{N}_m=\mathscr{Q}_m\mathscr{R}_m^{-1}$, we have $\mathscr{Q}_m=\mathscr{N}_m\mathscr{R}_m$. Direct calculations give
\begin{align*}
& \mathbf{n}_1\bm{\sigma}_1= \mathbf{q}_1,\\
& \mathbf{n}_1\bm{\delta}_1+ \mathbf{n}_2\bm{\sigma}_2= \mathbf{q}_2,\\
&  \mathbf{n}_{i-2}\bm{\varepsilon}_{i-2}+  \mathbf{n}_{i-1}\bm{\delta}_{i-1}+ \mathbf{n}_i\bm{\sigma}_i = \mathbf{q}_i,\,i=3,4,\cdots,m.
\end{align*}
This implies that
\begin{align}\label{4.15}
\left\{
  \begin{array}{l}
 \mathbf{n}_1= \mathbf{q}_1 \bm{\sigma}_1^{-1},\\
  \mathbf{n}_2=[  \mathbf{q}_2- \mathbf{n}_1\bm{\delta}_1] \bm{\sigma}_2^{-1},\\
  \mathbf{n}_i=[ \mathbf{q}_i-  \mathbf{n}_{i-2}\bm{\varepsilon}_{i-2}- \mathbf{n}_{i-1}\bm{\delta}_{i-1} ]\bm{\sigma}_i^{-1},~~ i=3,4,\cdots,m.
  \end{array}
\right.
\end{align}
By the above discussion, the QNHERQR method is summarized as in Algorithm \ref{alg5.1}.

\begin{algorithm}[!ht]
\caption{The QNHERQR method}
Input the coefficient matrix $\mathscr{A}\in\mathbb{Q}^{n\times n}$ and the right-hand side vector $\mathbf{b}\in\mathbb{Q}^n$. Choose an initial vector $\mathbf{x}_0\in\mathbb{Q}^n$ and the tolerance error $\varepsilon>0$.
\label{alg5.1}
\begin{algorithmic}[1]
\State  Compute $\mathbf{r}_0=\mathbf{b}-\mathscr{A}\mathbf{x}_0$, $\beta=\|\mathbf{r}_0\|_2\neq0$, and $\mathbf{q}_1=\mathbf{r}_0/\beta$;
\State  Choose $\mathbf{p}_1$ such that $\|\mathbf{p}_1\|_2=1$, e.g., $\mathbf{p}_1=\mathbf{q}_1$;
\State    $\bm{\alpha}_{1}=\langle\mathbf{A} \mathbf{q}_1,\mathbf{p}_1\rangle$,
$\widetilde{\mathbf{p}}=\mathbf{A} \mathbf{q}_1-\mathbf{p}_{1}\bm{\alpha}_{1}$,
         $\widetilde{\mathbf{q}}=\mathscr{A}^*\mathbf{p}_1-\mathbf{q}_1\bm{\alpha}_1^*$;
\State   $\beta_1=\|\widetilde{\mathbf{p}}\|_2$,  $\gamma_1=\|\widetilde{\mathbf{q}}\|_2$;

\State   $\mathbf{p}_{2}=\widetilde{\mathbf{p}}/\beta_1$, $\mathbf{q}_{2}=\widetilde{\mathbf{q}}/\gamma_1$;
\State  $c_{ 1}=\dfrac{|\bm{\alpha}_{ 1}|}{\sqrt{|\bm{\alpha}_{ 1}|^2+\beta_{ 1}^2}}$,
  $\mathbf{s}_{ 1}=\dfrac{\bm{\alpha}_{ 1}}{|\bm{\alpha}_{ 1}|}
  \dfrac{\beta_{ 1}}{\sqrt{|\bm{\alpha}_{ 1}|^2+\beta_{1}^2}}$,
  $\bm{\sigma}_{ 1}=\dfrac{\bm{\alpha}_{ 1}}{|\bm{\alpha}_{ 1}|}
\sqrt{|\bm{\alpha}_{ 1}|^2+\beta_{ 1}^2}$;

\State  $\mathbf{n}_1= \mathbf{q}_1 \bm{\sigma}_1^{-1}$;

\State  $\bm{\tau}_1=c_1\beta$, $\bm{\rho}_1=-\bar{\mathbf{s}}_1\beta$;
\State  $c_0=1$, $s_0=0$, $\mathbf{x}_1=\mathbf{x}_0+\mathbf{n}_1\bm{\tau}_1$;
\State   $m=1$;

\State   ${\bf while}~(|\bm{\rho}_m|>\beta \varepsilon)$
\State  \quad $\bm{\alpha}_{m+1}=\langle\mathbf{A} \mathbf{q}_{m+1},\mathbf{p}_{m+1}\rangle$;
\State  \quad  $\widetilde{\mathbf{p}}=\mathbf{A} \mathbf{q}_{m+1}-\mathbf{p}_{m+1}\bm{\alpha}_{m+1}-\mathbf{p}_{m}\gamma_{m}$,
$\widetilde{\mathbf{q}}=\mathscr{A}^*\mathbf{p}_{m+1}-\mathbf{q}_{m+1}\bm{\alpha}_{m+1}^*
-\mathbf{q}_{m}\beta_{m}$;
\State  \quad $\beta_{m+1}=\|\widetilde{\mathbf{p}}\|_2$, $\gamma_{m+1}=\|\widetilde{\mathbf{q}}\|_2$;

\State  \quad \textbf{if} $\beta_{m+1}=0$ or $\gamma_{m+1}=0$
\State  \quad\quad \quad \textbf{stop}
\State  \quad  \textbf{else}
\State  \quad\quad\quad   $\mathbf{p}_{m+2}=\widetilde{\mathbf{p}}/\beta_{m+1}$, $\mathbf{q}_{m+2}=\widetilde{\mathbf{q}}/\gamma_{m+1}$;
\State  \quad  \textbf{end}
\State  \quad $\bm{\varepsilon}_{m-1}=\mathbf{s}_{m-1}\gamma_m$, $\bm{\widehat\gamma}_m=c_{m-1}\gamma_m$;
\State  \quad $\bm{\delta}_m=c_m\bm{\widehat{\gamma}}_m+\mathbf{s}_m\bm{\alpha}_{m+1}$, $\bm{\widetilde{\alpha}}_{m+1}
=-\bar{\mathbf{s}}_m\bm{\widehat\gamma}_m+c_m\bm{\alpha}_{m+1}$;

\State  \quad  \textbf{if} $\bm{\widetilde{\alpha}}_{m+1}\neq\mathbf{0}$
\State  \quad  \quad $c_{m+1}\!=\!\dfrac{|\bm{\widetilde{\alpha}}_{m+1}|}{\sqrt{|\bm{\widetilde{\alpha}}_{m+1}|^2\!+\!\beta_{m+1}^2}}$,
  $\mathbf{s}_{m+1}\!=\!\dfrac{\bm{\widetilde{\alpha}}_{m+1}}{|\bm{\widetilde{\alpha}}_{m+1}|}
  \dfrac{\beta_{m+1}}{\sqrt{|\bm{\widetilde{\alpha}}_{m+1}|^2\!+\!\beta_{m+1}^2}}$,
$\bm{\sigma}_{m+1}\!=\!\dfrac{\bm{\widetilde{\alpha}}_{m+1}}{|\bm{\widetilde{\alpha}}_{m+1}|}
\sqrt{|\bm{\widetilde{\alpha}}_{m+1}|^2\!+\!\beta_{m+1}^2}$;

\State  \quad  \textbf{else}
\State  \quad  \quad  $c_{m+1}=0$, $\mathbf{s}_{m+1}=1$, $\bm{\sigma}_{m+1}=\beta_{m+1}$;

\State  \quad  \textbf{end}

\State  \quad $\bm{\tau}_{m+1} =c_{m+1}\bm{\rho}_{m}$,
$\bm{\rho}_{m+1} =-\bar{\mathbf{s}}_{m+1}\bm{\rho}_{m}$;

\State  \quad $\mathbf{n}_{m+1}=[ \mathbf{q}_{m+1}-  \mathbf{n}_{m-1}\bm{\varepsilon}_{m-1}- \mathbf{n}_{m}\bm{\delta}_{m} ]\bm{\sigma}_{m+1}^{-1}$;

\State  \quad $\mathbf{x}_{m+1}=\mathbf{x}_{m}+\mathbf{n}_{m+1}\bm{\tau}_{m+1}$;

\State  \quad $m=m+1$;

\State  ${\bf end}$
\end{algorithmic}
\end{algorithm}

\subsection{Breakdown of QNHERLQ and QNHERQR}
The possibility of a brea-\\kdown in the QNHERLQ and QNHERQR  algorithms (Algorithms \ref{alg5.2} and \ref{alg5.1}) is in the case when $\mathbf{q}_{m+1}=\mathbf{0}$ (or $\mathbf{p}_{m+1}=\mathbf{0}$), i.e.,  $\beta_m=0$ (or $\gamma_m=0$) at a
given step $m$ in the quaternion Saunders-Simon-Yip tridiagonalization.  In this situation, the algorithm stops because the next step of the tridiagonalization procedure can not go on.  In this case, we get the exact solution of the corresponding quaternion linear systems.

\begin{theorem}\label{break1}
If $\mathbf{p}_{m+1}=\mathbf{0}$, then {\rm QNHERLQ} and {\rm QNHERQR} break down at step $m$. In this case, $\mathbf{x}_m$ is the exact solution of the quaternion linear systems {\rm(\ref{1.1})}.
\end{theorem}
\begin{proof}
If $\mathbf{p}_{m+1}=\mathbf{0}$, then $\beta_m=0$ and the last row of $\widetilde{\mathscr{T}}_m$ is zero. Hence the least squares solution $\mathbf{y}_m$ of the quaternion least squares problem (\ref{5.1}) is just the solution of the quaternion linear systems $\mathscr{T}_m\mathbf{y}=\beta\mathbf{e}_1$ computed in (\ref{5.16}). Hence, by the relation (\ref{ree}), it follows that $\|\mathbf{r}_m\|_2=0$, i.e., $\mathbf{x}_m$ is the exact solution.
\end{proof}

Notice that both QNHERLQ and QNHERQR can simultaneously approximate the solution of the following quaternion linear systems
\begin{equation}\label{trans1}
 \mathscr{A}^*\mathbf{z}=\mathbf{c}.
\end{equation}
In fact,  for a given initial vector $\mathbf{z}_0$, we let $\gamma=\|\mathbf{c}-\mathscr{A}^*\mathbf{z}_0\|_2$ and choose $\mathbf{q}_1=(\mathbf{c}-\mathscr{A}^*\mathbf{z}_0)/\gamma$ instead of an arbitrary $\mathbf{q}_1$, then the approximation solution of (\ref{trans1}) can be generated by
\begin{equation}\label{trans2}
 \mathbf{z}_m=\mathbf{z}_0+\mathscr{P}_m\mathbf{h}_m,
\end{equation}
where $\mathbf{h}_m$ is the solution of the quaternion linear systems
\begin{equation}\label{trans4}
\mathscr{T}_m^*\mathbf{h}=\gamma\mathbf{e}_1
\end{equation}
 or the quaternion least squares problem, i.e.,
\begin{equation}\label{trans3}
  \mathbf{h}_m=\arg\min\|\gamma \mathbf{e}_1-\widetilde{\mathscr{S}}_m \mathbf{h}\|_2~~\text{with}~~
  \widetilde{\mathscr{S}}_m=
  \left[\begin{array}{c}
  \mathscr{T}_m^*\\
  \gamma_m\mathbf{e}_m^*
  \end{array}\right].
\end{equation}

Similarly, we obtain the following result.
\begin{theorem}\label{break2}
If $\mathbf{q}_{m+1}=\mathbf{0}$, then {\rm QNHERLQ} and {\rm QNHERQR} break down at step $m$. In this case, $\mathbf{z}_m$ is the exact solution to the quaternion linear systems {\rm(\ref{trans1})}.
\end{theorem}
\begin{proof}
If $\mathbf{q}_{m+1}=\mathbf{0}$, then $\gamma_m=0$ and the last row of $\widetilde{\mathscr{S}}_m$ is zero. It follows that the least squares solution $\mathbf{h}_m$ of the quaternion least squares problem (\ref{trans3}) is just the solution of the quaternion linear systems $\mathscr{T}_m^*\mathbf{h}=\gamma\mathbf{e}_1$ computed in (\ref{trans4}). Hence    \begin{eqnarray*}
% \nonumber to remove numbering (before each equation)
\|\mathbf{c}-\mathscr{A}^*\mathbf{z}_m \|_2^2&=&\|\mathscr{Q}_{m+1}(\gamma \mathbf{e}_1-\widetilde{\mathscr{S}}_m\mathbf{h}_m) \|_2^2
=\|\gamma \mathbf{e}_1-\widetilde{\mathscr{S}}_m\mathbf{h}_m \|_2^2\\
&=&\|\gamma \mathbf{e}_1- \mathscr{T} _m^*\mathbf{h}_m \|_2^2+|\gamma_m|^2=|\gamma_m|^2=0,
\end{eqnarray*}
which shows that $\mathbf{z}_m$ is the exact solution of the quaternion linear systems (\ref{trans1}).
\end{proof}

\subsection{The choice of the initial vector}
When the initial vector $\mathbf{q}_1$ is selected randomly, the quaternion linear systems (\ref{trans1}) may be solved before the original quaternion linear systems (\ref{1.1}), and thus an unwanted breakdown would occur. However, similar to the quaternion Lanczos tridiagonalization procedure for Hermitian matrices,  an exactly zero $\beta_{m}=0$ is unlikely to occur in practice. Therefore, Theorems \ref{break1} and \ref{break2} are mainly of theoretical interest.
And, from Theorems \ref{break1} and \ref{break2}, we know that the simplest choice of the initial vector is to take $\mathbf{q}_1=\mathbf{p}_1=\mathbf{r}_0/\|\mathbf{r}_0\|_2$ or $\mathbf{q}_1=\mathbf{\overline{r}}_0/\|\mathbf{r}_0\|_2$.

In the following, we study how to select $\mathbf{q}_1$ for obtaining the desirable properties of QNHERLQ and QNHERQR. Since $\mathscr{T}_m$ is unique determined by $\mathbf{p}_1$, $\mathbf{q}_1$ and $\mathscr{A}$, it is promising that $\mathscr{T}_m$  is a real symmetric tridiagonal matrix by choosing a special $\mathbf{q}_1$.

\begin{theorem}\label{initial}
If $\mathbf{q}_1=\sqrt{\mathscr{A}^*\mathscr{A}}\mathscr{A}^{-1}\mathbf{p}_1$, then   a real symmetric tridiagonal matrix $\mathscr{T}_m$ is obtained by the quaternion Saunders-Simon-Yip tridiagonalization algorithm.
\end{theorem}
\begin{proof}
Let the QSVD decomposition of $\mathscr{A}$ be given by
\begin{equation}\label{qsvd1}
 \mathscr{A}=\mathscr{U}\Sigma\mathscr{V}^*,
\end{equation}
where $\Sigma$ is   real diagonal   and $\mathscr{U}$ and $\mathscr{V}$ are unitary. Then $\sqrt{\mathscr{A}^*\mathscr{A}}=\mathscr{V}\Sigma\mathscr{V}^*$ and $\sqrt{\mathscr{A}\mathscr{A}^*}=\mathscr{U}\Sigma\mathscr{U}^*$. Consider the tridiagonalization of $\sqrt{\mathscr{A}^*\mathscr{A}}$ by the quaternion Lanczos algorithm  \cite{jn2021}  with the initial vector $\mathbf{q}_1$, we obtain
\begin{equation}\label{4.43}
\sqrt{\mathscr{A}^*\mathscr{A}}\mathscr{Q}_m=\mathscr{Q}_m\mathscr{T}_m+\beta_m\mathbf{q}_{m+1}\mathbf{e}_m^*,
\end{equation}
where $\mathscr{T}_m$ is a real symmetric tridiagonal matrix. Then
\begin{eqnarray*}
% \nonumber to remove numbering (before each equation)
\mathscr{V}\Sigma\mathscr{V}^*\mathscr{Q}_m=\mathscr{Q}_m\mathscr{T}_m+\beta_m\mathbf{q}_{m+1}\mathbf{e}_m^*.
\end{eqnarray*}
Multiplying both sides of the above equation by $\mathscr{U}\mathscr{V}^*$ from the left yields
\begin{equation*}
  \mathscr{U}\Sigma\mathscr{V}^*\mathscr{Q}_m=\mathscr{U}\mathscr{V}^*\mathscr{Q}_m\mathscr{T}_m
  +\beta_m\mathscr{U}\mathscr{V}^*\mathbf{q}_{m+1}\mathbf{e}_m^*.
\end{equation*}
Denote $\mathbf{p}_j=\mathscr{U}\mathscr{V}^*\mathbf{q}_j$ for $j=1,2,\cdots,m+1$ and $\mathscr{P}_m=[\mathbf{p}_1,\cdots,\mathbf{p}_m]$. Then
\begin{equation*}
 \mathscr{A} \mathscr{Q}_m=\mathscr{P}_m \mathscr{T}_m +\beta_m\mathbf{p}_{m+1}\mathbf{e}_m^*,
\end{equation*}
which is precisely (\ref{eq10_1}) with a real symmetric matrix $\mathscr{T}_m$. By the uniqueness of the tridiagonalization procedure, we obtain a real symmetric tridiagonal matrix $\mathscr{T}_m$ for   an arbitrary $\mathbf{p}_1$ and
\begin{equation*}
\mathbf{q}_1=\mathscr{V}\mathscr{U}^*\mathbf{p}_1=(\mathscr{V}\Sigma\mathscr{V}^*)
(\mathscr{V}\Sigma^{-1}\mathscr{U}^*)\mathbf{p}_1=\sqrt{\mathscr{A}^*\mathscr{A}}\mathscr{A}^{-1}\mathbf{p}_1.
\end{equation*}
The proof is completed.
\end{proof}

Since the computation of  $\sqrt{\mathscr{A}^*\mathscr{A}}\mathscr{A}^{-1}\mathbf{p}_1$ is more difficult than the solution of the original quaternion linear systems,  Theorem \ref{initial}  has little practical value. However, it shows that an approximation to $\sqrt{\mathscr{A}^*\mathscr{A}}\mathscr{A}^{-1}\mathbf{p}_1$ is a good  selection of $\mathbf{q}_1$. In addition, we see that QNHERLQ and QNHERQR contain at least implicitly the natural extension of the conjugate gradient method to the non-Hermitian case.

\section{Comparison to other approaches}
In this section, we first show that our proposed QNHERLQ and QNHERQR are not Krylov subspace methods, unlike QGMRES \cite{jn2021}, QFOM \cite{lw2023}, QBiCG \cite{lw2024} and QQMR \cite{lwz2024}, then we prove that the convergence of  QNHERLQ and QNHERQR depends on the singular values of $\mathscr{A}$. Finally, we compare the proposed methods with the corresponding real versions and the QQMR method \cite{lwz2024}.

In the following, we always assume that $\mathbf{p}_1=\mathbf{q}_1=\mathbf{r}_0/\beta$, where $\beta=\|\mathbf{r}_0\|_2$. Then the following proposition gives a representation of the subspace $\mathscr{K}_m=\mathrm{span}(\mathbf{q}_1$, $\cdots$, $\mathbf{q}_m)$.

\begin{proposition}\label{prop6.1}
Let $\mathscr{K}_m=\mathrm{span}(\mathbf{q}_1,\cdots,\mathbf{q}_m)$. Then
\begin{eqnarray*}
% \nonumber to remove numbering (before each equation)
  \mathscr{K}_m &=&  \left\{\begin{array}{ll}
  \mathrm{span}(\mathbf{r}_0,\mathscr{A}^*\mathbf{r}_0, \mathscr{A}^*\mathscr{A}\mathbf{r}_0,\mathscr{A}^*\mathscr{A}\mathscr{A}^*\mathbf{r}_0,
  \cdots,
(\mathscr{A}^*\mathscr{A})^{(m-2)/2}\mathscr{A}^*\mathbf{r}_0),&\text{if}~k~\text{is even}\\
 \mathrm{span}(\mathbf{r}_0,\mathscr{A}^*\mathbf{r}_0, \mathscr{A}^*\mathscr{A}\mathbf{r}_0,\mathscr{A}^*\mathscr{A}\mathscr{A}^*\mathbf{r}_0,
  \cdots,
(\mathscr{A}^*\mathscr{A})^{(m-1)/2}\mathbf{r}_0),&\text{if}~k~\text{is odd}\\
  \end{array}\right.
  \\
&=&L_{\lfloor(m-1)/2\rfloor}( \mathscr{A}^*\mathscr{A},\mathbf{r}_0)+
L_{\lfloor(m-2)/2\rfloor}( \mathscr{A}^*\mathscr{A},\mathscr{A}^*\mathbf{r}_0),
\end{eqnarray*}
where $\lfloor a\rfloor$ represents the largest integer less or equal to $a$.
\end{proposition}
\begin{proof}
The results can be followed directly from Algorithms \ref{alg5.2} and \ref{alg5.1}.
\end{proof}

Obviously $\mathscr{K}_m$ is not a Krylov subspace defined by $\mathcal{K}_m(\mathscr{A},\mathbf{v})=\mathrm{span}(\mathbf{v}$, $\mathscr{A}\mathbf{v}$, $\cdots$,
  $\mathscr{A}^{m-1}\mathbf{v})$, but a union of two Krylov subspaces $\mathcal{K}_{\lfloor(m-1)/2\rfloor}( \mathscr{A}^*\mathscr{A}$, $\mathbf{r}_0)$ and
$\mathcal{K}_{\lfloor(m-2)/2\rfloor}$ $( \mathscr{A}^*\mathscr{A}$, $\mathscr{A}^*\mathbf{r}_0)$.  The proposed QNHERLQ and QNHERQR may find an approximation to the quaternion linear systems (\ref{1.1}) from the subspace $\mathscr{K}_m$, and thus are not  the Krylov subspace method.

From Proposition \ref{prop6.1}, it is known that the singular values of $\mathscr{A}$ play an important role in the convergence analysis of our proposed algorithms, while the spectrum of $\mathscr{A}$ take over this part for the QGMRES \cite{jn2021}.

In fact, let the QSVD decomposition of   $\mathscr{A}\in\mathbb{Q}^{n\times n}$ be given by
\begin{equation*}
 \mathscr{A}=\mathscr{U}\Sigma\mathscr{V}^*,
\end{equation*}
where $\mathscr{U}\in\mathbb{Q}^{n\times n}$ and $\mathscr{V}\in\mathbb{Q}^{n\times n}$ are unitary quaternion matrices and $\Sigma=\mathrm{diag}(\sigma_1$, $\cdots$, $\sigma_n)\in\mathbb{R}^{n\times n}$ is a diagonal matrix with $\sigma_i\geq0$ for $i=1,2,\cdots,n $. Since $\mathbf{x}_m\in\mathbf{x}_0+\mathscr{K}_m$ in QNHERLQ and QNHERQR, it follows from Proposition \ref{prop6.1} that
\begin{eqnarray*}
  \|\mathbf{r}_m\|_2&=&\|\mathbf{r}_0-L_{\lfloor(m-1)/2\rfloor}( \mathscr{A} \mathscr{A}^*,\mathscr{A} \mathbf{r}_0)-
L_{\lfloor(m-2)/2\rfloor}( \mathscr{A}\mathscr{A}^*, \mathscr{A}\mathscr{A}^*\mathbf{r}_0)\|_2\\
&=&\|\mathbf{r}_0-\mathscr{A} \mathbf{r}_0\mathbf{a}_0-  (\mathscr{A}\mathscr{A}^*)\mathscr{A} \mathbf{r}_0\mathbf{a}_1
-(\mathscr{A}\mathscr{A}^*)^{\lfloor(m-1)/2\rfloor}\mathscr{A} \mathbf{r}_0 \mathbf{a}_{\lfloor(m-1)/2\rfloor}\\
&&-
 \mathscr{A}\mathscr{A}^*\mathbf{r}_0\mathbf{b}_0-
 (\mathscr{A}\mathscr{A}^*)\mathscr{A}\mathscr{A}^*\mathbf{r}_0\mathbf{b}_1
 -(\mathscr{A}\mathscr{A}^*)^{\lfloor m/2-1\rfloor}\mathscr{A}\mathscr{A}^*\mathbf{r}_0\mathbf{b}_{\lfloor m/2-1\rfloor}\|_2\\
 &=&\|\mathscr{U}[\mathscr{U}^*\mathbf{r}_0-\Sigma L_{\lfloor(m-1)/2\rfloor}( \Sigma^2, \mathscr{V}^*\mathbf{r}_0)-
\Sigma^2L_{\lfloor(m-2)/2\rfloor}(\Sigma^2,\mathscr{U}^*\mathbf{r}_0)]\|_2\\
&=&\| \mathscr{U}^*\mathbf{r}_0-\Sigma L_{\lfloor(m-1)/2\rfloor}( \Sigma^2, \mathscr{V}^*\mathbf{r}_0)-
\Sigma^2L_{\lfloor(m-2)/2\rfloor}(\Sigma^2,\mathscr{U}^*\mathbf{r}_0) \|_2.
\end{eqnarray*}
If $m$ is even, then
\begin{eqnarray}
  \|\mathbf{r}_m\|_2&\leq &
   \|\mathbf{r}_0\|_2[\|\Sigma|\mathbf{a}_0|+\Sigma^3|\mathbf{a}_1|+\cdots+
\Sigma^{m-1}|\mathbf{a}_{m/2-1}|\|_2\nonumber\\
&&+
 \|\mathbf{I}+\Sigma^2|\mathbf{b}_0|+\Sigma^4|\mathbf{b}_1|+\cdots+
\Sigma^{m}|\mathbf{b}_{m/2-1}|\|_2]\nonumber\\
&\leq &\max\limits_{1\leq i\leq n}(1+\sigma_i|\mathbf{a}_0|+\sigma_i^2|\mathbf{b}_0|+\sigma_i^3|\mathbf{a}_1|
+\sigma_i^4|\mathbf{b}_1|+\cdots\nonumber\\
&&+
\sigma_i^{m-1}|\mathbf{a}_{m/2-1}|+\sigma_i^{m}|\mathbf{b}_{m/2-1}|)  \|\mathbf{r}_0\|_2.\label{6.1}
\end{eqnarray}
If $m$ is odd, then
\begin{eqnarray}
  \|\mathbf{r}_m\|_2&\leq&
   \|\mathbf{r}_0\|_2[\|\Sigma|\mathbf{a}_0|+\Sigma^3|\mathbf{a}_1| +\cdots+
\Sigma^{m}|\mathbf{a}_{(m-1)/2}|\|_2\nonumber\\
&&+\|\mathbf{I}+\Sigma^2|\mathbf{b}_0|+\Sigma^4|\mathbf{b}_1|+\cdots+
\Sigma^{m-1}|\mathbf{b}_{(m-3)/2}|\|_2]\nonumber\\
&\leq &\max\limits_{1\leq i\leq n}(1+\sigma_i|\mathbf{a}_0|+\sigma_i^2|\mathbf{b}_0|+\sigma_i^3|\mathbf{a}_1|
+\sigma_i^4|\mathbf{b}_1|+\cdots\nonumber\\
&&+
\sigma_i^{m-1}|\mathbf{b}_{(m-3)/2}|+\sigma_i^{m} |\mathbf{a}_{(m-1)/2}|)  \|\mathbf{r}_0\|_2.\label{6.2}
\end{eqnarray}

Let $\mathscr{A}=A_0+A_1\mathbf{i}+A_2\mathbf{j}+A_3\mathbf{k}$, $\mathbf{x}=x_0+x_1\mathbf{i}+x_2\mathbf{j}+x_3\mathbf{k}$ and $\mathbf{b}=b_0+b_1\mathbf{i}+b_2\mathbf{j}+b_3\mathbf{k}$. It is obvious that the quaternion linear systems (\ref{1.1}) can be equivalently rewritten as the following real linear systems
\begin{equation}\label{6.3}
\Upsilon_\mathscr{A}\Upsilon_\mathbf{x}^c=\Upsilon_\mathbf{b}^c,
\end{equation}
where
\begin{equation*}
\Upsilon_\mathscr{A}=\left[
  \begin{array}{cccc}
A_0&-A_1&-A_2&-A_3\\
  A_1&A_0&-A_3&A_2\\
  A_2&A_3&A_0&-A_1\\
 A_3&-A_2&A_1&A_0
  \end{array}
  \right],
  \Upsilon_\mathbf{x}^c=\left[\begin{array}{c}
 x_0\\
 x_1\\
 x_2\\
 x_3
 \end{array}\right],
 \Upsilon_\mathbf{b}^c=\left[\begin{array}{c}
b_0\\
 b_1\\
 b_2\\
 b_3
 \end{array}\right].
\end{equation*}
Denote $\mathbf{x}_m=x_{m,0}+x_{m,1}\mathbf{i}+x_{m,2}\mathbf{j}+x_{m,3}\mathbf{k}$, $\mathbf{x}_0=x_{0,0}+x_{0,1}\mathbf{i}+x_{0,2}\mathbf{j}+x_{0,3}\mathbf{k}$  and $\mathbf{r}_0=r_{0}+r_{1}\mathbf{i}+r_{2}\mathbf{j}+r_{3}\mathbf{k}$. Then the CG-type %{\rm (USYMLQ} or {\rm USYMQR)}
 iteration  for the real linear systems (\ref{6.3}) is given by
\begin{equation}\label{re2}
\left[\begin{array}{c}
 x_{m,0}\\
 x_{m,1}\\
 x_{m,2}\\
 x_{m,3}
 \end{array}\right]=    \left[\begin{array}{c}
 x_{0,0}\\
 x_{0,1}\\
 x_{0,2}\\
 x_{0,3}
 \end{array}\right] +P(\Upsilon_\mathscr{A}^T\Upsilon_\mathscr{A})\left[\begin{array}{c}
 \!\!r_0\!\! \\
 \!\!r_1\!\! \\
 \!\!r_2\!\! \\
\!\!r_3\!\!
 \end{array}\right]+
Q( \Upsilon_\mathscr{A}^T\Upsilon_\mathscr{A}) \Upsilon_\mathscr{A}^T\left[\begin{array}{c}
 \!\!r_0\!\! \\
 \!\!r_1\!\! \\
 \!\!r_2\!\! \\
\!\!r_3\!\!
 \end{array}\right],
\end{equation}
where $P,Q$ are real coefficient polynomials of degree $\lfloor(m-1)/2\rfloor$ and $\lfloor(m-2)/2\rfloor$, respectively.

\begin{proposition}\label{prop6.2}
The {\rm QCG}-type {\rm (QNHERLQ} or {\rm QNHERQR)} iteration
\begin{equation}\label{re4}
 \mathbf{x}_m=\mathbf{x}_0+ L_{\lfloor(m-1)/2\rfloor}( \mathscr{A}^*\mathscr{A},\mathbf{r}_0)+
L_{\lfloor(m-2)/2\rfloor}( \mathscr{A}^*\mathscr{A},\mathscr{A}^*\mathbf{r}_0)
\end{equation}
is equivalent to
\begin{eqnarray}
% \nonumber to remove numbering (before each equation)
  \hspace{-5mm} \left[\begin{array}{c}
  \!\!x_{m,0}\!\! \\
 \!\!x_{m,1}\!\! \\
 \!\! x_{m,2}\!\! \\
  \!\!x_{m,3}\!\!
 \end{array}\right]&=&    \left[\begin{array}{c}
  \!\!x_{0,0}\!\! \\
  \!\!x_{0,1}\!\! \\
 \!\!x_{0,2}\!\! \\
 \!\!x_{0,3}\!\!
 \end{array}\right] +P_1(\Upsilon_\mathscr{A}^T\Upsilon_\mathscr{A})\!\!\left[\begin{array}{c}
 \!\!r_0\!\! \\
 \!\!r_1\!\! \\
 \!\!r_2\!\! \\
\!\!r_3\!\!
 \end{array}\right]
 +P_2(\Upsilon_\mathscr{A}^T\Upsilon_\mathscr{A})\!\!\left[\begin{array}{c}
 \!\!-r_1\!\!  \\
  \!\!r_0\!\!  \\
 \!\!r_3\!\!  \\
 \!\!-r_2\!\!
 \end{array}\right]
  +P_3(\Upsilon_\mathscr{A}^T\Upsilon_\mathscr{A})\!\!\left[\begin{array}{c}
 \!\!-r_2\!\!  \\
 \!\!-r_3\!\!  \\
 \!\!r_0\!\!  \\
 \!\!r_1\!\!
 \end{array}\right]\nonumber\\
 && +P_4(\Upsilon_\mathscr{A}^T\Upsilon_\mathscr{A})\!\!\left[\begin{array}{c}
 \!\!-r_3\!\!  \\
 \!\!r_2\!\!  \\
 \!\!-r_1\!\!  \\
 \!\!r_0\!\!
 \end{array}\right]
   +Q_1(\Upsilon_\mathscr{A}^T\Upsilon_\mathscr{A})\Upsilon_\mathscr{A}^T\!\!\left[\begin{array}{c}
 \!\!r_0\!\! \\
 \!\!r_1\!\! \\
 \!\!r_2\!\! \\
\!\!r_3\!\!
 \end{array}\right]
 \!+\!Q_2(\Upsilon_\mathscr{A}^T\Upsilon_\mathscr{A})\Upsilon_\mathscr{A}^T\!\!\left[\begin{array}{c}
  \!\!-r_1\!\!  \\
  \!\!r_0\!\!  \\
  \!\!r_3\!\!  \\
 \!\!-r_2\!\!
 \end{array}\right]\nonumber\\
  && \!+\!Q_3(\Upsilon_\mathscr{A}^T\Upsilon_\mathscr{A})\Upsilon_\mathscr{A}^T\!\!\left[\begin{array}{c}
 \!\! -r_2\!\!  \\
 \!\! -r_3\!\!  \\
 \!\! r_0\!\!  \\
\!\! r_1\!\!
 \end{array}\right]
 \!+\!Q_4(\Upsilon_\mathscr{A}^T\Upsilon_\mathscr{A})\Upsilon_\mathscr{A}^T\!\!\left[\begin{array}{c}
 \!\! -r_3\!\!  \\
 \!\! r_2\!\!  \\
 \!\! -r_1\!\!  \\
\!\! r_0\!\!
 \end{array}\right],\label{re1}
\end{eqnarray}
where $P_1,P_2,P_3,P_4$ and $Q_1,Q_2,Q_3,Q_4$ are real coefficient polynomials of degree $\lfloor(m-1)/2\rfloor$ and $\lfloor(m-2)/2\rfloor$, respectively.
\end{proposition}
\begin{proof}
By (\ref{re4}), we have
\begin{eqnarray*}
% \nonumber to remove numbering (before each equation)
   \mathbf{x}_m&=&\mathbf{x}_0+L_{\lfloor(m-1)/2\rfloor}( \mathscr{A}^*\mathscr{A},\mathbf{r}_0)+
L_{\lfloor(m-2)/2\rfloor}( \mathscr{A}^*\mathscr{A},\mathscr{A}^*\mathbf{r}_0)\nonumber \\
&=&\mathbf{x}_0+\mathbf{r}_0\mathbf{a}_0+ (\mathscr{A}^*\mathscr{A}) \mathbf{r}_0\mathbf{a}_1+\cdots+
 (\mathscr{A}^*\mathscr{A})^{\lfloor(m-1)/2\rfloor} \mathbf{r}_0\mathbf{a}_{\lfloor(m-1)/2\rfloor}\nonumber \\
 &&+\mathscr{A}^*\mathbf{r}_0\mathbf{b}_0+ (\mathscr{A}^*\mathscr{A}) \mathscr{A}^*\mathbf{r}_0\mathbf{b}_1+\cdots+
 (\mathscr{A}^*\mathscr{A})^{\lfloor(m-2)/2\rfloor}\mathscr{A}^* \mathbf{r}_0\mathbf{b}_{\lfloor(m-2)/2\rfloor}.\label{re3}
\end{eqnarray*}
According to the properties of  real representations, we have
 \begin{eqnarray}
% \nonumber to remove numbering (before each equation)
   \Upsilon_{\mathbf{x}_m}^c
&=& \Upsilon_{\mathbf{x}_0}^c+\Upsilon_{\mathbf{r}_0\mathbf{a}_0}^c+ (\Upsilon_\mathscr{A}^T\Upsilon_\mathscr{A}) \Upsilon_{\mathbf{r}_0\mathbf{a}_1}^c+\cdots+
 (\Upsilon_\mathscr{A}^T\Upsilon_\mathscr{A})^{\lfloor(m-1)/2\rfloor}\Upsilon_{ \mathbf{r}_0\mathbf{a}_{\lfloor(m-1)/2\rfloor}}^c\nonumber \\
 &&+\Upsilon_\mathscr{A}^T\Upsilon_{\mathbf{r}_0\mathbf{b}_0}^c+ (\Upsilon_\mathscr{A}^T\Upsilon_\mathscr{A}) \Upsilon_\mathscr{A}^T\Upsilon_{\mathbf{r}_0\mathbf{b}_1}^c+\cdots+
 (\Upsilon_\mathscr{A}^T\Upsilon_\mathscr{A})^{\lfloor(m-2)/2\rfloor}\Upsilon_\mathscr{A}^T\Upsilon_{ \mathbf{r}_0\mathbf{b}_{\lfloor(m-2)/2\rfloor}}^c\nonumber \\
 &=& \Upsilon_{\mathbf{x}_0}^c+
 \left[\begin{smallmatrix}
 \!\!r_0a_{0,0}-r_1a_{0,1}-r_2a_{0,2}-r_3a_{0,3}\!\! \\
 \!\!r_1a_{0,0}+r_0a_{0,1}-r_3a_{0,2}+r_2a_{0,3}\!\! \\
 \!\!r_2a_{0,0}+r_3a_{0,1}+r_0a_{0,2}-r_1a_{0,3}\!\! \\
\!\!r_3a_{0,0}+r_2a_{0,1}+r_1a_{0,2}+r_0a_{0,3}\!\!
 \end{smallmatrix}\right]
 + (\Upsilon_\mathscr{A}^T\Upsilon_\mathscr{A})\left[\begin{smallmatrix}
 \!\!r_0a_{1,0}-r_1a_{1,1}-r_2a_{1,2}-r_3a_{1,3}\!\! \\
 \!\!r_1a_{1,0}+r_0a_{1,1}-r_3a_{1,2}+r_2a_{1,3}\!\! \\
 \!\!r_2a_{1,0}+r_3a_{1,1}+r_0a_{1,2}-r_1a_{1,3}\!\! \\
\!\!r_3a_{1,0}+r_2a_{1,1}+r_1a_{1,2}+r_0a_{1,3}\!\!
 \end{smallmatrix}\right]+\cdots\nonumber\\
 &&+
 (\Upsilon_\mathscr{A}^T\Upsilon_\mathscr{A})^{\lfloor(m-1)/2\rfloor}
 \left[\begin{smallmatrix}
 \!\!r_0a_{\lfloor(m-1)/2\rfloor,0}-r_1a_{\lfloor(m-1)/2\rfloor,1}-r_2a_{\lfloor(m-1)/2\rfloor,2}-r_3a_{\lfloor(m-1)/2\rfloor,3}\!\! \\
 \!\!r_1a_{\lfloor(m-1)/2\rfloor,0}+r_0a_{\lfloor(m-1)/2\rfloor,1}-r_3a_{\lfloor(m-1)/2\rfloor,2}+r_2a_{\lfloor(m-1)/2\rfloor,3}\!\! \\
 \!\!r_2a_{\lfloor(m-1)/2\rfloor,0}+r_3a_{\lfloor(m-1)/2\rfloor,1}+r_0a_{\lfloor(m-1)/2\rfloor,2}-r_1a_{\lfloor(m-1)/2\rfloor,3}\!\! \\
\!\!r_3a_{\lfloor(m-1)/2\rfloor,0}+r_2a_{\lfloor(m-1)/2\rfloor,1}+r_1a_{\lfloor(m-1)/2\rfloor,2}+r_0a_{\lfloor(m-1)/2\rfloor,3}\!\!
 \end{smallmatrix}\right]\nonumber \\
 &&+\Upsilon_\mathscr{A}^T\left[\begin{smallmatrix}
 \!\!r_0b_{0,0}-r_1b_{0,1}-r_2b_{0,2}-r_3b_{0,3}\!\! \\
 \!\!r_1b_{0,0}+r_0b_{0,1}-r_3b_{0,2}+r_2b_{0,3}\!\! \\
 \!\!r_2b_{0,0}+r_3b_{0,1}+r_0b_{0,2}-r_1b_{0,3}\!\! \\
\!\!r_3b_{0,0}+r_2b_{0,1}+r_1b_{0,2}+r_0b_{0,3}\!\!
 \end{smallmatrix}\right]+ (\Upsilon_\mathscr{A}^T\Upsilon_\mathscr{A}) \Upsilon_\mathscr{A}^T\left[\begin{smallmatrix}
 \!\!r_0b_{1,0}-r_1b_{1,1}-r_2b_{1,2}-r_3b_{1,3}\!\! \\
 \!\!r_1b_{1,0}+r_0b_{1,1}-r_3b_{1,2}+r_2b_{1,3}\!\! \\
 \!\!r_2b_{1,0}+r_3b_{1,1}+r_0b_{1,2}-r_1b_{1,3}\!\! \\
\!\!r_3b_{1,0}+r_2b_{1,1}+r_1b_{1,2}+r_0b_{1,3}\!\!
 \end{smallmatrix}\right]+\cdots\nonumber\\
 &&+
 (\Upsilon_\mathscr{A}^T\Upsilon_\mathscr{A})^{\lfloor(m-2)/2\rfloor}\Upsilon_\mathscr{A}^T\left[\begin{smallmatrix}
 \!\!r_0b_{\lfloor(m-2)/2\rfloor,0}-r_1b_{\lfloor(m-2)/2\rfloor,1}-r_2b_{\lfloor(m-2)/2\rfloor,2}-r_3b_{\lfloor(m-2)/2\rfloor,3}\!\! \\
 \!\!r_1b_{\lfloor(m-2)/2\rfloor,0}+r_0b_{\lfloor(m-2)/2\rfloor,1}-r_3b_{\lfloor(m-2)/2\rfloor,2}+r_2b_{\lfloor(m-2)/2\rfloor,3}\!\! \\
 \!\!r_2b_{\lfloor(m-2)/2\rfloor,0}+r_3b_{\lfloor(m-2)/2\rfloor,1}+r_0b_{\lfloor(m-2)/2\rfloor,2}-r_1b_{\lfloor(m-2)/2\rfloor,3}\!\! \\
\!\!r_3b_{\lfloor(m-2)/2\rfloor,0}+r_2b_{\lfloor(m-2)/2\rfloor,1}+r_1b_{\lfloor(m-2)/2\rfloor,2}+r_0b_{\lfloor(m-2)/2\rfloor,3}\!\!
 \end{smallmatrix}\right]\nonumber 
\end{eqnarray}
\begin{eqnarray*}
  &=& \Upsilon_{\mathbf{x}_0}^c+\{a_{0,0}+(\Upsilon_\mathscr{A}^T\Upsilon_\mathscr{A})a_{1,0}+\cdots+
  (\Upsilon_\mathscr{A}^T\Upsilon_\mathscr{A})^{\lfloor(m-1)/2\rfloor}a_{\lfloor(m-1)/2\rfloor,0}\}\left[\begin{smallmatrix}
 \!\!r_0\!\! \\
 \!\!r_1\!\! \\
 \!\!r_2\!\! \\
\!\!r_3\!\!
\end{smallmatrix}\right]\nonumber\\
 &&+\{a_{0,1}+(\Upsilon_\mathscr{A}^T\Upsilon_\mathscr{A})a_{1,1}+\cdots+
  (\Upsilon_\mathscr{A}^T\Upsilon_\mathscr{A})^{\lfloor(m-1)/2\rfloor}a_{\lfloor(m-1)/2\rfloor,1}\}\left[\begin{smallmatrix}
 \!\!-r_1\!\! \\
 \!\!r_0\!\! \\
 \!\!r_3\!\! \\
\!\!-r_2\!\!
\end{smallmatrix}\right]\nonumber\\
&&+\{a_{0,2}+(\Upsilon_\mathscr{A}^T\Upsilon_\mathscr{A})a_{1,2}+\cdots+
  (\Upsilon_\mathscr{A}^T\Upsilon_\mathscr{A})^{\lfloor(m-1)/2\rfloor}a_{\lfloor(m-1)/2\rfloor,2}\}\left[\begin{smallmatrix}
 \!\!-r_2\!\! \\
 \!\!-r_3\!\! \\
 \!\!r_0\!\! \\
\!\!r_1\!\!
\end{smallmatrix}\right]\nonumber\\
&&+\{a_{0,3}+(\Upsilon_\mathscr{A}^T\Upsilon_\mathscr{A})a_{1,3}+\cdots+
  (\Upsilon_\mathscr{A}^T\Upsilon_\mathscr{A})^{\lfloor(m-1)/2\rfloor}a_{\lfloor(m-1)/2\rfloor,3}\}\left[\begin{smallmatrix}
 \!\!-r_3\!\! \\
 \!\!r_2\!\! \\
 \!\!-r_1\!\! \\
\!\!r_0\!\!
\end{smallmatrix}\right]\nonumber\\
\end{eqnarray*}
\begin{eqnarray*}
&&+\{\Upsilon_\mathscr{A}^Tb_{0,0}+(\Upsilon_\mathscr{A}^T\Upsilon_\mathscr{A})\Upsilon_\mathscr{A}^Tb_{1,0}+\cdots+
  (\Upsilon_\mathscr{A}^T\Upsilon_\mathscr{A})^{\lfloor(m-2)/2\rfloor}\Upsilon_\mathscr{A}^T
  b_{\lfloor(m-2)/2\rfloor,0}\}\left[\begin{smallmatrix}
 \!\!r_0\!\! \\
 \!\!r_1\!\! \\
 \!\!r_2\!\! \\
\!\!r_3\!\!
\end{smallmatrix}\right]\nonumber\\
 &&+\{\Upsilon_\mathscr{A}^Tb_{0,1}+(\Upsilon_\mathscr{A}^T\Upsilon_\mathscr{A})\Upsilon_\mathscr{A}^Tb_{1,1}+\cdots+
  (\Upsilon_\mathscr{A}^T\Upsilon_\mathscr{A})^{\lfloor(m-2)/2\rfloor}\Upsilon_\mathscr{A}^T
  b_{\lfloor(m-2)/2\rfloor,1}\}\left[\begin{smallmatrix}
 \!\!-r_1\!\! \\
 \!\!r_0\!\! \\
 \!\!r_3\!\! \\
\!\!-r_2\!\!
\end{smallmatrix}\right]\nonumber\\
&&+\{\Upsilon_\mathscr{A}^Tb_{0,2}+(\Upsilon_\mathscr{A}^T\Upsilon_\mathscr{A})\Upsilon_\mathscr{A}^Tb_{1,2}+\cdots+
  (\Upsilon_\mathscr{A}^T\Upsilon_\mathscr{A})^{\lfloor(m-2)/2\rfloor}\Upsilon_\mathscr{A}^T
  b_{\lfloor(m-2)/2\rfloor,2}\}\left[\begin{smallmatrix}
 \!\!-r_2\!\! \\
 \!\!-r_3\!\! \\
 \!\!r_0\!\! \\
\!\!r_1\!\!
\end{smallmatrix}\right]\nonumber\\
&&+\{\Upsilon_\mathscr{A}^Tb_{0,3}+(\Upsilon_\mathscr{A}^T\Upsilon_\mathscr{A})\Upsilon_\mathscr{A}^Tb_{1,3}+\cdots+
  (\Upsilon_\mathscr{A}^T\Upsilon_\mathscr{A})^{\lfloor(m-2)/2\rfloor}\Upsilon_\mathscr{A}^T
  b_{\lfloor(m-2)/2\rfloor,3}\}\left[\begin{smallmatrix}
 \!\!-r_3\!\! \\
 \!\!r_2\!\! \\
 \!\!-r_1\!\! \\
\!\!r_0\!\!
\end{smallmatrix}\right].\nonumber
\end{eqnarray*}
This proves the desired result.
\end{proof}

We see from Proposition \ref{prop6.2} that the corresponding real equivalence of QCG-type schemes (\ref{re4}) are iterations of the type (\ref{re1}) and not the obvious real
 CG-type iterations (\ref{re2}). By using the schemes of the type (\ref{re2}), from the first, we give up $3m$ of the $4m$ real parameters which are available for optimizing the QCG-type method (\ref{re4}). Consequently, it is preferable to solve the quaternion linear systems (\ref{1.1}) rather than the real version (\ref{6.3}) by QCG-type methods. %In addition, the numerical result shows that the convergence behavior of the two methods is different.

One kind of  QQMR  method proposed by Li, Wang, and Zhang \cite{lwz2024} is based on quaternion coupled three-term biconjugate orthonormalization procedure and generates two sequences of vectors $\mathbf{v}_1,\mathbf{v}_2,\cdots$ and $\mathbf{w}_1,\mathbf{w}_2,\cdots$ by two three-term recurrence relations such that for all $m=1,2,\cdots$, $\mathscr{V}_m=[\mathbf{v}_1,\mathbf{v}_2,\cdots,\mathbf{v}_m]$ and $\mathscr{W}_m=[\mathbf{w}_1,\mathbf{w}_2,\cdots,\mathbf{w}_m]$, we have
\begin{subequations}
\begin{align}
 \mathscr{A}\mathscr{V}_m &=\mathscr{V}_m\mathscr{H}_m+\rho_{m+1}\mathbf{v}_{m+1}\mathbf{e}_m^*,\\
  \mathscr{A}^*\mathscr{W}_m &=\mathscr{W}_{m}\Gamma_{m}^{-1}\widehat{\mathscr{H}}_{m}\Gamma_m+\mathbf{w}_{m+1}\rho_{m+1}\mathbf{e}_m^*\Gamma_m,\\
  \mathscr{W}_m^*\mathscr{V}_m&=\mathscr{D}_m,\mathscr{W}_m^*\mathscr{A}\mathscr{V}_m=\mathscr{D}_m\mathscr{H}_m,
\end{align}
\end{subequations}
where $\mathscr{D}_m\in\mathbb{Q}^{ m\times m}$ is an invertible    diagonal quaternion matrix,  $\Gamma_m\in\mathbb{R}^{m\times m}$ is a real diagonal matrix, $\mathscr{H}_m$ and $\widehat{\mathscr{H}}_m\in\mathbb{Q}^{ m\times m}$ are   tridiagonal quaternion matrices.
%where $\mathscr{D}_m=\mathrm{diag}(\sigma_1,\sigma_2,\cdots,\sigma_m)\in\mathbb{Q}^{m\times m}$ with $\sigma_i\neq0$, $i=1,\cdots,m$, and $\Gamma_m=\mathrm{diag}(s_1,s_2,\cdots,s_m)\in\mathbb{R}^{m\times m}$, and

The approximations $\mathbf{x}_{m}=\mathbf{x}_0+\mathscr{V}_m\mathbf{z}_m$ are constructed such that $\mathbf{z}_m\in\mathbb{Q}^m$ is the solution of the minimization problem
\begin{equation*}
  \min\limits_{\mathbf{z}\in\mathbb{Q}^m}\|\|\mathbf{r}_0\|_2\mathbf{e}_1-\mathscr{H}_{m+1,m}\mathbf{z}\|_2,
\end{equation*}
where
\begin{equation*}
  \mathscr{H}_{m+1,m}=
  \left[\begin{array}{c}
  \mathscr{H}_m\\
  \rho_{m+1}\mathbf{e}_m^*
  \end{array}\right]\in\mathbb{Q}^{(m+1)\times m}.
\end{equation*}
Since $\mathcal{K}_m(\mathscr{A},\mathbf{r}_0)=\mathrm{span}(\mathbf{v}_1,\cdots,\mathbf{v}_m)$, QQMR is a Krylov subspace method with residuals
\begin{equation*}
 \mathbf{r}_m=\mathbf{r}_0- \mathscr{A}L_{m}(  \mathscr{A}, \mathbf{r}_0)
  =\mathscr{V}_{m+1}(\|\mathbf{r}_0\|_2\mathbf{e}_1-\mathscr{H}_{m+1,m}\mathbf{z}).
\end{equation*}
But $\|\mathbf{r}_m\|_2$ is not minimal over the Krylov subspace since $\mathscr{V}_{m+1}$ is not partially unitary. Comparing with QNHERLQ and QNHERQR, we have the following remarks:
\begin{itemize}
\item QQMR is a Krylov subspace method, so the convergence depends on the spectrum and Jordan structure of $\mathscr{A}$. %But the spectrum and Jordan structure of $\mathscr{A}$ may not be special.
\item  It is difficult to compare the number of iteration steps in general because convergence of the  QQMR  method depends on the eigenvalues while that of both QNHERLQ and QNHERQR methods depends on the singular values.
\item Computational complexity at each iteration   is almost the same for QQMR and the proposed methods. In fact, all methods work with two matrix-vector multiplications and the solution of a least squares problem or quaternion linear system with  tridiagonal coefficients matrices by an updated LQ or QR factorization.
\end{itemize}

\section{Numerical experiments}
In this section,   several examples, including the Lorenz attractor, randomly data coming from the suite
sparse Matrix Collection and color image deblurring problems, are given to illustrate the feasibility and validity of the proposed methods compared with QGMRES \cite{jn2021}, QQMR \cite{lwz2024}  and deconvtv \cite{vtv}
solvers. All experiments were conducted on a personal computer in Matlab 2019b with AMD Ryzen 7-7735H 3.20GHz and 32.00 GB memory and started from the initial guess with the zero vector. The symbols $\textbf{CPU}$ and $\textbf{IT}$ denote the elapsed computing time in seconds and the number of iteration steps, respectively.  Without further illustration, all computations were terminated once the relative residual  $\textbf{RR}<10^{-6}$ or the maximum iteration  exceeds 5000, where
\begin{equation}\label{6-1}
    \textbf{RR}=\frac{\|\textbf{b}-\mathscr{A}\textbf{x}_j\|_2}{\|\textbf{b}\|_2}.
\end{equation}

\begin{example}\label{example-1}
{\rm Consider the quaternion linear systems \eqref{1.1} with its coefficient matrix and right-hand side vector defined as
    $$\begin{aligned}
\mathscr{A}&={A}_0+{A}_1\mathbf{i}+{A}_2\mathbf{j}+{A}_3\mathbf{k},\\
\mathbf{b}&={b}_0+{b}_1\mathbf{i}+{b}_2\mathbf{j}+{b}_3\mathbf{k},
\end{aligned}
    $$
    where ${A}_1=1.5*{A}_0,$ $ {A}_2=2*{A}_0,$ $ {A}_3=0.5*{A}_0$ with ${A}_0$ coming from the ``Suite Sparse Matrix Collection"
\footnote{The Suite Sparse Matrix Collection is available on https://sparse.tamu.edu/} at the University of Florida. Here, the right-hand side vector was chosen such that the vector of all ones is the exact solution of Eq.\eqref{1.1}, and the matrix ${A}_0$ was set as $\textbf{rw136}$ of size $136\times 136$, $\textbf{gre512}$ of size $512\times 512$, $\textbf{rdb1250}$ of size $1182\times 1182$, and $ \textbf{dw1024}$ of size $2048\times 2048$, respectively.}
\end{example}

We perform  QGMRES, QQMR, QNHERLQ and QNHERQR solvers for solving Eq.\eqref{1.1}, and then
report  the numerical results in Table \ref{Table-1}, in which we know that the QNHERLQ and QNHERQR solvers vastly outperform QGMRES and QQMR in terms of the elapsed computing time, despite the fact that the number of iterations of QGMRES among all tested algorithms is smallest. More precisely, the convergence rate of QNHERLQ and QNHERQR solvers is about twice the latter two. Moreover, the  elapsed computing time and iterations required by QNHERQR are commonly less than those of QNHERLQ, which converges faster than the latter.
The obtained observations are also intuitively shown in Figure \ref{fig:1}, in which we depict  the convergence curves of all algorithms for the case $\textbf{rw136}$. The curves in this figure illustrate that the convergence behavior of QNHERQR performs smoother than  QNHERLQ, which is more reliable.

\setlength{\tabcolsep}{0.3\tabcolsep}
\begin{table}[!t]
	\caption{Numerical results of Example \ref{example-1}.}\label{Table-1}
 \renewcommand\arraystretch{1.0}
	\centering
	\scalebox{.82}{\begin{tabular}{|c|c|c|c|c|c|c|}
		\hline
		\textbf{Data set}     &\bm{$n$}  &\textbf{Algorithms}    &\textbf{CPU}         &\textbf{IT}     &\textbf{RR}                &\textbf{Application discipline}\\
		\hline
		\multirow{4}{*}{\textbf{rw136}}
		&\multirow{4}{*}{136} &{QGMRES}\cite{jn2021}        &0.2960&136  & 8.7893e-07&\multirow{4}{*}{Statistical/Mathematical Problem}\\
		&\multicolumn{1}{c|}{}&QQMR \cite{lwz2024}  &0.1974&525&9.4431e-07
&\multicolumn{1}{c|}{}\\
        &\multicolumn{1}{c|}{}&QNHERLQ           & 0.0882&216&9.2136e-07&\multicolumn{1}{c|}{}\\
        &\multicolumn{1}{c|}{}&QNHERQR           &0.0418   &192 & 9.4764e-07 &\multicolumn{1}{c|}{}\\

\hline
		\multirow{4}{*}{\textbf{gre512}}
		&\multirow{4}{*}{512} &{QGMRES} \cite{jn2021}        &5.9708&513 & 9.9340e-07&\multirow{4}{*}{Directed Weighted Graph}\\
		&\multicolumn{1}{c|}{}&QQMR\cite{lwz2024}   &3.9659 &979&9.6255e-07
&\multicolumn{1}{c|}{}\\
        &\multicolumn{1}{c|}{}&QNHERLQ           & 1.7113& 365&8.5943e-07 &\multicolumn{1}{c|}{}\\
        &\multicolumn{1}{c|}{}&QNHERQR           & 1.6339&358&9.5202e-07&\multicolumn{1}{c|}{}\\

        	\hline
		\multirow{4}{*}{\textbf{rdb1250}}
		&\multirow{4}{*}{1182}  &{QGMRES} \cite{jn2021} &42.2114&1076&9.2663e-07&\multirow{4}{*}{Computational Fluid Problem}\\
		&\multicolumn{1}{c|}{}&QQMR \cite{lwz2024}  &31.4523&3193&9.7686e-07  &\multicolumn{1}{c|}{}\\
        &\multicolumn{1}{c|}{}&QNHERLQ           &
    14.7541&2516& 9.0242e-07&\multicolumn{1}{c|}{}\\
        &\multicolumn{1}{c|}{}&QNHERQR           &   11.1605&2208& 9.8902e-07&\multicolumn{1}{c|}{}\\
		\hline
		\multirow{4}{*}{\textbf{dw1024}}
		&\multirow{4}{*}{2048}  &{QGMRES} \cite{jn2021}      &179.6525&2012     &8.6783e-07&\multirow{4}{*}{Electromagnetics Problem}\\
		&\multicolumn{1}{c|}{}&QQMR \cite{lwz2024} &133.2831& 4465 &9.9353e-07&\multicolumn{1}{c|}{}\\
        &\multicolumn{1}{c|}{}&QNHERLQ          &84.1379 &1669    &9.8209e-07&\multicolumn{1}{c|}{}\\
        &\multicolumn{1}{c|}{}&QNHERQR           &61.2343& 1497&9.9149e-07&\multicolumn{1}{c|}{}\\
		\hline
	\end{tabular}}
\end{table}

\begin{figure}[!htbp]
	\hspace{-8em}
	\begin{center}
		\begin{minipage}[c]{0.7\textwidth}
	\includegraphics[width=3in]{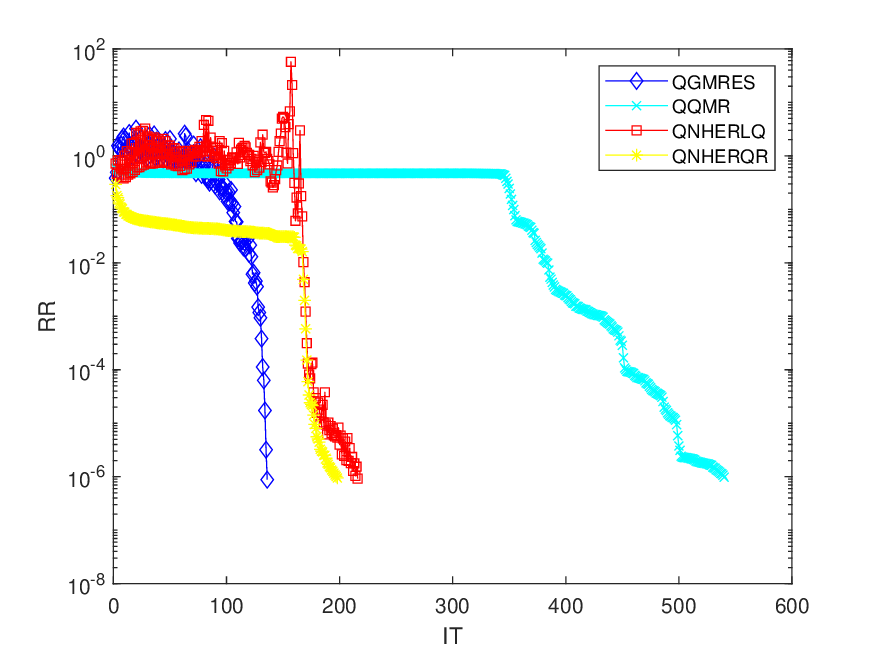}
		\end{minipage}
	\end{center}
	\vspace{-0.5em}\caption{ Convergence curves of Example \ref{example-1} with \textbf{rw136}.\label{fig:1}}
\end{figure}

Note that sensors collect various signals from engineering problems in three dimensions. Herein, we denote a three-dimensional signal by a pure quaternion function of time $\mathbf{x}(t)={x}_r(t)\mathbf{i}+{x}_g(t)\mathbf{j}+{x}_b(t)\mathbf{k}$, in which ${x}_r(t),{x}_g(t)$, and ${x}_b(t)$ represent real functions, such as red, green and blue channels, respectively. As shown in \cite{jn2021}, the collected signals $\mathbf{x}(t-p),\cdots,\mathbf{x}(t)$, regarded as input signals, often act on the desired quaternion filters $\mathbf{w}(0),\cdots,\mathbf{w}(p)$ with $\mathbf{w}(s)={w}_0(s)+{w}_r(s)\mathbf{i}+{w}_g(s)\mathbf{j}+{w}_b(s)\mathbf{k}$, $s=0,\cdots,p$, such that the filtered output can match the target signal $\mathbf{y}(t)$, i.e.,
\begin{equation}\label{6-2}
\mathbf{y}(t)=\sum_{s=0}^{p}\mathbf{x}(t-s)\mathbf{w}(s).
\end{equation}
If we now take $q$ signal data along the time series, this model can be reformulated as the following quaternion linear systems
\begin{equation}\label{6-3}
\mathscr{X}\mathbf{w}=\mathbf{y},
\end{equation}
where $$\begin{aligned}
\mathscr{X}&=\begin{bmatrix}
\mathbf{x}(t)&\mathbf{x}(t-1)&\mathbf{x}(t-2)&\cdots&\mathbf{x}(t-p)\\
\mathbf{x}(t+1)&\mathbf{x}(t)&\mathbf{x}(t-1)&\cdots&\mathbf{x}(t-p+1)\\
\vdots&\vdots&\ddots&\vdots&\vdots\\
\mathbf{x}(t+q)&\mathbf{x}(t+q-1)&\mathbf{x}(t+q-2)&\cdots&\mathbf{x}(t+q-p)\\
\end{bmatrix},
\end{aligned}
$$
$$
\begin{aligned}
\mathbf{w}&=[\mathbf{w}(0) \quad \mathbf{w}(1) \quad \mathbf{w}(2)\quad \cdots \quad\mathbf{w}(p)]^T,\\
\mathbf{y}&=[\mathbf{y}(t)\quad \mathbf{y}(t+1) \quad \mathbf{y}(t+2) \quad\cdots \quad\mathbf{y}(t+q)]^T.\\
\end{aligned}$$
Let $\mathscr{X}=X_0+X_1\mathbf{i}+X_2\mathbf{j}+X_3\mathbf{k}\in \mathbb{Q}^{n\times n},$ $\mathbf{y}=y_0+y_1\mathbf{i}+y_2\mathbf{j}+y_3\mathbf{k}\in \mathbb{Q}^{n}$ and $\mathbf{w}=w_0+w_1\mathbf{i}+w_2\mathbf{j}+w_3\mathbf{k}\in \mathbb{Q}^{n}$. Then  we employ   QGMERS, QQMR, QNHERQR and QNHERLQ to solve the quaternion linear systems \eqref{6-3}.

\begin{table}[htbp]
	\caption{Numerical results of Example  \ref{example-2}.}\label{Table-2}
 \renewcommand\arraystretch{1.1}
	\centering
		\scalebox{.9}{\begin{tabular}{|c|c|c|c|c|c|}
\hline
\textbf{Dimension}  &\textbf{Algorithms}  &\textbf{CPU}  &\textbf{IT}&\textbf{RR}\\   \hline
\multirow{5}{*}{100}  &{QGMRES} \cite{jn2021}    &3.2784   &476    &7.8387e-07\\
		&QQMR \cite{lwz2024}   &2.3891	 &1415   &9.9413e-07\\
        &QNHERLQ &1.1579    & 595& 8.0501e-07\\
        &QNHERQR &1.0573&593     &7.3231e-07\\
		\hline
\multirow{5}{*}{200}  &{QGMRES}\cite{jn2021}    &9.3561    &759     &8.8059e-07\\
		&QQMR   \cite{lwz2024}  &6.9652	 &1747 &9.9424e-03\\
        &QNHERLQ  &3.4860    &828     & 8.8999e-07\\
        &QNHERQR&3.2452&815    &8.5495e-07\\
		\hline
\multirow{5}{*}{300}  &{QGMRES} \cite{jn2021} & 57.9254&1145&9.3568e-07\\
		&QQMR  \cite{lwz2024}&42.9523&2782     &9.7893e-07\\
        &QNHERLQ&31.1678&1396&9.6234e-07\\
        &QNHERQR&29.7549&1374&9.4581e-07\\
		\hline
        \multirow{5}{*}{400}  &{QGMRES} \cite{jn2021}   &156.8792 &1589     &9.6138e-07\\
		&QQMR \cite{lwz2024} &123.4569 &  4962   &9.9070e-07\\
        &QNHERLQ&89.8536&1802&9.1405e-07\\
        &QNHERQR&85.7947&1762&8.1917e-07\\
		\hline
	\end{tabular}}
\end{table}

\begin{example}\label{example-2}
\rm Consider Lorenz attractor, i.e.,  a three-dimensional nonlinear system used
originally to model atmospheric turbulence \cite{S2001}. Mathematically, the Lorenz system is the following coupled differential equations
\begin{equation}\label{6-4}
\frac{\partial x}{\partial t}=\alpha(y-x),\quad \frac{\partial y}{\partial t}=x(\rho-z)-y,\quad \frac{\partial z}{\partial t}=xy-\beta z,
\end{equation}
where $\alpha,\beta$ and $\rho$ are positive. After selecting the parameters $\alpha$, $\beta$   and $\rho$,  the above equations \eqref{6-4}  can be soloved by the MATLAB command $\text{ODE45}(f(t, [x, y, z]),$ $ [0, T],[2,3,4])$, where $T>0$. Here we set $\alpha=10$, $\beta=8/3$ and $\rho=28$ for characterizing the chaotic behavior of the Lorenz attractor.  Let $y_r(t),y_g(t)$ and $y_b(t)$ be the solutions of Eqs.\eqref{6-4}, i.e., the target signal
$$
\mathbf{y}(t)=y_r(t)\mathbf{i}+y_g(t)\mathbf{j}+y_b(t)\mathbf{k}.
$$
Then the input signal  with the random noise has the following form
$$\mathbf{x}(t)=y_r(t-1)\mathbf{i}+y_g(t-1)\mathbf{j}+y_b(t-1)\mathbf{k}+\mathbf{n}(t),$$
where $\mathbf{n}(t)$ is a random noise. The obtained quaternion systems are shown as Eq.\eqref{6-3} \cite{jn2021}.
\end{example}

Starting with the zero vector, we run all mentioned iterative methods for solving the above problem with different dimensions and  report  CPU, IT, and RR in  Table \ref{Table-2}, from which one may see that under the same tolerance,  the iteration number   of QGMRES is commonly less than the other solvers, whereas the elapsed computing time required by QNHERLQ and QNHERQR is much less than QGMRES and QQMR. The convergence curves with $n=100$ are plotted in Figure \ref{fig:2}, in which one can see that QNHERQR still converges more smoothly than that of QNHERLQ. This coincides with the given results.

\begin{figure}[!htbp]
	\hspace{-8em}
	\begin{center}
		\begin{minipage}[c]{0.7\textwidth}
	\includegraphics[width=3in]{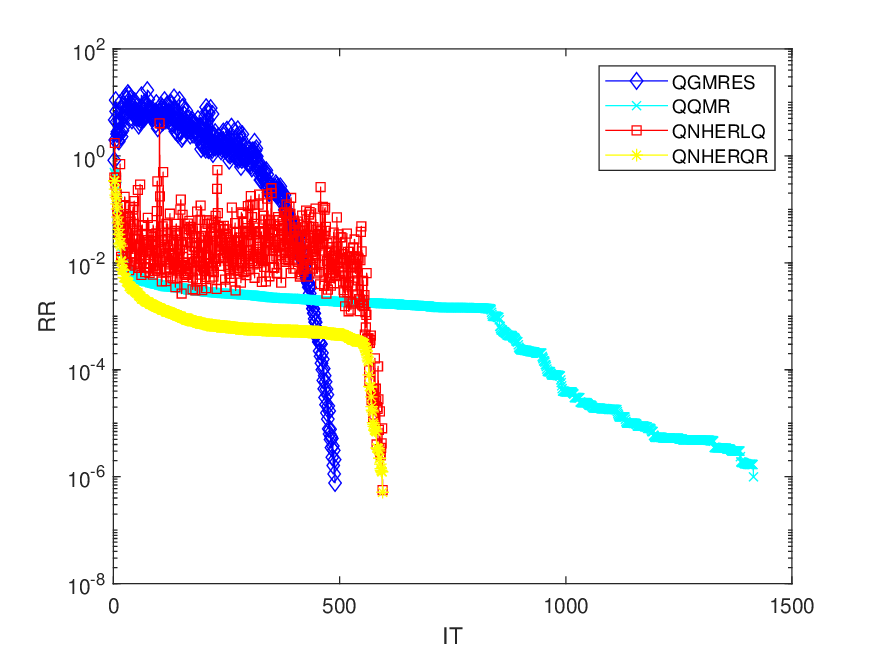}
		\end{minipage}
	\end{center}
	\vspace{-0.5em}\caption{ Convergence curves of Example \ref{example-2} with $n=100$.\label{fig:2}}
\end{figure}

In what follows, we test the effectiveness of the proposed algorithms for solving model \eqref{1.1} arising from color image deblurring problems.  The original color image can be expressed by a pure quaternion matrix  $\mathscr{X}=X_1\mathbf{i}+X_2\mathbf{j}+X_3\mathbf{k}\in \mathbb{Q}^{n\times n}$, %consists of $n\times n$ color image pixel values in the range $[0, d]$, where $d=255$ represents the maximum possible pixel value of the image,
where $X_1,X_2,X_3\in \mathbb{R}^{n\times m}$ denote the red, green and blue channels, respectively.
Let $\mathbf{x}=\mbox{vec}(\mathscr{X})\in \mathbb{Q}^{n^2}$ be a quaternion vector obtained by stacking the columns of $\mathscr{X}$, and $\mathscr{A}$ be a blurring matrix. Such color image contaminated by $\mathscr{A}$ is the blurred color image $\mathbf{b}\in\mathbb{Q}^{n^2}$. It is easy to see that the model of color image formation can be modelled as    the quaternion linear systems \eqref{1.1}.

Note that the total variation regularized least-squares deconvolution (deconvtv) \cite{vtv}, which can be referred as a variation of the popularly known alternating direction methods of multipliers (ADMM), is a popular method for color image deblurring problems. Hence we compare the proposed algorithms with QGMRES  \cite{jn2021}, QQMR \cite{lwz2024}, and deconvtv \cite{vtv}.

We evaluate the performance of all tested algorithms  by four measures, i.e., the computing time ($\mathbf{CPU}$), the peak signal-to-noise ratio ($\mathbf{PSNR}$) in decibel ($\mathbf{dB}$), structural similarity ($\mathbf{SSIM}$) \cite{WBS2004}, and relative error ($\mathbf{RE}$). The RE and PSNR between the recovered color image
$\mathscr{X}_{\text{restored}}\in\mathbb{Q}^{I_1\times I_2}$ and  true color image
$\mathscr{X}_{\text{true}}\in\mathbb{Q}^{I_1\times I_2}$ are defined by
\begin{equation*}
 \mathrm{RE}( \mathscr{X}_{\text{true}},\mathscr{X}_{\text{restored}})=
 \dfrac{\|\mathscr{X}_{\text{true}}-\mathscr{X}_{\text{restored}}\|}{\|\mathscr{X}_{\text{true}}\|}
\end{equation*}
 and
\begin{equation*}
 \mathrm{PSNR}( \mathscr{X}_{\text{true}},\mathscr{X}_{\text{true}})=
 10\log_{10}\Big(\dfrac{3I_1I_2d^2}{\|\mathscr{X}_{\text{true}}-\mathscr{X}_{\text{restored}}\|^2}\Big),
\end{equation*}
respectively. SSIM is used to measure the similarity of true image  $\mathscr{X}_{\text{true}}\in\mathbb{Q}^{I_1\times I_2}$ and restored image  $\mathscr{X}_{\text{restored}}\in\mathbb{Q}^{I_1\times I_2}$ in luminance, contrast and structure, which is defined by
\begin{equation*}
  \mathrm{SSIM}(\mathscr{X}_{\text{true}},\mathscr{X}_{\text{restored}})=
  \dfrac{(2\mu_{\mathscr{X}_{\text{restored}}} \mu_{\mathscr{X}_{\text{true}}}+c_1)
  (\sigma_{\mathscr{X}_{\text{restored}}\mathscr{X}_{\text{true}}}+c_2)}{(\mu_{\mathscr{X}_{\text{restored}}}^2+\mu_{\mathscr{X}_{\text{true}}}^2+c_1)
  (\sigma_{\mathscr{X}_{\text{restored}}}^2+\sigma_{\mathscr{X}_{\text{true}}}^2+c_2)},
\end{equation*}
where $\mu_{\mathscr{X}_{\text{restored}}}$ and $\mu_{\mathscr{X}_{\text{true}}}$ are the mean values of  $\mathscr{X}_{\text{restored}}$ and   $\mathscr{X}_{\text{true}}$, respectively, $\sigma_{\mathscr{X}_{\text{restored}}}^2$ and $\sigma_{\mathscr{X}_{\text{true}}}^2$ are the standard variances of $\mathscr{X}_{\text{restored}}$ and $\mathscr{X}_{\text{true}}$, respectively, $\sigma_{\mathscr{X}_{\text{restored}}\mathscr{X}_{\text{true}}}$ is the covariance of $\mathscr{X}_{\text{restored}}$ and $\mathscr{X}_{\text{true}}$, $c_1$ and $c_2$ are two constants to stabilize the division with small denominator.
Here and after, we stop all tested solvers if the elapsed computing time reaches 200 seconds.

\begin{example}\label{example-3}
{\rm In this example, we first consider three original color images, including the ``\textbf{Dawn}", ``\textbf{Desk}", ``\textbf{Balls}" color images of size $512\times 512$, blurred by a single channel blurring matrices $\mathscr{A}=(a_{ij})\in \mathbb{R}^{n\times n}$, which models the motion blur. We generate  the convolution kernel by the MATLAB function `fspecial(`motion', len, theta)', with the blur length (len = 50) and the direction (theta = 0), indicating horizontal motion. The convolution operation, performed under zero boundary conditions, was then linearized to obtain the blurring matrix $\mathscr{A}$, see \cite{psf} in detail.}
\end{example}

We execute  the deconvtv, QGMRES, QQMR, QNHERLQ, and QNHERQR for solving Eq.\eqref{1.1}, respectively, and then report   the numerical results and the restored images in Table \ref{Table-3} and Figure \ref{fig:3}. We see from Table  \ref{Table-3} that the restoration quality achieved by QGMRES, QQMR, QNHERLQ, and QNHERQR  is markedly superior to the deconvtv.  Compared with  QGMRES and QQMR, the proposed methods restore images with higher quality, which are more competitive. Besides, we observe that QNHERQR performs slightly better than   QNHERLQ in both PSNR and SSIM.

%\begin{table}[htbp]
%	%\footnotesize
% \caption{Numerical results of Example \ref{example-3}\label{Table-3}}
%	\renewcommand\arraystretch{1.1}
%	\centering
%		\scalebox{.9}{\begin{tabular}{cccccccccccccccccccc} \hline
%		&       &       \bf Dawn   &       &       &       & \bf Desk       &       &       &       & \bf Balls   &       &       &       &       &       &  \\
%		\cline{3-5}\cline{7-9} \cline{11-13} 
%		\bf Algorithm  & &  \bf PSNR &   \bf SSIM & \bf RE & &  \bf PSNR &  \bf SSIM & \bf RE& & \bf PSNR &  \bf SSIM &  \bf RE \\\hline
%		{deconvtv}  &&23.19&0.84&6.86e-03&   &25.88&0.81&7.23e-03	&&17.29&0.64&9.85e-03\\
%		{QGMRES}  &&34.87&0.97&2.67e-03	&   &28.01&0.91&5.67e-03& &23.70&0.85&7.91e-03\\
%		{QQMR} &&35.12&0.97&2.55e-03&   &28.33&0.91&5.09e-03	&&23.98&0.86&7.78e-03\\
%		{QNHERLQ }&&36.16&0.98&2.01e-03&&29.25&0.94&4.63e-03	&&24.60&0.88&6.97e-03\\
%		{QNHERQR }&&36.34&0.99&1.93e-03& &29.46&0.94&4.32e-03	&&24.84&0.89&6.12e-03\\ \hline	
%	\end{tabular}}
%\end{table}

\begin{table}[htbp]
	%\footnotesize
	\caption{Numerical results of Example \ref{example-3}\label{Table-3}}
 \renewcommand\arraystretch{1.1}
	\centering
			\scalebox{.9}{\begin{tabular}{|c|c|c|c|c|} \hline
			\bf Case &\bf Algorithm &\bf PSNR & \bf SSIM &\bf RE \\  \hline
 &{deconvtv} \cite{vtv}&23.19&0.84&6.86e-03\\
  &{QGMRES}\cite{jn2021} &34.87&0.97&2.67e-03\\
 \textbf{Dawn} &{QQMR}\cite{lwz2024}&35.12&0.97&2.55e-03\\
  &{QNHERLQ }&36.16&0.98&2.01e-03\\
 &{QNHERQR }&36.34&0.99&1.93e-03\\ \hline

&{deconvtv} \cite{vtv}&25.88&0.81&7.23e-03\\
  &{QGMRES}\cite{jn2021} &28.01&0.91&5.67e-03\\
 \textbf{Desk} &{QQMR}\cite{lwz2024}&28.33&0.91&5.09e-03\\
  &{QNHERLQ}&29.25&0.94&4.63e-03\\
 &{QNHERQR }&29.46&0.94&4.32e-03\\ \hline

&{deconvtv} \cite{vtv}&17.29&0.64&9.85e-03\\
  &{QGMRES}\cite{jn2021} &23.70&0.85&7.91e-03\\
 \textbf{Balls} &{QQMR}\cite{lwz2024}&23.98&0.86&7.78e-03\\
  &{QNHERLQ }&24.60&0.88&6.97e-03\\
 &{QNHERQR }&24.84&0.89&6.12e-03\\ \hline

		\end{tabular}}
\end{table}

\begin{figure}[!htbp]
	\hspace{-5em}
	\begin{center}
		\begin{minipage}[c]{1.0\textwidth}
			\includegraphics[width=5in]{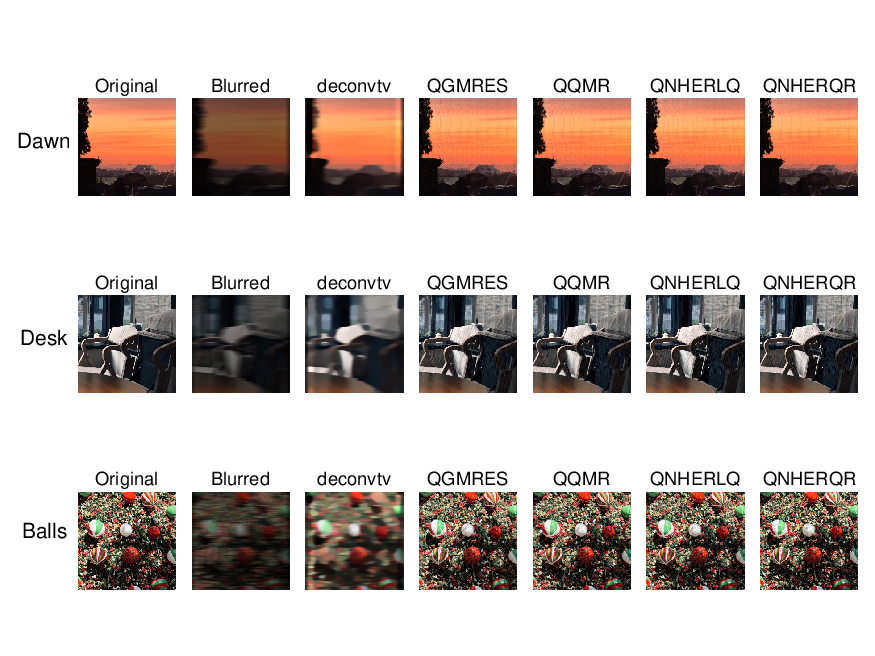}
		\end{minipage}
	\end{center}
\vspace{-1.5em}\caption{Left to Right: the original $512\times 512$ images, the blurred images with
 single channel, the deblurring images were restored by QGMRES, QQMR, QNHERLQ, and QNHERQR.}\label{fig:3}
\end{figure}

\begin{example}\label{example-4}
\rm In this example, we  recover three color images, ``\textbf{Ship}'', ``\textbf{Dog}'' and ``\textbf{Pizza}'' of size {$100\times 100$} in PNG format to further show the effectiveness of our algorithms. Notice that these color images contain  four channels, i.e., transparency, red, green, and blue channels {\rm \cite{Smith2}}, and can be characterized by a quaternion matrix $\mathscr{\hat{X}}=\hat{{X}}_0+\hat{{X}}_1\mathbf{i}+\hat{X}_2\mathbf{j}+\hat{X}_3\mathbf{k}\in \mathbb{Q}^{n\times n}$. Assume now that the three color images are blurred by a multichannel blurring matrix $\mathscr{A}$, e.g., $\mathscr{A}=A_0+A_1\mathbf{i}+A_2\mathbf{j}+A_3\mathbf{k}\in \mathbb{Q}^{n^2\times n^2}$, where $A_0=B_1\otimes B_2\in\mathbb{R}^{n^2\times n^2}$ with
$B_1=(b^{(1)}_{ij})$ and $B_2=(b^{(2)}_{ij})$ being the $n\times n$ Toeplitz matrices given by
\begin{equation*}
\label{rsr}
\begin{aligned}
    b^{(1)}_{ij}&=\left\{
\begin{matrix}
&\frac{1}{\sigma\sqrt{2\pi}}exp(-\frac{{(i-j)}^2}{2\sigma^2}),  &|i-j|\leq r,\\
	&0,                                                             &\mbox{otherwise},
\end{matrix}\right.
\end{aligned}
\end{equation*}
and
\begin{equation*}
\begin{aligned}
b^{(2)}_{ij}&=\left\{
\begin{matrix}
&\frac{1}{2s-1},    &|i-j|\leq s,\\
&0,                 &\mbox{otherwise},	
\end{matrix}\right.
\end{aligned}
\end{equation*}
respectively, and $A_1=A_0, A_2=1.5*A_0, A_3=2*A_0$. Here the parameters are set to be $\sigma=1$, $r=4$ and $s=7$.
\end{example}

Applying QGMRES, QQMR, QNHERLQ, and QNHERQR solvers for restoring three blurred color images, we then record  the numerical results in Table \ref{Table-4}, and display  the restored color images in Figure \ref{fig:4}. As observed from   Table \ref{Table-4}, all tested algorithms for restoring the
blurred color images are feasible, in which  QNHERLQ and QNHERQR seem to recover
with much higher quality than other solvers.
Specifically, in comparison with QGMRES and QQMR,
QNHERLQ improves the PSNR
and SSIM values about $3\%$, and QNHERQR improves  the PSNR and SSIM values about $5\%$.
This also illustrates that QNHERQR still performs
slightly better than QNHERLQ.

\begin{table}[htbp]
	%\footnotesize
		\caption{Numerical results of Example \ref{example-4}\label{Table-4}}
 \renewcommand\arraystretch{1.1}
		\centering
			\scalebox{.9}{\begin{tabular}{|c|c|c|c|c|} \hline
			\bf Case &\bf Algorithm &\bf PSNR & \bf SSIM &\bf RE \\  \hline
  &{QGMRES}\cite{jn2021} &24.94&0.90&6.01e-03\\
 \textbf{Ship} &{QQMR}\cite{lwz2024}&25.38&0.90&5.73e-03\\
  &{QNHERLQ }&26.17&0.92&5.25e-03\\
 &{QNHERQR }&26.90&0.93&4.84e-03\\ \hline

  &{QGMRES}\cite{jn2021} &22.74&0.90&7.36e-03\\
 \textbf{Dog} &{QQMR}\cite{lwz2024}&23.68&0.90&6.58e-03\\
  &{QNHERLQ}&24.76&0.92&7.65e-03\\
 &{QNHERQR }&25.63&0.93&5.12e-03\\ \hline
  &{QGMRES} \cite{jn2021} &20.93&0.88&8.89e-03\\
 \textbf{Pizza} &{QQMR}\cite{lwz2024}&21.09&0.88&8.46e-03\\
  &{QNHERLQ }&21.84&0.89&7.91e-03\\
 &{QNHERQR }&22.78&0.90&7.25e-03\\ \hline
		\end{tabular}}

\end{table}

\begin{figure}[!htbp]
	\hspace{-5em}
	\begin{center}
		\begin{minipage}[c]{1.0\textwidth}
			\includegraphics[width=5in]{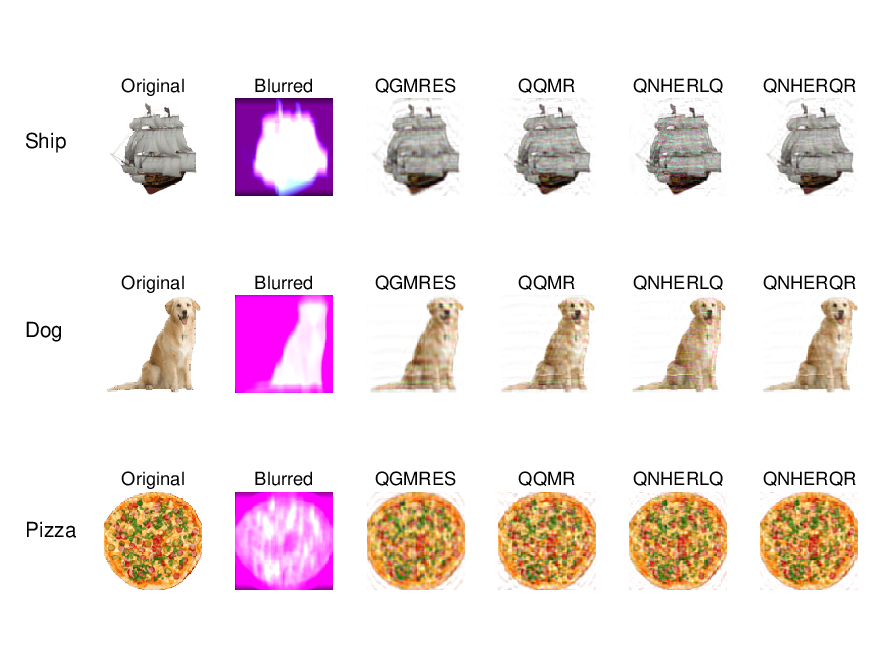}
		\end{minipage}
	\end{center}
\vspace{-1.5em}\caption{Left to Right: the original $100\times 100$ images, the blurred images with
multichannel, the deblurring images were restored by QGMRES, QQMR, QNHERLQ  and QNHERQR.}\label{fig:4}
\end{figure}

\section{Conclusion}\label{sec13}
In this paper, we propose two structure-preserving quaternion conjugate gradient-type methods, i.e., the QNHERLQ method and the QNHERQR method, which are based on   the Galerkin condition and minimal residual condition, respectively, for solving the quaternion linear systems (\ref{1.1}). Theoretical analysis for QNHERLQ and  QNHERQR is discussed. Some numerical results on random data, three-dimensional signal filtering problem  and color image restoration demonstrate the feasibility and effectiveness of the proposed method, which show that the proposed algorithms are not only with theoretically interesting  but also are valuable tools for solving non-Hermitian quaternion linear systems.  In the future, we will study the preconditioned variants to speed up the convergence.

\appendix
\section{Proof of Theorem \ref{thm2}}\label{secA1}
We will prove the assertion by induction on the order $n$. For $n=1$, the result is trivial. Assume that the assertion holds for $1\leq n<l$, i.e., there exist orthogonally JRS-symplectic matrices $P,Q\in\mathbb{R}^{4n\times 4n}$ such that
\begin{equation}\label{eq1}
  P^TMQ=T=
  \left[
  \begin{array}{cccc}
  T_0&T_2&T_1&T_3\\
  -T_2&T_0&T_3&-T_1\\
  -T_1&-T_3&T_0&T_2\\
  -T_3&T_1&-T_2&T_0
  \end{array}
  \right]\in\mathbb{R}^{4n\times 4n}
\end{equation}
is a strictly JRS-tridiagonal matrix,  where $T_0$ is tridiagonal and $T_1,T_2,T_3\in\mathbb{R}^{n\times n}$ are diagonal. For $n=l$, by Proposition \ref{thm1}, there exists an orthogonally JRS-symplectic matrix $W\in\mathbb{R}^{4l\times 4l}$ such that $W^TMW=H$ is an upper JRS-Hessenberg
matrix. Denote
\begin{eqnarray*}
% \nonumber to remove numbering (before each equation)
   && H_0=\left[
   \begin{array}{cccc}
   h_{11}^{(0)}&h_{12}^{(0)}&h_{13}^{(0)}&H_{14}^{(0)}\\
   h_{21}^{(0)}&h_{22}^{(0)}&h_{23}^{(0)}&H_{24}^{(0)}\\
    0&h_{32}^{(0)}&h_{33}^{(0)}&H_{34}^{(0)}\\
     0&0&H_{43}^{(0)}&H_{44}^{(0)}\\
   \end{array}
   \right],
   H_1=\left[
   \begin{array}{cccc}
   h_{11}^{(1)}&h_{12}^{(1)}&h_{13}^{(1)}&H_{14}^{(1)}\\
   0&h_{22}^{(1)}&h_{23}^{(1)}&H_{24}^{(1)}\\
    0& 0&h_{33}^{(1)}&H_{34}^{(1)}\\
     0&0&0&H_{44}^{(1)}\\
   \end{array}
   \right],\\
   &&  H_2=\left[
   \begin{array}{cccc}
   h_{11}^{(2)}&h_{12}^{(2)}&h_{13}^{(2)}&H_{14}^{(2)}\\
   0&h_{22}^{(2)}&h_{23}^{(2)}&H_{24}^{(2)}\\
    0& 0&h_{33}^{(2)}&H_{34}^{(2)}\\
     0&0&0&H_{44}^{(2)}\\
   \end{array}
   \right],
   H_3=\left[
   \begin{array}{cccc}
   h_{11}^{(3)}&h_{12}^{(3)}&h_{13}^{(3)}&H_{14}^{(3)}\\
   0&h_{22}^{(3)}&h_{23}^{(3)}&H_{24}^{(3)}\\
    0& 0&h_{33}^{(3)}&H_{34}^{(3)}\\
     0&0&0&H_{44}^{(3)}\\
   \end{array}
   \right],
\end{eqnarray*}
where $h_{ij}^{(r)}\in\mathbb{R}$ for $i,j=1,2,3$, $r=0,1,2,3$, $H_{i4}^{(r)}\in\mathbb{R}^{1\times (l-3)}$ for $i=1,2,3$, $H_{43}^{(0)}=[h_{43}^{(0)},0,\cdots,0]^T\in\mathbb{R}^{(l-3)\times 1}$, $H_{44}^{(0)}\in\mathbb{R}^{(l-3)\times (l-3)}$ is an upper Hessenberg matrix, and
$H_{44}^{(r)}\in\mathbb{R}^{(l-3)\times (l-3)}$ are upper triangular matrices for $r=1,2,3$.

Denote $\gamma_{12}=\sqrt{(h_{12}^{(0)})^2+(h_{12}^{(1)})^2+(h_{12}^{(2)})^2+(h_{12}^{(3)})^2}$ and take a $4l\times 4l$ generalized  symplectic Givens rotation $G^{(2)}$ defined as in (\ref{givens}) with $\cos\alpha_i=(h_{12}^{(i)})/\gamma_{12}$ for $i=0,1,2,3$ such that
\begin{equation}\label{eq2}
 \widehat{H}=HG^{(2)^T} =\left[
  \begin{array}{cccc}
   \widehat{H}_0& \widehat{H}_2& \widehat{H}_1& \widehat{H}_3\\
  - \widehat{H}_2& \widehat{H}_0& \widehat{H}_3&- \widehat{H}_1\\
  - \widehat{H}_1&- \widehat{H}_3& \widehat{H}_0& \widehat{H}_2\\
  - \widehat{H}_3& \widehat{H}_1&- \widehat{H}_2& \widehat{H}_0
  \end{array}
  \right],
\end{equation}
where
\begin{eqnarray*}
% \nonumber to remove numbering (before each equation)
   && \widehat{H}_0=\left[
   \begin{array}{cccc}
   h_{11}^{(0)}&\gamma_{12}&h_{13}^{(0)}&H_{14}^{(0)}\\
   h_{21}^{(0)}&\widehat{h}_{22}^{(0)}&h_{23}^{(0)}&H_{24}^{(0)}\\
    0&\widehat{h}_{32}^{(0)}&h_{33}^{(0)}&H_{34}^{(0)}\\
     0&0&H_{43}^{(0)}&H_{44}^{(0)}\\
   \end{array}
   \right],
   \widehat{H}_1=\left[
   \begin{array}{cccc}
  h_{11}^{(1)}&0&h_{13}^{(1)}&H_{14}^{(1)}\\
   0&\widehat{h}_{22}^{(1)}&h_{23}^{(1)}&H_{24}^{(1)}\\
    0& \widehat{h}_{32}^{(1)}&h_{33}^{(1)}&H_{34}^{(1)}\\
     0&0&0&H_{44}^{(1)}\\
   \end{array}
   \right],\\
   &&  \widehat{H}_2=\left[
   \begin{array}{cccc}
  h_{11}^{(2)}&0&h_{13}^{(2)}&H_{14}^{(2)}\\
   0&\widehat{h}_{22}^{(2)}&h_{23}^{(2)}&H_{24}^{(2)}\\
    0& \widehat{h}_{32}^{(2)}&h_{33}^{(2)}&H_{34}^{(2)}\\
     0&0&0&H_{44}^{(2)}\\
   \end{array}
   \right],
   \widehat{H}_3=\left[
   \begin{array}{cccc}
  h_{11}^{(3)}&0&h_{13}^{(3)}&H_{14}^{(3)}\\
   0&\widehat{h}_{22}^{(3)}&h_{23}^{(3)}&H_{24}^{(3)}\\
    0& \widehat{h}_{32}^{(3)}&h_{33}^{(3)}&H_{34}^{(3)}\\
     0&0&0&H_{44}^{(3)}\\
   \end{array}
   \right],
\end{eqnarray*}
Likewise, there exist a series of generalized symplectic Givens rotations $G^{(3)}$, $\cdots$, $G^{(l)}\in\mathbb{R}^{4l\times 4l}$ such that
\begin{equation}\label{eq3}
 \overline{\widehat{H}}=(HG^{(2)^T}\!)(G^{(l)}\cdots G^{(3)} )^T
 \!=\!HG^{(2)^T}\!\!G^{(3)^T}\!\!\cdots\! G^{(l)^T}\!=\!\left[
  \begin{array}{cccc}
  \!\! \overline{\widehat{H}}_0& \overline{\widehat{H}}_2& \overline{\widehat{H}}_1& \overline{\widehat{H}}_3\!\!\\
 \!\! - \overline{\widehat{H}}_2& \overline{\widehat{H}}_0& \overline{\widehat{H}}_3&- \overline{\widehat{H}}_1\!\!\\
  \!\!- \overline{\widehat{H}}_1&- \overline{\widehat{H}}_3& \overline{\widehat{H}}_0& \overline{\widehat{H}}_2\!\!\\
  \!\!- \overline{\widehat{H}}_3& \overline{\widehat{H}}_1&- \overline{\widehat{H}}_2& \overline{\widehat{H}}_0\!\!
  \end{array}
  \right]
\end{equation}
with
\begin{eqnarray*}
% \nonumber to remove numbering (before each equation)
   && \overline{\widehat{H}}_0=\left[
   \begin{array}{cccc}
   h_{11}^{(0)}&\gamma_{12}&\gamma_{13}&\Gamma_{14}^{(0)}\\
   h_{21}^{(0)}&\overline{\widehat{h}}_{22}^{(0)}&\overline{\widehat{h}}_{23}^{(0)}&\overline{\widehat{H}}_{24}^{(0)}\\
    0&\overline{\widehat{h}}_{32}^{(0)}&\overline{\widehat{h}}_{33}^{(0)}&\overline{\widehat{H}}_{34}^{(0)}\\
     0&0&\overline{\widehat{H}}_{43}^{(0)}&\overline{\widehat{H}}_{44}^{(0)}\\
   \end{array}
   \right],
   \overline{\widehat{H}}_1=\left[
   \begin{array}{cccc}
  h_{11}^{(1)}&0&0&0\\
   0&\overline{\widehat{h}}_{22}^{(1)}&\overline{\widehat{h}}_{23}^{(1)}&\overline{\widehat{H}}_{24}^{(1)}\\
    0& \overline{\widehat{h}}_{32}^{(1)}&\overline{\widehat{h}}_{33}^{(1)}&\overline{\widehat{H}}_{34}^{(1)}\\
     0&0&\overline{\widehat{H}}_{43}^{(1)}&\overline{\widehat{H}}_{44}^{(1)}\\
   \end{array}
   \right],\\
   &&  \overline{\widehat{H}}_2=\left[
   \begin{array}{cccc}
  h_{11}^{(2)}&0&0&0\\
   0&\overline{\widehat{h}}_{22}^{(2)}&\overline{\widehat{h}}_{23}^{(2)}&\overline{\widehat{H}}_{24}^{(2)}\\
    0& \overline{\widehat{h}}_{32}^{(2)}&\overline{\widehat{h}}_{33}^{(2)}&\overline{\widehat{H}}_{34}^{(2)}\\
     0&0&\overline{\widehat{H}}_{43}^{(2)}&\overline{\widehat{H}}_{44}^{(2)}\\
   \end{array}
   \right],
   \overline{\widehat{H}}_3=\left[
   \begin{array}{cccc}
  h_{11}^{(3)}&0&0&0\\
   0&\overline{\widehat{h}}_{22}^{(3)}&\overline{\widehat{h}}_{23}^{(3)}&\overline{\widehat{H}}_{24}^{(3)}\\
    0& \overline{\widehat{h}}_{32}^{(3)}&\overline{\widehat{h}}_{33}^{(3)}&\overline{\widehat{H}}_{34}^{(3)}\\
     0&0&\overline{\widehat{H}}_{43}^{(3)}&\overline{\widehat{H}}_{44}^{(3)}\\
   \end{array}
   \right],
\end{eqnarray*}
where $\Gamma_{14}^{(0)}=[\gamma_{14},\cdots,\gamma_{1l}]\in\mathbb{R}^{1\times (l-3)}$ with $\gamma_{1t}\!=\!\sqrt{(h_{1t}^{(0)})^2\!+\!(h_{1t}^{(1)})^2\!+\!(h_{1t}^{(2)})^2\!+\!(h_{1t}^{(3)})^2}$ for $t=4,5,\cdots,l$,
$\overline{\widehat{H}}_{43}^{(r)} =[\overline{\widehat{h}}_{43}^{(r)},0,\cdots,0]^T\in\mathbb{R}^{(l-3)\times 1}$ for $r=0,1,2,3$, $\overline{\widehat{H}}_{i4}^{(r)}\in\mathbb{R}^{1\times (l-3)}$ for $i=2,3$, and $\overline{\widehat{H}}_{44}^{(r)}$ are upper Hessenberg matrices for $r=0,1,2,3$.

Then we generate a Householder matrix $H^{(2)}\in\mathbb{R}^{l\times l}$ such that
\begin{equation*}
\overline{\widehat{H}}_0(1,:)H^{(2)^T}=[h_{11}^0,\overline{\gamma}_{12},0,\cdots,0],
\end{equation*}
where $\overline{\gamma}_{12}=\sqrt{\gamma_{12}^2+\gamma_{13}^2+\cdots+\gamma_{1l}^2}$. So there exists an orthogonally JRS-symplectic matrix $[H^{(2)}\oplus H^{(2)}\oplus H^{(2)}\oplus H^{(2)}]^T$ such that
\begin{equation}\label{eq4}
 \breve{H}= \overline{\widehat{H}}[H^{(2)}\oplus H^{(2)}\oplus H^{(2)}\oplus H^{(2)}]^T=\left[
  \begin{array}{cccc}
   \breve{H}_0& \breve{H}_2& \breve{H}_1& \breve{H}_3\\
  - \breve{H}_2& \breve{H}_0& \breve{H}_3&- \breve{H}_1\\
  - \breve{H}_1&- \breve{H}_3& \breve{H}_0& \breve{H}_2\\
  - \breve{H}_3& \breve{H}_1&- \breve{H}_2& \breve{H}_0
  \end{array}
  \right]
\end{equation}
with
\begin{eqnarray*}
% \nonumber to remove numbering (before each equation)
   && \breve{H}_0=\left[
   \begin{array}{cccc}
   h_{11}^{(0)}&\overline{\gamma}_{12}&0&0\\
   h_{21}^{(0)}&\breve{h}_{22}^{(0)}&\breve{h}_{23}^{(0)}&\breve{H}_{24}^{(0)}\\
    0&\breve{h}_{32}^{(0)}&\breve{h}_{33}^{(0)}&\breve{H}_{34}^{(0)}\\
     0&\breve{h}_{42}^{(0)}&\breve{H}_{43}^{(0)}&\breve{H}_{44}^{(0)}\\
   \end{array}
   \right],
   \breve{H}_1=\left[
   \begin{array}{cccc}
  h_{11}^{(1)}&0&0&0\\
   0&\breve{h}_{22}^{(1)}&\breve{h}_{23}^{(1)}&\breve{H}_{24}^{(1)}\\
    0& \breve{h}_{32}^{(1)}&\breve{h}_{33}^{(1)}&\breve{H}_{34}^{(1)}\\
     0&\breve{h}_{42}^{(1)}&\breve{H}_{43}^{(1)}&\breve{H}_{44}^{(1)}\\
   \end{array}
   \right],\\
   &&  \breve{H}_2=\left[
   \begin{array}{cccc}
  h_{11}^{(2)}&0&0&0\\
   0&\breve{h}_{22}^{(2)}&\breve{h}_{23}^{(2)}&\breve{H}_{24}^{(2)}\\
    0& \breve{h}_{32}^{(2)}&\breve{h}_{33}^{(2)}&\breve{H}_{34}^{(2)}\\
     0&\breve{h}_{42}^{(2)}&\breve{H}_{43}^{(2)}&\breve{H}_{44}^{(2)}\\
   \end{array}
   \right],
   \breve{H}_3=\left[
   \begin{array}{cccc}
  h_{11}^{(3)}&0&0&0\\
   0&\breve{h}_{22}^{(3)}&\breve{h}_{23}^{(3)}&\breve{H}_{24}^{(3)}\\
    0& \breve{h}_{32}^{(3)}&\breve{h}_{33}^{(3)}&\breve{H}_{34}^{(3)}\\
     0&\breve{h}_{42}^{(3)}&\breve{H}_{43}^{(3)}&\breve{H}_{44}^{(3)}\\
   \end{array}
   \right].
\end{eqnarray*}

Denote $\beta_{21}=|h_{21}^{(0)}|$ and take a $4l\times 4l$ orthogonally JRS-symplectic matrix
\begin{equation*}
  G\!\!=\!\!\left[
   \begin{smallmatrix}
  \!\!I_{1}&0&0\!\!\\
  \!\!0& h_{21}^{(0)}/\beta_{21}&0\!\!\\
     \!\!0&0&I_{l-2}\!\!\\
   \end{smallmatrix}
   \right]
   \oplus \left[
 \begin{smallmatrix}
  \!\!I_{1}&0&0\!\!\\
  \!\!0& h_{21}^{(0)}/\beta_{21}&0\!\!\\
    \!\! 0&0&I_{l-2}\!\!\\
   \end{smallmatrix}
   \right]
   \oplus\left[
  \begin{smallmatrix}
  \!\!I_{1}&0&0\!\!\\
  \!\!0& h_{21}^{(0)}/\beta_{21}&0\!\!\\
    \!\! 0&0&I_{l-2}\!\!\\
\end{smallmatrix}
   \right]
   \oplus\left[
\begin{smallmatrix}
  \!\!I_{1}&0&0\!\!\\
  \!\!0& h_{21}^{(0)}/\beta_{21}&0\!\!\\
     \!\!0&0&I_{l-2}\!\!\\
\end{smallmatrix}
   \right]
\end{equation*}
such that
\begin{equation}\label{eq5}
  \dot{H}=G\breve{H}=\left[
  \begin{array}{cccc}
   \dot{H}_0& \dot{H}_2& \dot{H}_1& \dot{H}_3\\
  - \dot{H}_2& \dot{H}_0& \dot{H}_3&- \dot{H}_1\\
  - \dot{H}_1&- \dot{H}_3& \dot{H}_0& \dot{H}_2\\
  - \dot{H}_3& \dot{H}_1&- \dot{H}_2& \dot{H}_0
  \end{array}
  \right]
\end{equation}
with
\begin{eqnarray*}
% \nonumber to remove numbering (before each equation)
   && \dot{H}_0=\left[
   \begin{array}{cccc}
   h_{11}^{(0)}&\overline{\gamma}_{12}&0&0\\
  \beta_{21}&\dot{h}_{22}^{(0)}&\dot{h}_{23}^{(0)}&\dot{H}_{24}^{(0)}\\
    0&\breve{h}_{32}^{(0)}&\breve{h}_{33}^{(0)}&\breve{H}_{34}^{(0)}\\
     0&\breve{h}_{42}^{(0)}&\breve{H}_{43}^{(0)}&\breve{H}_{44}^{(0)}\\
   \end{array}
   \right],
   \dot{H}_1=\left[
   \begin{array}{cccc}
  h_{11}^{(1)}&0&0&0\\
   0&\dot{h}_{22}^{(1)}&\dot{h}_{23}^{(1)}&\dot{H}_{24}^{(1)}\\
    0& \breve{h}_{32}^{(1)}&\breve{h}_{33}^{(1)}&\breve{H}_{34}^{(1)}\\
     0&\breve{h}_{42}^{(1)}&\breve{H}_{43}^{(1)}&\breve{H}_{44}^{(1)}\\
   \end{array}
   \right],\\
   &&  \dot{H}_2=\left[
   \begin{array}{cccc}
  h_{11}^{(2)}&0&0&0\\
   0&\dot{h}_{22}^{(2)}&\dot{h}_{23}^{(2)}&\dot{H}_{24}^{(2)}\\
    0& \breve{h}_{32}^{(2)}&\breve{h}_{33}^{(2)}&\breve{H}_{34}^{(2)}\\
     0&\breve{h}_{42}^{(2)}&\breve{H}_{43}^{(2)}&\breve{H}_{44}^{(2)}\\
   \end{array}
   \right],
   \dot{H}_3=\left[
   \begin{array}{cccc}
  h_{11}^{(3)}&0&0&0\\
   0&\dot{h}_{22}^{(3)}&\dot{h}_{23}^{(3)}&\dot{H}_{24}^{(3)}\\
    0& \breve{h}_{32}^{(3)}&\breve{h}_{33}^{(3)}&\breve{H}_{34}^{(3)}\\
     0&\breve{h}_{42}^{(3)}&\breve{H}_{43}^{(3)}&\breve{H}_{44}^{(3)}\\
   \end{array}
   \right].
\end{eqnarray*}
Notice that the submatrix of $\dot{H}$ obtained by deleting the $1$, $l+1$, $2l+1$, $3l+1$ rows and columns is a $4(l-1)\times 4(l-1)$ JRS-symmetric matrix. By the induction assumption, the assertion is true for $n=l$. This prove the theorem.

\end{document}

%% file: ex_shared.tex
\usepackage{lipsum}
\usepackage{amsfonts}
\usepackage{graphicx}
\usepackage{epstopdf}
\usepackage{algorithm}
\usepackage{algorithmicx}
\usepackage{amsmath,amssymb,amsfonts}
\usepackage{mathrsfs}
\usepackage{bm}
\usepackage{appendix}
\ifpdf
  \DeclareGraphicsExtensions{.eps,.pdf,.png,.jpg}
\else
  \DeclareGraphicsExtensions{.eps}
\fi

\newsiamremark{remark}{Remark}
\newsiamremark{hypothesis}{Hypothesis}
\crefname{hypothesis}{Hypothesis}{Hypotheses}
\newsiamthm{claim}{Claim}

\headers{Structure preserving QCG-type methods}{ B.H Huang, T Li, W Li}

\title{Structure preserving quaternion conjugate gradient-type methods for  solving  non-Hermitian quaternion linear systems\thanks{Corresponding author: Wen Li.
\funding{This work was funded by the National Natural Science Foundation of China [grant numbers 12471351, 12401493, 12171168 and 12071159], and the Fujian Province Natural Science Foundation of China [grant number 2022J01194].}}}
\author{Baohua Huang\thanks{School of Mathematics and Statistics, Key Laboratory of Analytical Mathematics and Applications  (Ministry of Education), Fujian Normal University, Fuzhou 350117, P. R. China
  (\email{baohuahuang@fjnu.edu.cn}).}
  \and Tao Li\thanks{School of Mathematics and Statistics, Hainan University, Haikou 570228, P. R. China
  (\email{tli@hainanu.edu.cn}).}
  \and Wen Li\thanks{School of Mathematical Sciences, South China Normal University, Guangzhou 510631, P. R. China
  (\email{liwen@scnu.edu.cn}).}
}
\usepackage{amsopn}

%% file: main.bbl
\begin{thebibliography}{99}
\bibitem{h1844}
W.R. Hamilton, On quaternions, \textit{In: Proceeding of the Royal Irish Academy}, November 11, 1844.

\bibitem{h1866}
W.R. Hamilton, Elements of Quaternions, \textit{Longmans, Green and Co.}, London, 1866.

\bibitem{a1994}
S.L. Adler, Quaternionic quantum mechanics and quantum fields, \textit{Oxford U.P.}, New York, 1994.

\bibitem{p1989}
D. Platnick, Quaternion calculus as a basic tool in computer graphics, \textit{Visual Comput.}, 5 (1989), 2-13.

\bibitem{gcc2018}
Y. Guan, M.T. Chu, D.L. Chu, SVD-based algorithms for the best rank-1 approximation
of a symmetric tensor, \textit{SIAM J. Matrix Anal. Appl.}, 39 (2018), 1095-1115.

\bibitem{pml2020}
T. Parcollet, M. Morchid, G. Linares, A survey of quaternion neural networks, \textit{AI
Rev.}, 53 (2020), 2957-2982.

\bibitem{sgr2017}
L.S. Saoud, R. Ghorbani,  F. Rahmoune, Cognitive quaternion valued neural network
and some applications, \textit{Neurocomputing}, 221 (2017), 85-93.

\bibitem{bm2004}
N. Le Bihan, J. Mars, Singular value decomposition of quaternion matrices: A new tool
for vector-sensor signal processing, \textit{Signal Process.}, 84 (2004), 1177-1199.

\bibitem{sv2011}
\"{O}.N. Subakan, B.C. Vemuri, A quaternion framework for color image smoothing and
segmentation, \textit{Int. J. Comput. Vis.}, 91 (2011), 233-250.

\bibitem{jns2019}
Z.G. Jia, M.K. Ng,  G.J. Song, Robust quaternion matrix completion with applications
to image inpainting, Numer. \textit{Linear Algebra Appl.}, 26 (2019), 1070-5325. %(4)

\bibitem{jnw2019}
Z.G. Jia, M.K. Ng, W. Wang, Color image restoration by saturation-value total variation, \textit{SIAM J. Imaging Sci.}, 12 (2019), 972-1000.

\bibitem{cjpp2021}
Y. Chen, Z.G. Jia, Y. Peng,  Y.X. Peng, A new structure-preserving quaternion QR
decomposition method for color image blind watermarking, \textit{Signal Process.}, 185 (2021),
108088.

\bibitem{sdn2021}
G.J. Song, W. Ding,   M.K. Ng, Low-rank pure quaternion approximation for pure
quaternion matrices, \textit{SIAM J. Matrix Anal. Appl.}, 42 (2021), 58-82.

\bibitem{zt2013}
M.K. Zak, F. Toutounian, Nested splitting conjugate gradient method for matrix equation $AXB=C$ and preconditioning, \textit{Comput. Math. Appl.}, 66 (2013), 269-278.%(3):

\bibitem{zwz2016}
R. Zhou, X. Wang, P. Zhou, A modified HSS iteration method for solving the complex linear matrix equation $AXB = C$, \textit{J. Comput. Math.}, 34 (2016), 437-450.%(4):

\bibitem{ttlx2017}
Z.L. Tian, M.Y. Tian, Z.Y. Liu, T.Y. Xu, The Jacobi and Gauss Seidel-type iteration methods for the matrix equation, \textit{Appl. Math. Comput.}, 292 (2017), 63-75.

\bibitem{hm2016}
N. Huang, C.F. Ma, Modified conjugate gradient method for obtaining the minimum-norm solution
of the generalized coupled Sylvester-conjugate matrix equations, \textit{Appl. Math. Modell.}, 40 (2016), 1260-1275.

\bibitem{s1981}
Y. Saad, Krylov subspace methods for solving large unsymmetric linear systems, \textit{Math. Comput.}, 37 (1981), 105-126.

\bibitem{s1989}
P. Sonneveld, CGS, a fast Lanczos-type solver for nonsymmetric linear systems, \textit{SIAM J.
Sci. Stat. Comput.}, 10 (1989), 36-52.

\bibitem{fn1991}
R.W. Freund, N.M. Nachtigal, QMR: a quasi-minimal residual method for non-Hermitian linear
systems, \textit{Numer. Math.}, 60 (1991) 315-339.

\bibitem{l1952}
C. Lanczos, Solution of systems of linear equations by minimized iterations, \textit{J. Res. Nat. Bur.
Standards}, 49 (1952), 33-53.

\bibitem{ssz2009}
T. Sogabe, M. Sugihara,  S.L. Zhang, An extension of conjugate residual method to
nonsymmetric linear systems, \textit{J. Comput. Appl. Math.}, 226 (2009), 103-113.

\bibitem{ssy1988}
M.A. Saunders, H.D. Simon, E.L. Yip, Two conjugate-gradient-type methods for unsymmetric linear equations, \textit{SIAM J. Numer. Anal.},  25 (1988), 927-940.

\bibitem{jn2021}
Z.G. Jia,  M.K. Ng, Structure preserving quaternion generalized minimal residual method,
\textit{SIAM J. Matrix Anal. Appl.}, 42 (2021), 616-634.

\bibitem{lw2023}
T. Li,  Q.W. Wang, Structure preserving quaternion full orthogonalization
method with applications, \textit{Numer Linear Algebra Appl.}, 30 (2023), e2495.

\bibitem{lw2024}
T. Li, Q.W. Wang, Structure preserving quaternion biconjugate gradient method,
\textit{SIAM J. Matrix Anal. Appl.}, 45 (2024), 306-326.

\bibitem{lwz2024}
T. Li, Q.W. Wang
QQMR: A structure preserving quaternion quasi-minimal residual method. \textit{Math. Comp.}, (2025), 1-30. DOI:https://doi.org/10.1090/mcom/4074.

\bibitem{o2005}
G. Opfer, The conjugate gradient algorithm applied to quaternion valued matrices, \textit{Journal of Applied Mathematics \& Mechanics}, 85 (2005), 660-672. %DOI:10.1002/zamm.200410191. (9):

\bibitem{ps1975}
C.C. Paige, M. A. Saunders, Solutions of sparse indefinite systems of linear equations, \textit{SIAM J. Numer. Anal.}, 12 (1975), 617-629.

\bibitem{gmp2013}
R. Ghiloni, V. Moretti,  A. Perotti, Continuous slice functional calculus in quaternionic Hilbert spaces, \textit{Rev. Math. Phys.}, 25 (2013), 1350006.

\bibitem{jwzc2018}
Z.G. Jia, M.X. Wei, M. Zhao, Y. Chen, A new real structure-preserving quaternion
QR algorithm, \textit{J. Comput. Appl. Math.}, 343 (2018), 26-48.

\bibitem{jo2003}
D. Janovsk$\acute{a}$, G. Opfer, Givens' transformation applied to quaternion valued matrices, \textit{BIT Numer. Math.}, 43 (2003), 991-1002.

\bibitem{S2001}
S.H. Strogatz, Nonlinear Dynamics and Chaos: With Applications to Physics, Biology,
Chemistry and Engineering, \textit{Westview Press, Boulder, Co}, 2001.

\bibitem{vtv} 
S.H. Chan, R. Khoshabeh, K.B. Gibson, et al., An augmented Lagrangian method for total variation video restoration, \textit{IEEE Trans. Image Process.}, 20 (2011), 3097-3111.

\bibitem{WBS2004} 
Z. Wang, A. C. Bovik, H. R. Sheikh, Image quality assessment: from error visibility to structural similarity, \textit{IEEE Trans. Image Process.}, 13 (2004), 600-612.

\bibitem{psf} 
P.C. Hansen, J.G. Nagy, D.P. O'leary, Deblurring images: matrices, spectra, and filtering, \textit{SIAM}, 2006.

\bibitem{Smith2} 
A.R. Smith, Alpha and the history of digital compositing, \textit{Technical Memo 7, Microsoft Corporation}, 1995.

%{\sc Z.~G. Jia, M.~X. Wei, M. Zhao, and Y. Chen}, {\em A new real structure-preserving
%quaternion QR algorithm}, J. Comput. Appl. Math., 343 (2018),  pp.~26--48.

\end{thebibliography}
